\documentclass[11pt,reqno]{amsproc}

\usepackage[sc]{mathpazo}\renewcommand{\mathbf}{\mathbold}
\normalfont
\usepackage[T1]{fontenc}

\usepackage[margin=1in]{geometry}
\usepackage{comment}
\usepackage[small,bf]{caption}
\usepackage{scalerel,nicefrac}
\usepackage[all,cmtip]{xy}
\usepackage{tikz-cd}

\usepackage{rotating, setspace}

\usepackage{bbm}
\usepackage{listings}
\usepackage{amsthm,amssymb}
\usepackage{thmtools}
\usepackage{color}
\usepackage{subfigure}
\usepackage{shuffle}
\usepackage{mathrsfs}
\usepackage{multicol}
\usepackage[utf8]{inputenc}
\usepackage{longtable}
\usepackage{booktabs}
\usepackage{ragged2e}  %
\usepackage{relsize}
\usepackage{enumitem}

\makeatletter
\renewcommand{\@secnumfont}{\bfseries}
\makeatother

\usepackage{mathtools}
\usepackage{hyperref}
\usepackage{cleveref}
\usepackage{autonum}

\hypersetup{
    colorlinks=true,
    linkcolor=blue,
    filecolor=blue,      
    }

\makeatletter
\def\l@subsection{\@tocline{2}{0pt}{2.5pc}{5pc}{}}
\makeatother
\numberwithin{equation}{section}

\makeatletter
\newcommand{\mylabel}[2]{#2\def\@currentlabel{#2}\label{#1}}
\makeatother

\newcommand{\ps}[1]{\mkern-.25mu\mathbin{\left(\mkern-3.5mu\left({#1}\right)\mkern-3.5mu\right)}}

\newcommand{\Z}{\mathbb{Z}}
\newcommand{\N}{\mathbb{N}}

\newcommand{\R}{\mathbb{R}}
\newcommand{\C}{\mathbb{C}}

\newcommand{\K}{\mathbb{K}}
\newcommand{\RP}{\mathbb{RP}}

\newcommand{\bx}{\mathbf{x}}
\newcommand{\bX}{\mathbf{X}}
\newcommand{\bY}{\mathbf{Y}}
\newcommand{\bZ}{\mathbf{Z}}
\newcommand{\by}{\mathbf{y}}
\newcommand{\bz}{\mathbf{z}}

\newcommand{\bq}{\mathbf{q}}
\newcommand{\bT}{\mathbf{T}}
\newcommand{\bv}{\mathbf{v}}
\newcommand{\bw}{\mathbf{w}}
\newcommand{\bp}{\mathbf{p}}
\newcommand{\ba}{\mathbf{a}}
\newcommand{\bd}{\mathbf{d}}
\newcommand{\be}{\mathbf{e}}
\newcommand{\bc}{\mathbf{c}}
\newcommand{\bb}{\mathbf{b}}
\newcommand{\bS}{\mathbf{S}}
\newcommand{\bW}{\mathbf{W}}
\newcommand{\bu}{\mathbf{u}}
\newcommand{\bF}{\mathbf{F}}

\newcommand{\bpi}{\boldsymbol{\pi}}
\newcommand{\bC}{\mathbf{C}}
\newcommand{\balpha}{{\boldsymbol{\alpha}}}

\newcommand{\cH}{\mathcal{H}}
\newcommand{\cE}{\mathcal{E}}
\newcommand{\cX}{\mathcal{X}}
\newcommand{\cS}{\mathcal{S}}

\newcommand{\fg}{\mathfrak{g}}
\newcommand{\fh}{\mathfrak{h}}

\newcommand{\fk}{\mathfrak{k}}
\newcommand{\fa}{\mathfrak{a}}

\newcommand{\fX}{\mathfrak{X}}
\newcommand{\gl}{\mathfrak{gl}}
\newcommand{\fsl}{\mathfrak{sl}}
\newcommand{\fcone}{\mathfrak{c}}
\newcommand{\fsusp}{\mathfrak{s}}

\newcommand{\hbX}{\widehat{\bX}}

\newcommand{\hC}{\widehat{C}}
\newcommand{\hlambda}{\widehat{\lambda}}
\newcommand{\hL}{\widehat{L}}

\newcommand{\tbb}{\tilde{\bb}}
\newcommand{\tbX}{\tilde{\bX}}
\newcommand{\tX}{\tilde{X}}

\newcommand{\hW}{\widehat{W}}
\newcommand{\hepsilon}{\widehat{\epsilon}}

\newcommand{\tf}{\tilde{f}}

\newcommand{\tFr}{\widetilde{\Fr}}

\newcommand{\hw}{\hat{w}}
\newcommand{\hb}{\hat{b}}
\newcommand{\hrho}{\hat{\rho}}
\newcommand{\tF}{\tilde{F}}
\newcommand{\txi}{\tilde{\xi}}

\newcommand{\cmG}{\mathbf{G}}
\newcommand{\cmK}{\mathbf{K}}
\newcommand{\cmH}{\mathbf{H}}
\newcommand{\cmg}{{\boldsymbol{\mathfrak{g}}}}
\newcommand{\cmh}{{\boldsymbol{\mathfrak{h}}}}
\newcommand{\cmk}{{\boldsymbol{\mathfrak{k}}}}

\newcommand{\gt}{\vartriangleright} %

\newcommand{\gtd}{\gtrdot}

\newcommand{\cmb}{\delta} %
\newcommand{\concat}{\star} %

\newcommand{\Set}{\mathsf{Set}}
\newcommand{\Grp}{\mathsf{Grp}}
\newcommand{\Grpd}{\mathsf{Grpd}}
\newcommand{\Lie}{\mathsf{Lie}}
\newcommand{\Vect}{\mathsf{Vect}}
\newcommand{\Aff}{\mathsf{Aff}}
\newcommand{\Rep}{\mathsf{Rep}}

\newcommand{\XGrp}{\mathsf{XGrp}}
\newcommand{\XLGrp}{\mathsf{XLGrp}}
\newcommand{\XLie}{\mathsf{XLie}}
\newcommand{\XMod}{\mathsf{XMod}}
\newcommand{\DGrp}{\mathsf{DGrp}}

\newcommand{\For}{\mathsf{For}}
\newcommand{\FL}{\mathsf{FL}}
\newcommand{\FMon}{\mathsf{FMon}}
\newcommand{\Fr}{\mathsf{Fr}}
\newcommand{\Mod}{\mathsf{Mod}}
\newcommand{\sE}{\mathsf{E}}
\newcommand{\sQ}{\mathsf{Q}}

\newcommand{\slice}[2]{#1_{/#2}}
\newcommand{\VL}{\mathsf{VL}}
\newcommand{\SG}{\mathsf{SG}}
\newcommand{\FG}{\mathsf{FG}}
\newcommand{\sF}{\mathsf{F}}

\newcommand{\bsF}{\boldsymbol{\mathsf{F}}}

\newcommand{\thinhom}{\operatorname{th}}

\newcommand{\trans}{\operatorname{trans}}
\newcommand{\rank}{\operatorname{rank}}
\newcommand{\Lip}{\operatorname{Lip}}

\newcommand{\im}{\operatorname{im}}
\newcommand{\Aut}{\operatorname{Aut}}
\newcommand{\Ad}{\operatorname{Ad}}
\newcommand{\ad}{\operatorname{ad}}
\newcommand{\ab}{{\operatorname{ab}}}
\newcommand{\SPAN}{\operatorname{span}}
\newcommand{\PL}{\operatorname{PL}}
\newcommand{\cmPL}{{\boldsymbol{\operatorname{PL}}}}
\newcommand{\cl}{\operatorname{cl}}
\newcommand{\id}{\operatorname{id}}
\newcommand{\GL}{{\operatorname{GL}}}
\newcommand{\coker}{{\operatorname{coker}}}
\newcommand{\End}{{\operatorname{End}}}
\newcommand{\Prim}{{\operatorname{Prim}}}
\newcommand{\Der}{{\operatorname{Der}}}
\newcommand{\Pf}{{\operatorname{Pf}}}
\newcommand{\LCS}{{\operatorname{LCS}}}
\newcommand{\Kite}{{\operatorname{Kite}}}

\newcommand{\Cone}{{\operatorname{Cone}}}
\newcommand{\Susp}{{\operatorname{Susp}}}
\newcommand{\Pair}{{\operatorname{Pair}}}
\newcommand{\Loop}{{\operatorname{Loop}}}

\newcommand{\SL}{{\operatorname{SL}}}
\newcommand{\V}{V}

\newcommand{\supp}{{\operatorname{supp}}}

\newcommand{\planarloop}{{\operatorname{PlanarLoop}}}

\newcommand{\thingroupoid}{\mathfrak{T}}
\newcommand{\thingroup}{\tau}
\newcommand{\cmthingroup}{{\boldsymbol{\tau}}}
\newcommand{\realization}{R}

\newcommand{\cmrealization}{\mathbf{R}}
\newcommand{\cmfundgroup}{{\boldsymbol{\pi}}}

\newcommand{\hol}{F_0} %
\newcommand{\Hol}{F_1} %
\newcommand{\sig}{S_0} %
\newcommand{\sigPL}{S_{\PL,0}}
\newcommand{\Sig}{S_1} %
\newcommand{\SigPL}{S_{\PL,1}}

\newcommand{\tPL}{\widetilde{\PL}}

\newcommand{\con}{\gamma_0}
\newcommand{\Con}{\gamma_1}
\newcommand{\curv}{\kappa}
\newcommand{\Curv}{\mathcal{K}}

\newcommand{\conk}{\zeta_0}
\newcommand{\Conk}{\zeta_1}

\newcommand{\gtrans}{\theta}
\newcommand{\Gtrans}{\Theta}

\newcommand{\hurewicz}{\mathcal{H}}

\newcommand{\poly}[1]{\overline{#1}} %
\newcommand{\com}[1]{\hat{#1}} %

\newcommand{\andd}{\quad \text{and} \quad}

\newtheorem{counter}{Counter}[section]
\newtheorem{lemma}[counter]{Lemma}

\newtheorem{proposition}[counter]{Proposition}
\newtheorem{theorem}[counter]{Theorem}
\newtheorem{corollary}[counter]{Corollary}

\theoremstyle{definition}
\newtheorem{definition}[counter]{Definition}
\newtheorem{example}[counter]{Example}
\newtheorem{remark}[counter]{Remark}

\title{Thin Homotopy and the Signature of Piecewise Linear Surfaces}
\date{\today}

\author{Francis Bischoff}
\email{francis.bischoff@uregina.ca}
\address{Department of Mathematics and Statistics, University of Regina, Regina, SK S4S 0A2, Canada}

\author{Darrick Lee}
\email{darrick.lee@ed.ac.uk}
\address{School of Mathematics and Maxwell Institute, University of Edinburgh, Edinburgh EH9 3FD, Scotland}

\begin{document}

\begin{abstract}
    We introduce a crossed module of piecewise linear surfaces and study the signature homomorphism, defined as the surface holonomy of a universal translation invariant $2$-connection. This provides a transform whereby surfaces are represented by formal series of tensors. Our main result is that the signature uniquely characterizes a surface up to translation and thin homotopy, also known as tree-like equivalence in the case of paths. This generalizes a result of Chen and positively answers a question of Kapranov in the setting of piecewise linear surfaces. As part of this work, we provide several equivalent definitions of thin homotopy, generalizing the plethora of definitions which exist in the case of paths. Furthermore, we develop methods for explicitly and efficiently computing the surface signature.  
\end{abstract}

\maketitle
\vspace{40pt}
\begin{figure}[!h]
    \includegraphics[width=\linewidth]{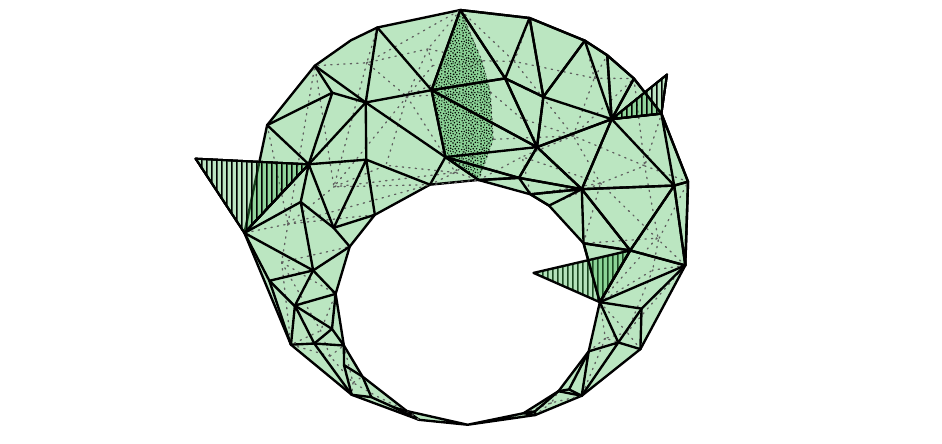}
\end{figure}

\clearpage
$\,$\vspace{5pt}
{\small \tableofcontents}

\clearpage

\section{Introduction}

The signature of a path $\bx = (\bx^1, \ldots, \bx^n) \in C^1([0,1], \R^n)$ is a non-commutative power series 
\begin{align}
\sig(\bx) = \sum_{I} \sig^{I}(\bx) \,e^{I} \in  \prod_{k=0}^\infty (\R^n)^{\otimes k}
\end{align}
whose coefficients are obtained by taking iterated integrals of the path: 
\begin{align}
\sig^I(\bx) = \int_{0 \leq t_{1} \leq ... \leq t_{k} \leq 1} d\bx^{i_1}_{t_1} \cdots d\bx^{i_k}_{t_k}
\end{align}
where $I = (i_{1}, ..., i_{k}) \in [n]^k$ is a multi-index and $e^I = e^{i_1} \otimes \ldots \otimes e^{i_k}$, where $\{e^1, \ldots, e^n\}$ is a basis of $\R^n$. The path signature was introduced by Chen \cite{chen_iterated_1954}, who proved that this invariant uniquely characterizes a path up to translation and \emph{thin homotopy}~\cite{chen_integration_1958}. Thin homotopy is an equivalence relation on paths that essentially consists in two basic equivalences: reparametrizations, and cancellation of retracings.
This allows paths to be treated analogously to words in a free group, a fact which plays a crucial role in Chen's proof of the injectivity of the signature. \medskip

Thin homotopy is a significant strengthening of the conventional homotopy relation in algebraic topology, from which the fundamental group of a manifold $M$ arises. This leads to the \emph{thin} fundamental group, $\pi_1^{\thinhom}(M)$, which is defined as the group of loops in a manifold modulo the thin homotopy equivalence relation. This notion first arose in differential geometry and physics through the study of connections and the invariances inherent in their parallel transport. The thin fundamental group plays a key role in the generalization of the Riemann-Hilbert correspondence to the setting of non-flat connections due to ~\cite{barrett_holonomy_1991, caetano_axiomatic_1994}. Given a Lie group $G$, this states that there is an equivalence between the category of all $G$-connections on $M$ and the category of smooth $G$-representations of $\pi_1^{\thinhom}(M)$. In fact, the path signature can be seen as the holonomy of a universal translation-invariant connection on $\mathbb{R}^n$ valued in the free Lie algebra generated by $\mathbb{R}^n$. \medskip

Since their introduction, Chen's iterated integrals have become highly influential in many areas of geometry and topology (eg. \cite{MR482748, hain_iterated_1984,block_higher_2014,arias_abad_Ainfty_2013,kohno_formal_2022,komendarczyk_diagram_2020, MR1975178, MR992201, MR2543329}). More recently, the path signature was foundational in developing the theory of rough paths by Lyons in \cite{lyons_differential_1998,lyons_differential_2007}, which has  played a prominent role in the areas of stochastic analysis~\cite{friz_course_2020} and machine learning~\cite{mcleod_signature_2025, lee_signature_2023}. In \cite{hambly_uniqueness_2010}, Hambly and Lyons extended the injectivity of the path signature to bounded variation paths, while in \cite{boedihardjo_signature_2016}, Boedihardjo et al.~generalized this further to highly irregular \emph{rough paths}. 
In this context, the thin homotopy equivalence relation is known as \emph{tree-like} equivalence. \medskip 

The motivation for the present paper is to generalize Chen's injectivity result in another direction to the setting of surfaces. 
Surface holonomy is the generalization of parallel transport to surfaces, originally developed to study higher gauge theory~\cite{baez_higher_2004,martins_two-dimensional_2010,schreiber_smooth_2011}.
The \emph{surface signature} was introduced by Kapranov in \cite{kapranov_membranes_2015} as the surface holonomy of a universal translation-invariant $2$-connection on $\mathbb{R}^n$. This can be formulated as a homomorphism of crossed modules $\bS: \cmthingroup(\mathbb{R}^n) \to \com{\cmK}(\mathbb{R}^n)$, between a crossed module $\cmthingroup(\mathbb{R}^n)$ of thin equivalence classes of surfaces in Euclidean space, and a crossed module $\com{\cmK}(\mathbb{R}^n)$ of `formal surfaces', which is defined by a formal integration of a free crossed module of Lie algebras $\com{\cmk}(\R^n)$. %
This construction has recently received interest in the rough paths literature \cite{lee_surface_2024,chevyrev_multiplicative_2024-1} where it forms the basis of a theory of rough surfaces. One of the key advantages of Kapranov's notion of the surface signature over alternate approaches to generalizing the signature~\cite{giusti_topological_2025,diehl_two-parameter_2022,diehl_signature_2024} is the fact that it preserves the concatenation structure of surfaces. This is crucial to proving the extension results in \cite{lee_surface_2024,chevyrev_multiplicative_2024-1}, and leads to parallelizable computations for potential applications in machine learning. \medskip

In \cite[Question 2.5.6]{kapranov_membranes_2015}, Kapranov poses the question of whether the signature characterizes a surface uniquely up to translation and thin homotopy. Thin homotopy for surfaces, first introduced by~\cite{caetano_family_1998}, is the equivalence relation whereby two surfaces are thinly equivalent if there is a homotopy between them which does not sweep out any volume. As in the case of paths, this includes generalized reparametrizations and cancellation of folds. However, there are also more general `non-local' thin homotopies. For example, in~\Cref{prop:rp2_example} we construct a closed surface which factors through the real projective plane
\begin{align}
f: S^2 \to \mathbb{RP}^2 \to \mathbb{R}^n
\end{align}
which is thinly null homotopic, even though it does not exhibit any folds to cancel. This example illustrates the crucial fact that surfaces do not admit unique reductions via folds. In contrast, every path can be uniquely reduced, up to reparametrization, to a path which does not contain any retracings. As the proofs of injectivity of the path signature rely on this property, these approaches cannot be immediately generalized to the case of surfaces. \medskip

In this paper, we focus on the special case of piecewise linear surfaces in order to avoid the analytical subtleties and focus on the underlying algebraic structures. In~\Cref{sec:PL} we define a crossed module of piecewise linear surfaces $\cmPL(V)$, which is functorial in the vector space $V$. Because the formal integration $\com{\cmK}(V)$ of the free crossed module $\com{\cmk}(V)$ is also functorial in $V$, this allows us to define the surface signature $\bS_{\PL}: \cmPL \Rightarrow \com{\cmK}$ as a natural transformation. This level of abstraction is immediately paid off by the following remarkable uniqueness result. 

\begin{theorem} \label{thm:intro_unique_signature}
    The piecewise linear surface signature $\bS_{\PL}: \cmPL \Rightarrow \com{\cmK}$ is the unique natural transformation extending the piecewise linear path signature. Furthermore, the smooth surface signature $\bS  : \cmthingroup \Rightarrow \com{\cmK}$ is the unique \emph{continuous} natural transformation extending the smooth path signature. 
\end{theorem}

Our main result is the injectivity of the surface signature for piecewise linear surfaces. 

\begin{theorem}
    Let $\bX$ be a piecewise linear surface such that its signature is trivial, $\Sig(\bX) = 0.$ Then $\bX$ is thinly homotopic to the constant surface. 
\end{theorem}

Our proof of this result makes use of two main ideas:
\begin{itemize}
    \item[a)] First, because of the injectivity of the path signature, any smooth surface $\bX$ with vanishing signature can be assumed to be closed. In~\cite{kapranov_membranes_2015}, Kapranov suggests that the signature of a closed surface is given by its associated current. In~\Cref{sec:abelianization}, we use a gauge transformation to abelianize the universal $2$-connection and give a proof of this fact in~\Cref{thm:ssig_equiv_to_abelianization}. As a result, we show in~\Cref{cor:trivial_ssig_current} that any closed surface $\bX$ with vanishing signature has the property that
    \begin{align} \label{eq:intro_zero_integral}
        \int_{\bX} \omega = 0 \quad \text{for all compactly support $2$-forms} \quad \omega \in \Omega^2_c(V).
    \end{align}
    \item[b)] In~\Cref{sec:PL_thin_homotopy}, we complete the proof of injectivity. A closed surface $\bX : S^2 \to V$ which satisfies~\eqref{eq:intro_zero_integral} has vanishing homology in its image $C = \im(\bX)$,
    \begin{align}
        H_2(\bX)([S^2]) = 0 \in H_2(C).
    \end{align}
    To conclude that $\bX$ is thinly null homotopic, it suffices to show that $[\bX] = 0 \in \pi_{2}(C)$, and by Hurewicz, this would be true if $\pi_{1}(C) = 0$. In other words, the obstruction is the fundamental group of $C$. Hence, by attaching sufficiently many discs to $C$, we can kill the fundamental group and thereby produce the desired thin null homotopy. In order to implement this, we make use of Whitehead's theorem~\cite{whitehead_combinatorial_1949} that the fundamental crossed module of a $2$-dimensional CW complex is free. 
\end{itemize}
We remark that the above proof works for surfaces that are much more general than piecewise linear ones, but fails to work for general smooth surfaces because of the complicated nature of the image of a general smooth map. Our focus on piecewise linear surfaces is partially justified by the fact that they can be used to approximate general smooth surfaces. We hope to make use of this fact to prove a general injectivity result in future work. 

\medskip

For the remainder of this introduction, we highlight further results obtained in our paper.

\medskip

\textbf{Algebraic Models of Piecewise Linear Paths and Surfaces.}
Elements of the free group on $n$ letters can be viewed as thin homotopy equivalence classes of lattice paths on $\Z^n$: word reduction coincides with path reduction via retracing. In~\Cref{ssec:pl_paths}, we construct the group $\PL_0(V)$ of piecewise linear paths on $V$ as the quotient of the free group on $V$, by certain relations to account for retracings. We show that this satisfies several properties analogous to free groups, such as the existence of minimal words and a universal property. %

\begin{figure}[!h]
    \includegraphics[width=\linewidth]{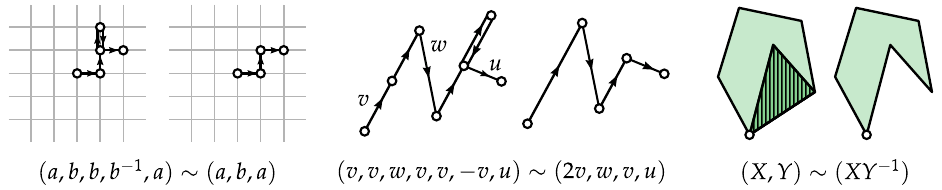}
\end{figure}

In~\Cref{ssec:pl_surfaces}, we extend this construction to define a crossed module of piecewise linear surfaces $\cmPL(V)$. The distinction between local and non-local cancellations is reflected in our construction.
We first construct a pre-crossed module $\widetilde{\cmPL}(V)$ as the free group of planar regions, quotiented by relations which account for local fold cancellations, similar to the case of paths. Next, we quotient out the Peiffer subgroup $\mathrm{Peiff}(V)$ to obtain a crossed module, which accounts for the remaining non-local cancellations. This crossed module also satisfies a universal property, which can then be used to prove~\Cref{thm:intro_unique_signature}. We remark that the injectivity of the signature further implies that the thin homotopy relation for piecewise linear surfaces is a formal consequence of the local fold cancellations and the Peiffer identity. Therefore, the normal subgroup $\mathrm{Peiff}(V)$ contains within it the surfaces which are thinly null homotopic in a non-trivial way. 
\medskip

\textbf{Computational Methods for Surface Signature.}
In recent years, the path signature has been used as a rich feature set for sequential and time series data~\cite{mcleod_signature_2025}. Theoretical properties guarantee its effectiveness in approximating functions and characterizing measures on path space~\cite{chevyrev_signature_2022}, thereby justifying its use in machine learning tasks. 
Furthermore, because the signature preserves concatenation of paths, it lends itself to parallelizable algorithms~\cite{kidger_signatory_2021,toth_seq2tens_2020}. While the study of the surface signature is still fairly new, recent work has studied theoretical properties of surface holonomy valued in matrix groups to provide features for 2-dimensional data~\cite{lee_random_2023}. However, there are two significant challenges in applying the surface signature itself. These are the lack of established computational methods and the absence of a canonical coordinate representation for surface signatures due to the opaque nature of the Peiffer identity\footnote{This should be compared with the path signature, which is valued in the tensor algebra: by choosing a basis on the underlying vector space, this induces a basis on the tensor powers.}.
 \medskip

In~\Cref{ssec:computational}, we develop a natural decomposition of the surface signature into boundary and abelian components in both the piecewise linear and smooth settings. In particular, we find that the linear structure of a vector space induces canonical splittings of the crossed modules $\cmPL(V)$, $\cmthingroup(V)$, and $\com{\cmK}(V)$, and show that the surface signature preserves the resulting decompositions. The main result of this section is~\Cref{thm:main_computation_result}, which gives an explicit formula for the signature of a surface. Given a surface $\bX$, the boundary component simply records the path signature of the boundary, while the abelian component records the integrals of all monomial 2-forms over the closed surface $\mathcal{C}(\bX)$ obtained by coning off the boundary of $\bX$. Schematically, this abelian component is given by the following formula
\[
    \Sig^{\Gamma}(\bX) = \sum_{k \geq 0} \frac{1}{k!} \int_{\mathcal{C}(\bX)} (\mathrm{id}_V)^{k},
\]
where $(\mathrm{id}_V)^{k}$ is viewed as a function on $V$ valued in the symmetric power $S^k(V)$ (see Remark \ref{rem:canonicalexp}). This elucidates the information contained in the surface signature, suggests methods for computing the signature, and provides explicit coordinates to represent the signature. \medskip

\textbf{Characterizing Thin Homotopy.}
Thin homotopy equivalence for paths has been studied widely, and a variety of equivalent definitions have been proposed. Chen's definition~\cite{chen_integration_1958} involved reducing the path along retracings, which was generalized by Tlas~\cite{tlas_holonomic_2016} to $C^1$ paths. Geometric definitions in terms of the existence of homotopies which satisfy additional conditions were given in~\cite{caetano_axiomatic_1994,barrett_holonomy_1991}, and a characterization based on holonomy was considered in~\cite{tlas_holonomic_2016,meneses_thin_2021}. On the analytic side, the equivalent notion of tree-like equivalence was defined in terms of factorization through a tree or the existence of a height function~\cite{hambly_uniqueness_2010}. Finally, due to the injectivity of the path signature, we obtain yet another definition. We summarize these results in~\Cref{thm:path_thin_homotopy}. \medskip

This collection of equivalent definitions provides a plethora of distinct ways to understand thin homotopy of paths. While some of these definitions can be easily generalized to the case of surfaces, others require modification due to fundamental differences in the two dimensional setting: we must take into account non-local cancellations. In~\Cref{thm:surface_thin_homotopy}, we propose a generalization of each definition, and show that they are equivalent in the piecewise linear setting.\medskip

Finally, in~\Cref{sec:group_homology}, we highlight a connection between thin homotopies and group homology. This allows us to further classify the non-local cancellations into those due to conjugation by elements in the fundamental group $G = \pi_1(C)$ of the image $C = \im(\bX)$ of a surface $\bX$, and those which have further complexity. We show that thinly null homotopic surfaces in this latter case are classified by $H_3(G)$, leading to a new geometric interpretation of group homology. \medskip

\textbf{Acknowledgments.}
We are very grateful to Camilo Arias Abad, who has met with us to discuss signatures and surface holonomy for countless hours over the past few years, has contributed several important ideas, and has significantly influenced the direction of this project. We would also like to thank Harald Oberhauser for several insightful discussions at the beginning of this project. The first author wishes to thank Tim Porter for introducing him to crossed modules and explaining various fundamental aspects of the theory. Furthermore, numerous stimulating conversations with Martin Frankland, Marco Gualtieri, and Jim Stasheff have influenced this work. 
F.B. is supported by an NSERC Discovery grant. 
D.L. was supported by the Hong Kong Innovation and Technology Commission (InnoHK Project CIMDA) during part of this work. 

\section{Thin Homotopy and the Path Signature}

In this section, we provide some background on parallel transport and the path signature. We consider thin homotopy of paths, and relate several known definitions of thin homotopy. We then focus on the piecewise linear setting: while the signature of piecewise linear paths is well understood, we provide a novel functorial construction of the piecewise linear signature, which arises naturally using universal properties. This section serves as a guide to the results we will generalize to surfaces in the remainder of this article.
Throughout this article, $\V$ denotes a finite-dimensional real vector space $\V \cong \R^d$, and we let $\Vect \label{pg:vect}$ denote the category of finite dimensional vector spaces and linear maps. Furthermore, we always consider paths with \emph{sitting instants}. 

\begin{definition} \label{def:1d_sitting_instants}
    A path $\bx \in C([0,1], V)$ has \emph{sitting instants} if there exists some $\epsilon > 0$ such that
\begin{align}
    \bx_{s} = \bx_0 \andd \bx_{1-s} = \bx_1 \quad \text{for all $s \in [0, \epsilon]$}.
\end{align}
    Unless otherwise specified, we assume all paths have sitting instants.
\end{definition}
This ensures that when we consider \emph{composable} paths $\bx, \by \in C^1([0,1],V)$, such that $\bx_1 = \by_0$, with some smoothness condition, the \emph{concatenation} \label{pg:path_concat}
\begin{align}
    (\bx \concat \by)_s \coloneqq \left\{ \begin{array}{cl}
        \bx_{2s} & : s \in [0,\frac12] \\
        \by_{2s-1} & : s \in [\frac12, 1]
    \end{array}\right.
\end{align}
preserves the smoothness condition $\bx \concat \by \in C^1([0,1], V)$.

\subsection{Parallel Transport and Path Signature} \label{ssec:path_signature}
In this section, we define the path signature as the parallel transport of the universal translation invariant connection. Since we are working over a vector space $V$, all principal bundles can be assumed trivial and as a result, we use the following simplified definition of a connection.  

\begin{definition} \label{def:con}
    Let $\mathfrak{g}$ be a Lie algebra. A \emph{$\mathfrak{g}$-connection} on $V$ is a Lie algebra valued $1$-form $\con \in \Omega^1(V, \fg)$. A connection is \emph{translation-invariant} if it has the form 
    \begin{align}
        \con = \sum_{i=1}^d \con^i \, dz_i,
    \end{align}
    where $\con^i \in \fg$ and $z_i$ are linear coordinates on $V$. 
    The \emph{curvature of $\con$}, denoted $\curv^{\con} \in \Omega^2(V, \fg)$, is defined by the following formula
    \begin{align} \label{eq:1_curvature}
        \curv^{\con} \coloneqq d\con + \frac{1}{2}[\con, \con].
    \end{align}
\end{definition}
\begin{definition} \label{def:ph}
    Let $G$ be a Lie group with Lie algebra $\fg$. Let $\con$ be a $\fg$-connection and $\bx \in C^1([0,1], V)$ a path. Consider the following differential equation for $\hol^{\con}(\bx): [0,1] \to G$,
    \begin{align}
        \frac{d\hol^{\con}(\bx)_t}{dt} = dL_{\hol^{\con}(\bx)_t} \con\left(\frac{d\bx_t}{dt}\right), \quad \hol^{\con}(\bx)_0 = e.
    \end{align}
    The \emph{parallel transport of $\con$ along $\bx$} is defined to be
    \begin{align}
        \hol^{\con}(\bx) \coloneqq \hol^{\con}(\bx)_1.
    \end{align}
\end{definition}

The path signature is the parallel transport of the \emph{universal translation-invariant connection}. This is a connection $\conk$ valued in the free Lie algebra generated by $V$, $\fk_0 \coloneqq \FL(\V)$, and it can be understood as the identity endomorphism of $V$
\begin{align}
    \conk = \mathrm{id}_V \in V^* \otimes V \subset \Omega^1(\V, \fk_0),
\end{align}
where we view $V^*$ as the subspace of translation invariant $1$ forms, and $V$ as the subspace of generators of the free Lie algebra. If $\{ e_i \}_{i = 1}^{d}$ form a basis of $V$, with dual basis $\{ z_i \}_{i = 1}^{d} \in V^*$, viewed as coordinates on $V$, then the connection has the following explicit expression 
\begin{align} \label{eq:univ_connection}
    \conk = \sum_{i=1}^d  dz_i e_i \in \Omega^1(\V, \fk_0).
\end{align}
Because the free Lie algebra is infinite dimensional, it is convenient to consider its truncations. These can be formally expressed via the lower central series of $\fk_0$, defined recursively as
\begin{align}
    \LCS_1(\fk_0) \coloneqq \fk_0 \andd \LCS_r(\fk_0) \coloneqq [\fk_0, \LCS_{r-1}(\fk_0)].
\end{align}
Then, we define the \emph{$n$-truncated free Lie algebra} $\fk_0^{(\leq n)}$ as
\begin{align}
    \fk_0^{(\leq n)} \coloneqq \fk_0/ \LCS_{n+1}(\fk_0).
\end{align}
As we assume that $V$ is finite dimensional, $\fk_0^{(\leq n)}$ is a finite-dimensional Lie algebra. Then, we can explicitly integrate $\fk_0^{(\leq n)}$ to the Lie group of exponential elements $K_0^{(\leq n)}$ of the truncated tensor algebra $T^{(\leq n)}_0(V) = \prod_{k=0}^n V^{\otimes k}$. This is defined as follows
\begin{align}
    K_0^{(\leq n)}(V) \coloneqq \left\{ \exp(x) \in T^{(\leq n)}(V) \, : \, x \in \fk_0^{(\leq n)}\right\}.
\end{align}
We denote the projective limit of these Lie groups, and their corresponding Lie algebras, by 
\begin{align} \label{eq:k0_completion}
    \com{K}_0(V) \coloneqq \lim_{\longleftarrow} K_0^{(\leq n)}(V) \subset T\ps{V} = \prod_{k=0}^\infty V^{\otimes k} \andd \com{\fk}_0(V) = \lim_{\longleftarrow} \fk_0^{(\leq n)}(V).
\end{align}
Given a linear map $f: V \to W$, there is an induced group homomorphism $K_0^{(\leq n)}(V) \to K_0^{(\leq n)}(W)$ for all $n \in\N$. This induces a group homomorphism $\com{K}_0(V) \to \com{K}_0(W)$. Hence, we get a well-defined functor
\begin{align}
    \com{K}_0: \Vect \to \Grp.
\end{align}

We now provide the standard definition of the path signature, and return to this in~\Cref{ssec:pl_paths} on piecewise linear paths. 

\begin{definition} \label{def:psig}
    Let $\bx \in C^1([0,1], V)$. For $n \in \N$, consider the parallel transport of the universal connection $\conk$ of Equation~\eqref{eq:univ_connection} projected onto $\fk_0^{(\leq n)}$. In particular, consider the following differential equation for $\sig^{(n)}(\bx): [0,1] \to K_0^{(\leq n)}$,
    \begin{align}
        \frac{d\sig^{(n)}(\bx)_t}{dt} = \sig^{(n)}(\bx)_t \otimes \frac{d\bx_t}{dt}, \quad \sig^{(n)}(\bx)_0 = 1.
    \end{align}
    The \emph{$n$-truncated path signature of $\bx$} is defined to be $\sig^{(n)}(\bx) \coloneqq \sig^{(n)}(\bx)_1$. Then, we define the \emph{path signature of $\bx$} to be the projective limit 
    \begin{align}
        \sig(\bx) \coloneqq \lim_{\longleftarrow} \sig^{(n)}(\bx) \in \com{K}_0.
    \end{align}
\end{definition}

\begin{remark}
    The completion of the tensor algebra is often considered analytically by defining appropriate norms on $T^{(\leq n)}(V)$, or by considering a family of seminorms, as in~\cite{chevyrev_characteristic_2016}. In this article, as we do not need these analytic properties, we will consider formal completions. 
\end{remark}

\subsection{The Thin Path Group} \label{ssec:thin_path_group}
Parallel transport, and in particular the path signature, is invariant under \emph{thin homotopy equivalence}. In this section, we construct a group of paths up to translation and thin homotopy equivalence, and show that the signature defines a homomorphism out of this group. The meaning of thin homotopy will be explained further in the next section. 
\begin{definition}
    Two smooth paths $\bx, \by \in C^{1}([0,1], V)$ are \emph{thin homotopy equivalent}, denoted $\bx \sim_{\thinhom} \by$, if there exists an endpoint preserving smooth homotopy $h \in C^{1}([0,1]^2, V)$ between $\bx$ and $\by$ such that 
    \begin{itemize}
        \item \textbf{(homotopy condition)} $h_{0,t} = \bx_t$ and $h_{1,t} = \by_t$; 
        \item \textbf{(thinness condition)} $\rank(dh) \leq 1$, where $dh$ is the differential of $h$. 
    \end{itemize}
\end{definition}
Thin homotopy defines an equivalence relation on paths which is compatible with concatenation. The \emph{thin fundamental groupoid} of $V$ is defined to be the set of equivalence classes
\begin{align}
        \thingroupoid_1(V) \coloneqq C^1([0,1],V)/\sim_{\thinhom}.
\end{align}
It is a groupoid over $V$ with product given by the concatenation of paths. Given a Lie group $G$ with Lie algebra $\fg$, the parallel transport of a connection $\con \in \Omega^1(V,\fg)$ defines a groupoid homomorphism 
\begin{align}
    \hol^{\con} : \thingroupoid_1(V) \to G.
\end{align}
There is a natural action of the additive group $V$ on $\thingroupoid_1(V)$ by groupoid automorphisms given by translating paths. We define the \emph{thin path group} to be the quotient 
\begin{align}
        \thingroup_1(\V) \coloneqq \thingroupoid_1(V)/\sim_{\trans}
\end{align}
and we note that it is a group. Our convention will be to use paths starting at the origin, $\bx_0 = 0$, as representatives of the translation equivalence classes. There is a well-defined group homomorphism $t: \thingroup_1(\V) \to \V$ given by sending a path $\bx$ to its displacement $\bx_1 - \bx_0 \in V$. The thin fundamental groupoid can then be recovered as the corresponding action groupoid 
\begin{align}
    \thingroupoid_1(V) \cong \V \rtimes \thingroup_1(\V). 
\end{align}
Given a linear map $\phi: V \to W$ and a path $\bx \in C^1([0,1], \V)$, we can define a new path $\phi(\bx) \in C^1([0,1], W)$ by $\phi(\bx)_t = \phi(\bx_t)$. This preserves thin homotopy classes of maps, and is equivariant with respect to translation. Thus, we obtain a \emph{thin path group functor}
\begin{align}
    \thingroup_1 : \Vect \to \Grp.
\end{align}
Because the universal connection $\conk$ is translation invariant, its parallel transport is invariant under both thin homotopy and translations. Therefore, the path signature defines a homomorphism $\sig : \thingroup_1(V) \to \com{K}_0(V)$. Furthermore, by~\cite[Proposition 7.52]{friz_multidimensional_2010}, this fits into a natural transformation
\begin{align} \label{eq:psig_nt}
    \sig: \thingroup_1 \Rightarrow \com{K}_0. 
\end{align}

\subsection{Thin Homotopy of Paths} \label{ssec:thin_homotopy_paths} There are several different equivalent definitions for thin homotopy equivalence of paths. In the signatures and analysis literature, it is known as \emph{tree-like equivalence} and has been extended to the setting of \emph{rough paths}. In this section, we will review the equivalent definitions in the setting of $C^1$ paths. In particular, we will see that the path signature characterizes this equivalence relation. Roughly speaking, thin homotopy equivalence captures two main types of behavior:
\begin{enumerate}
    \item \textbf{Reparametrizations:} Given a path $\bx \in C^1([0,1], V)$ and a reparametrization $\phi: [0,1] \to [0,1]$, then $\bx \sim_{\thinhom} \bx \circ \phi$. 
    \item \textbf{Retracings:} Given paths $\bx, \by, \bz \in C^1([0,1], V)$, we say that a \emph{retracing} is a path segment of the form $\bz \concat \bz^{-1}$. Paths which differ by retracings are thin homotopy equivalent: \begin{align}\bx \concat \bz \concat \bz^{-1} \concat \by \sim_{\thinhom} \bx \concat \by.\end{align} The path $\bx \concat \by$ is called a \emph{reduction} of $\bx \concat \bz \concat \bz^{-1} \concat \by$.
\end{enumerate}

\begin{figure}[!h]
    \includegraphics[width=\linewidth]{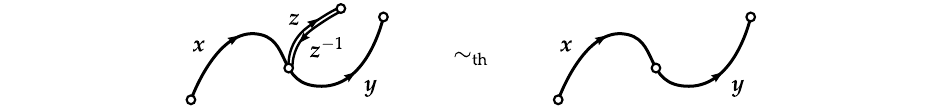}
\end{figure}

We will focus mainly on thinly null-homotopic paths: paths that are thin homotopy equivalent to the constant path. This is because two paths $\bx$ and $\by$ can then be defined to be thin homotopy equivalent if $\bx \concat \by^{-1}$ is thinly null-homotopic. Chen~\cite{chen_integration_1958} originally defined thinly null-homotopic paths in the piecewise regular setting via path reductions as shown above. Tlas~\cite{tlas_holonomic_2016} generalized this definition to $C^1$ paths, where there may be infinitely many retracings, using the notion of transfinite words.

\begin{definition}{\cite[Definition 3.1]{cannon_combinatorial_2000}} \label{def:word}
    Let $E$ be a set consisting of an alphabet, and let $E^{-1}$ denote a formal inverse set. A \emph{transfinite word} over $E$ is a function $W: B \to E \cup E^{-1}$, where $B$ is a totally ordered set, such that $W^{-1}(e)$ is finite for each $e \in E \cup E^{-1}$. A word is \emph{reducible to the trivial word} if and only if every finite truncation (mapping all but finitely many letters in the alphabet to the identity) reduces to the trivial word.
\end{definition}

\begin{theorem}{\cite[Theorem 1]{tlas_holonomic_2016}} \label{thm:tlas_result}
    Let $\bx: [0,1] \to V$ be a $C^1$ path. There exists a collection of mutually disjoint subsets $\{A_n\}_{n=0}^\infty$ such that
    \begin{enumerate}
        \item $A_0$ is closed, $A_n$ is open for $n > 0$, and $\bigcup_{n=0}^\infty A_n = [0,1]$.
        \item If $t \in A_n$, then the set inverse $\bx^{-1}(\bx_t) \subset A_n$ and for $n > 0$, then $|\bx^{-1}(\bx_t)| = n$.%
        \item $\bx'_t \neq 0$ if $t \in A_n$ for $n > 0$, while $\bx'_t = 0$ for $t \in A_0^\circ$ (the interior).
        \item Each $A_n$ for $n>0$ is a union of disjoint open intervals. The path $\bx$ restricted to any such interval is an embedding. The image of any two such embeddings are either disjoint or identical.%
    \end{enumerate}
\end{theorem}

Here, the set $A_0$ is used to ``collect'' all regions of the interval in which the path is constant. 
Let $B$ denote the set of intervals from $A_n$ for all $n > 0$, equipped with a linear order inherited from $\R$. Let $E = \{\bx(a) \, : \, a \in B\}$ be the set of embedded arcs (see (4) above) where we remove repetitions (up to reparametrization and switch in orientation). 
Let $E^{-1}$ be a copy of $E$ denoting formal inverses. Then, we define the \emph{transfinite word associated to $\bx$} to be the mapping $W_\bx : B \to E \cup E^{-1}$ sending each interval to its corresponding arc, taking into account the orientation. Any arc $e \in E \cup E^{-1}$ has finite length, and thus $W_{\bx}^{-1}(e)$ must be finite because $\bx$ is $C^1$.  %

\begin{definition}
    The path $\bx \in C^1([0,1], V)$ is \emph{word reduced} if $W_\bx$ is reducible to the trivial word. 
\end{definition}

This definition is instructive as it explicitly specifies a (possibly infinite) partition of the path, then matches up each constituent arc with an adjacent arc with the same image and opposite orientation. We will use a variation of this idea in our main injectivity proof for the surface signature later in the article. We now turn to other equivalent definitions of thin homotopy. %

\begin{definition} \label{def:real_tree}
    A metric space $T$ is an \emph{$\R$-tree} if for any pair of points $x,y \in T$, all topological embeddings $\sigma : [0,1] \to T$ where $\sigma_0 = x$ and $\sigma_1 = y$ have the same image. 
\end{definition}

The following theorem collects known results relating various definitions of thin homotopy.

\begin{theorem} \label{thm:path_thin_homotopy}
    Let $\bx: [0,1] \to V$ be a $C^1$ path with vanishing derivative at the endpoints. 
    It is \emph{thinly null-homotopic} or \emph{tree-like} if any of the following equivalent definitions hold:
    \begin{enumerate}
        \item[\mylabel{W1}{(\textbf{W1})}] \textbf{Word Condition.} The path $\bx$ is word reduced.
        \item[\mylabel{H1}{(\textbf{H1}$_G$)}] \textbf{Holonomy Condition.} For a semi-simple Lie group $G$ with Lie algebra $\mathfrak{g}$, the parallel transport of every smooth $\mathfrak{g}$-connection $\omega \in \Omega^1(V, \fg)$ along $\bx$ is trivial\footnote{Each semi-simple $G$ is treated as an independent condition.}.
        \item[\mylabel{R1}{(\textbf{R1})}] \textbf{Rank Condition.} There exists an endpoint-preserving $C^1$ homotopy $h: [0,1]^2 \to V$ from $\bx$ to the constant path at $0$ such that the rank of $dh$ is $\leq 1$ everywhere.
        \item[\mylabel{I1}{(\textbf{I1})}] \textbf{Image Condition.}~\cite{barrett_holonomy_1991} There exists an endpoint-preserving homotopy $h: [0,1]^2 \to V$ from $\bx$ to the constant path at $0$ such that $\im(h) \subset \im(\bx)$. 
        \item[\mylabel{F1}{(\textbf{F1})}] \textbf{Factorization Condition.}~\cite{boedihardjo_signature_2016} There exists a factorization of $\bx$ through an $\R$-tree $T$, namely $\bx: [0,1] \to T \to V$. 
        \item[\mylabel{A1}{(\textbf{A1})}] \textbf{Analytic Condition.}~\cite{hambly_uniqueness_2010} There exists a Lipschitz function $h:[0,1] \to \R$ such that $h(t) \ge 0$ for all $t \in [0,1]$, $h(0) = h(1)$ and if $h(s) = h(t) = \inf_{s \leq u \leq t} h(u)$, then $\bx(s) = \bx(t)$. The function $h$ is called the \emph{height function}.
        \item[\mylabel{S1}{(\textbf{S1})}] \textbf{Path Signature Condition.} The path signature $S_0$ of $\bx$ is trivial, $S_0(\bx) = 1$.
    \end{enumerate}
\end{theorem}
\begin{proof}
    In~\cite[Theorem 3]{tlas_holonomic_2016} it is proved that \ref{W1}, \ref{H1} for all semi-simple $G$, and \ref{R1} are equivalent. Furthermore, in the proof of~\cite[Theorem 2]{tlas_holonomic_2016} which shows \ref{W1} $\implies$ \ref{R1}, a homotopy is constructed which satisfies the \ref{I1} (see the top of~\cite[page 18]{tlas_holonomic_2016}), so \ref{W1}, \ref{H1}, or \ref{R1} also imply \ref{I1}. Furthermore, if there exists a homotopy which satisfies \ref{I1}, this homotopy must also satisfy the rank condition \ref{R1}. Thus the first four conditions are equivalent. Furthermore, we note that the definition of a \emph{tree} in~\cite[Definition 4]{tlas_holonomic_2016} is a special case of our definition. Then, \ref{W1}, \ref{H1}, \ref{R1} or \ref{I1} all imply \ref{F1}. \medskip

    Now,~\cite{hambly_uniqueness_2010} shows the equivalence of \ref{F1}, \ref{A1} and \ref{S1} in the case of bounded variation paths, which includes $C^1$ paths. Note that these three conditions do not use a homotopy (whose regularity we would need to consider). Thus, this implies these are equivalent in the $C^1$ setting.\medskip

    In order to connect these two classes of results, we show that \ref{S1} implies \ref{H1} for the specific case of $G = \SL_2(\R)$. Suppose $\bx \in C^1([0,1],V)$ such that $\sig(\bx) = 1$. Suppose to the contrary that there exists a smooth connection 
    \begin{align}
        \omega = \begin{pmatrix} \omega_1 & \omega_2 \\
            \omega_3 & -\omega_1 \end{pmatrix} \in \Omega^1(V, \fsl_2(\R)),
    \end{align}
    where $\omega_i \in \Omega^1(V)$, such that the holonomy (with initial condition at the identity) is not trivial $\hol^{\omega}(\bx) \neq I$. The holonomy can be expressed in terms of the iterated integrals of $\omega$, where specific entries are iterated integrals of $\omega_1, \omega_2, \omega_3$. This implies that there exists such an iterated integral which is nontrivial, and by following the proof of~\cite[Lemma 4.1]{chen_integration_1958}, this implies that the signature is nontrivial, which is a contradiction. 
\end{proof}

While these equivalent definitions are stated for $C^1$ paths, the tree condition~\ref{F1} can be extended to highly irregular \emph{rough paths}, where this theorem has been generalized to show that the kernel of the path signature consists of tree-like rough paths~\cite{boedihardjo_signature_2016}. We do not consider such rough paths in this article. Instead, we focus our attention on the piecewise linear setting where analogous questions regarding thin homotopy for 2-dimensional surfaces are still largely unexplored.

\subsection{Group of Piecewise Linear Paths} \label{ssec:pl_paths}
In this section, we give an algebraic construction of the group of piecewise linear paths and the path signature. Given a vector space $V$, consider the free monoid $(\FMon(V), \concat)$ on the underlying set of $V$. We define the \emph{group of piecewise-linear paths} as
\begin{align}
    \PL_0(V) \coloneqq \FMon(V)/ \sim
\end{align}
subject to the following relations:
\begin{enumerate}[label=(\textbf{PL0.\arabic*})]
    \item \label{PL0.1} $(v,w) \sim (v + w)$ if $v$ and $w$ are linearly dependent;
    \item \label{PL0.2} $(0) \sim \emptyset$, where $0 \in V$ and $\emptyset$ is the identity in $\FMon(V)$ (the empty word). 
\end{enumerate}

With these relations $\PL_0(V)$ is a group. Indeed, the inverse of $(v_1, \ldots, v_k) \in \PL_0(V)$ is
\begin{align}
    (v_1, \ldots, v_k)^{-1} \coloneqq (-v_k, \ldots, -v_1).
\end{align}
An important property of $\PL_0(V)$ is the existence of unique minimal representatives, analogous to the reduced word in a free group. The following result is proved in~\Cref{apxsec:minimal_pl_paths}.
\begin{proposition} \label{prop:minimal_pl_paths}
    An element $\bx \in \PL_0(V)$ has a unique \emph{minimal representative}
    \begin{align}
        \bx = (v_1, \ldots, v_n)_{\min}.
    \end{align}
    This is a word with the property that all $v_i \neq 0$, and every consecutive pair $(v_i, v_{i+1})$ is linearly independent in $V$. We use the subscript $(\cdot)_{\min}$ to denote the minimal representative.
\end{proposition}

Consider the map of sets $\eta_V : V \to \PL_0(V)$ given by the inclusion $\eta_V(v) = (v)$. This has the property that if we restrict to a 1-dimensional subspace $U \subset V$, then $\eta_V$ is a group homomorphism. 
This gives rise to a universal property for $\PL_0(V)$.

\begin{lemma} \label{lem:PL0_univ_property}
   Let $V$ be a vector space and let $G$ be a group. Let $f: V \to G$ be a map which restricts to a group homomorphism on subspaces $U \subset V$ of dimension 1, and satisfies $f(0) = e$, where $e \in G$ is the identity. Then, there exists a unique group homomorphism $F: \PL_0(V) \to G$ such that $F \circ \eta_V = f$. 
\end{lemma}
\begin{proof}
    Uniqueness of $F$ is immediate because $\PL_0(V)$ is generated by $V$. For existence, by the universal property of free monoids, there exists a unique monoid morphism $\tf: \FMon(V) \to G$. First, note that $\tf(0) = f(0) = e = \tf(\emptyset)$, where the second equality holds by assumption. 
    Second, if $v, w \in V$ are contained in a 1-dimensional subspace, we have 
    \begin{align}
        \tf(v, w) = \tf(v) \cdot \tf(w) = f(v) \cdot f(w) = f(v+w) = \tf(v+w).
    \end{align}
    Therefore, this descends to a homomorphism $F: \PL_0(V) \to G$ such that $F \circ \eta_V = f$.
    
\end{proof}

\begin{corollary} \label{cor:PL_functor}
The group of piecewise linear paths defines a functor
\begin{align}
    \PL_0: \Vect \to \Grp.
\end{align}
\end{corollary}
\begin{proof}
    Suppose $\phi: V \to W$ is a linear map. Then $\eta_W \circ \phi : V \to \PL_0(W)$ restricts to a group homomorphism on each one-dimensional subspace. By the universal property in~\Cref{lem:PL0_univ_property}, there is a unique group homomorphism $\PL_0(\phi): \PL_0(V) \to \PL_0(W)$ such that $\PL_0(\phi) \circ \eta_{V} = \eta_W \circ \phi$. This uniqueness implies functoriality. 
\end{proof}

For later use, we will also consider the notion of the span of a path.
\begin{corollary} \label{cor:span_pres}
    Let $\cS(V)$ be the set of linear subspaces of $V$. There is a well-defined map
    \begin{align} \label{eq:span_def}
        \SPAN: \PL_0(V) \to \cS(V) \quad \text{given by} \quad \bx \mapsto \SPAN(v_1, \ldots, v_k),
    \end{align}
    where $(v_1, \ldots, v_k)_{\min}$ is the minimal representative of $\bx \in \PL_0(V)$. If $\phi: V \to W$ is a linear map, then
    \begin{align}
        \SPAN(\PL_0(\phi)(\bx)) \subset \phi(\SPAN(\bx)).
    \end{align}
\end{corollary}

The universal property of~\Cref{lem:PL0_univ_property} provides an effective method for constructing homomorphisms out of $\PL_{0}(V)$. For example, the identity map $\mathrm{id}_{V}: V \to V$ automatically extends to a homomorphism $t: \PL_0(V) \to V$, and the corresponding action groupoid $\V \rtimes \PL_0(V)$ is the piecewise linear analogue of $\thingroupoid_1(V)$. Next, consider the map $r_0: V \to \thingroup_1(V)$ defined by $r_0(v) \coloneqq [\bx^v]$, where $[\bx^v]$ is the equivalence class of the path 
\[
\bx^v(t) \coloneqq  \psi(t) \cdot v,
\]
where $\psi \in C^\infty([0,1],[0,1])$ is a surjective map with sitting instants. Since $r_0$ restricts to a homomorphism on $1$-dimensional subspaces of $V$, by~\Cref{lem:PL0_univ_property}, it induces a group homomorphism,
\begin{align} \label{eq:realization0}
    \realization_0: \PL_0(V) \to \thingroup_1(V),
\end{align}
which we call the \emph{realization map}. In fact, as we vary $V$, these maps fit into a natural transformation
\begin{align} \label{eq:realization0_nt}
    \realization_0: \PL_0 \Rightarrow \thingroup_1.
\end{align}

\begin{lemma} \label{lem:realization0_injective}
    For all $V \in \Vect$, the realization map $\realization_0: \PL_0(V) \to \thingroup_1(V)$ is injective.
\end{lemma}
\begin{proof}
    Let $\bv = (v_1, \ldots, v_k)_{\min} \in \PL_0(V)$ be a minimal representative, and suppose $\realization_0(\bv)$ is trivial. The path $\realization_0(v_1,\ldots, v_k)$ can only be thinly null homotopic if there are retracings. Since the representative is minimal, this cannot occur if $k > 0$. Hence $\bv = \emptyset_0$.
\end{proof}

Finally, consider the map $s_0 \coloneqq \sig \circ r_0 : V \to \com{K}_0(V)$, which can be expressed explicitly as
\begin{align}
    s_0(v) = \sum_{k=0}^\infty \frac{v^{\otimes k}}{k!}.
\end{align}
Again, $s_0$ restricts to a group homomorphism on 1-dimensional subspaces of $V$, and so by~\Cref{lem:PL0_univ_property}, we obtain a homomorphism
\begin{align} 
    S_{\PL,0}: \PL_0(V) \to \com{K}_0(V),
\end{align}
which we call the \emph{piecewise linear path signature}. It factors through the path signature by definition, and this fact holds at the level of natural transformations. 
\begin{proposition} \label{prop:psigpl_nt}
    The maps $\realization_0$, $S_0$, and $S_{\PL,0}$ are natural transformations which factor as
    \begin{align}
        S_{\PL,0}: \PL_0 \xRightarrow{\realization_0} \thingroup_1 \xRightarrow{S_0} \com{K}_0.
    \end{align}
\end{proposition}

The construction of $\PL_0(V)$ can be viewed as a linear algebraic analogue of the construction of a free group. In fact, the following result shows that free groups can be embedded into $\PL_0(V)$.

\begin{proposition} \label{prop:embedding_free_group}
    Let $P: A \to V$ be a map of sets with the property that $P(a)$ and $P(a')$ are linearly independent if $a \neq a'$ (in particular, $P(a) \neq 0$ for all $a \in A$). Then the induced homomorphism $\tilde{P}: \FG(A) \to \PL_0(V)$ is injective.  
\end{proposition}
\begin{proof}
    Let $\tilde{P}: \FG(A) \to \PL_0(V)$ be the unique homomorphism that restricts to $\eta_{V} \circ P$ on $A$. Let $\bw \in \FG(A)$ be an element, which can be expressed uniquely as a reduced word in $A \cup A^{-1}$: 
    \begin{align}
    \bw = (a_1^{n_1}, a_{2}^{n_2}, ..., a_{r}^{n_{r}}), 
    \end{align}
    where all $n_{i}$ are non-zero integers, and $a_{i} \neq a_{i+1}$ for $1 \leq i \leq r-1$. Then 
    \begin{align}
    \tilde{P}(\bw) = (n_1 P(a_1), n_2 P(a_2), ..., n_r P(a_r)). 
    \end{align}
    By assumption, each vector in the list is non-zero, and consecutive pairs are linearly independent. Hence, $\tilde{P}(\bw)$ is a minimal representative by~\Cref{prop:minimal_pl_paths}. This implies that the kernel of $\tilde{P}$ consists of the empty word, implying that the homomorphism is injective. 
\end{proof}

We end this section by observing that the piecewise linear signature is unique, once we take into account the starting point of a path. Let $\Aff$ be the category of affine spaces, whose objects are finite dimensional vector spaces, but whose morphisms are \emph{affine-linear maps}. Given an affine space $V$, define action groupoids 
\begin{align}
    \Pi_{1, \PL}(V) = V \rtimes \PL_{0}(V), \qquad \widehat{\Pi}_1(V) = V \rtimes \com{K}_0(V),
\end{align}
where the actions are respectively defined by the homomorphism $t: \PL_0(V) \to V$ and the truncation $t: \com{K}_0(V) \to K_0^{(\leq 1)}(V) \cong V$, via the additive action of $\V$ on itself. Along with the thin fundamental groupoid, these define functors from the category of affine spaces to the category of groupoids
\begin{align}
    \Pi_{1, \PL}, \ \ \thingroupoid_1, \ \ \widehat{\Pi}_1 : \Aff \to \Grpd. 
\end{align}
Furthermore, the maps $\realization_0$, $S_0$, and $S_{\PL,0}$ give rise to natural transformations 
 \begin{align}
        \tilde{S}_{\PL,0}: \Pi_{1, \PL} \xRightarrow{\tilde{\realization}_0} \thingroupoid_1 \xRightarrow{\tilde{S}_0} \widehat{\Pi}_1
\end{align}
which restrict to the identity on the objects of the groupoids. 
\begin{proposition} \label{prop:groupoid_sig_uniqueness}
    There is a unique natural transformation 
    \begin{align}
        \tilde{S}_{\PL,0} : \Pi_{1, \PL} \xRightarrow{} \widehat{\Pi}_1,
    \end{align}
    called the \emph{piecewise linear groupoid signature}, which restricts to the identity on objects. 
\end{proposition}
\begin{proof}
    It suffices to show that any natural transformation $F: \Pi_{1, \PL} \xRightarrow{} \widehat{\Pi}_1$ which is the identity on objects must be given by $\tilde{S}_{\PL,0}$. First, let $U$ be a $1$-dimensional vector space. Then $t: \com{K}_0(U) \to U$ is an isomorphism and hence $\widehat{\Pi}_1(U) \cong \mathrm{Pair}(U)$, the terminal object in the category of groupoids over $U$. Therefore, we have equality of components $F_{U} = (\tilde{S}_{\PL,0})_{U}$. Now let $V$ be a general vector space. Elements of $\Pi_{1, \PL}(V)$ can be factored into products of elements of the form $(v, \eta_{V}(u))$, for $v, u \in V$. Hence, it suffices to show that $F_V(v,\eta_{V}(u)) = (\tilde{S}_{\PL,0})_{V}(v, \eta_{V}(u))$ for all such pairs. Given $(v,u)$, define the affine linear map 
    \begin{align}
        f_{(v,u)} : \mathbb{R} \to V, \qquad t \mapsto v + tu.
    \end{align}
    Then $\Pi_{1, \PL}(f_{(v,u)})(0, \eta_{\mathbb{R}}(1)) = (v, \eta_{V}(u))$, where $(0, \eta_{\mathbb{R}}(1)) \in \Pi_{1, \PL}(\mathbb{R})$. Therefore, using the naturality of $F$ and $\tilde{S}_{\PL,0}$, we obtain 
    \begin{align}
        F_V(v,\eta_{V}(u)) &= F_V \circ \Pi_{1, \PL}(f_{(v,u)})(0, \eta_{\mathbb{R}}(1)) = \widehat{\Pi}_1(f_{(v,u)}) \circ F_{\mathbb{R}}(0, \eta_{\mathbb{R}}(1)) \\
        &= \widehat{\Pi}_1(f_{(v,u)}) \circ (\tilde{S}_{\PL,0})_{\mathbb{R}}(0, \eta_{\mathbb{R}}(1)) \\
        &= (\tilde{S}_{\PL,0})_{V}(v,\eta_{V}(u)),
    \end{align}
    where the equality in the second line uses $F_{\mathbb{R}} = (\tilde{S}_{\PL,0})_{\mathbb{R}}$.
\end{proof}
\begin{remark}
    Piecewise linear paths are dense in $C^1_0([0,1],V)$, the $C^1$ paths starting at the origin, equipped with the Lipschitz topology. Therefore, the groupoid signature $\tilde{S}_{0} : \thingroupoid_1 \xRightarrow{} \widehat{\Pi}_1$ is the unique \emph{continuous} natural transformation which is the identity on objects. 
\end{remark}

\section{Surface Holonomy and the Surface Signature}

In this section, we introduce surface holonomy in the smooth setting, and discuss the surface signature introduced in~\cite{kapranov_membranes_2015}, and further developed in~\cite{lee_surface_2024,chevyrev_multiplicative_2024-1}. 
We begin with some conventions and notation for surfaces. 
\begin{definition} \label{def:2d_sitting_instants}
    A surface $\bX \in C([0,1]^2, V)$ has \emph{sitting instants} if there exists some $\epsilon > 0$ such that
\begin{align}
    \bX_{u,t} = \bX_{0,t}, \quad \bX_{1-u, t} = \bX_{1,t}, \quad \bX_{s,u} = \bX_{s,0}, \quad \bX_{s,1-u} = \bX_{s,1} \quad \text{for all} \quad u \in [0,\epsilon],\, s,t \in [0,1].
\end{align}
    Unless otherwise specified, we assume all surfaces have sitting instants.
\end{definition} 
For a surface $\bX \in C([0,1]^2, V)$, the \emph{bottom, right, top, and left boundary paths of $\bX$} are
\begin{align} \label{eq:individual_boundary_paths}
    (\partial_b\bX)_s \coloneqq \bX_{s,0}, \quad (\partial_r\bX)_s \coloneqq \bX_{1,s}, \quad (\partial_t\bX)_s \coloneqq \bX_{s, 1}, \quad (\partial_l \bX)_s \coloneqq \bX_{0,s}.
\end{align}
Furthermore, we define\footnote{Note that since the surface has sitting instants, the resulting boundary paths and boundary loop also have sitting instants.} the \emph{boundary loop of $\bX$} using the counter-clockwise convention to be
\begin{align} \label{eq:surface_boundary}
    \partial\bX = (\partial_b\bX \concat \partial_r \bX) \concat ((\partial_t \bX)^{-1} \concat (\partial_l \bX)^{-1}).
\end{align}

\begin{figure}[!h]
    \includegraphics[width=\linewidth]{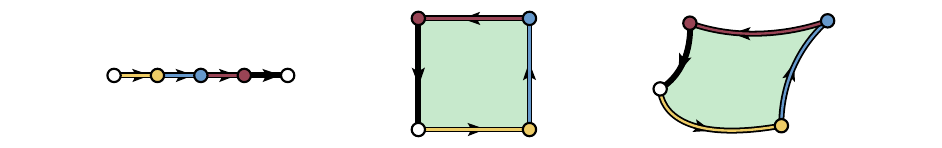}
\end{figure}

The sitting instants ensure that for horizontally and vertically composable smooth surfaces $\bX, \bY, \bZ \in C^1([0,1]^2, V)$, such that $\partial_r\bX = \partial_l \bY$ and $\partial_t \bX = \partial_b \bZ$, their horizontal and vertical compositions 
\begin{align} \label{eq:strict_surface_concatenation}
    (\bX \concat_h \bY)_{s,t} \coloneqq \left\{ \begin{array}{cl}
        \bX_{2s,t} & : s \in [0, \frac12] \\
        \bY_{2s-1, t} & : s \in [\frac12, 1] 
    \end{array} \right.
    \andd 
    (\bX \concat_v \bZ)_{s,t} \coloneqq \left\{ \begin{array}{cl}
        \bX_{s,2t} & : t\in [0,\frac12] \\
        \bZ_{s,2t-1} & : t \in [\frac12,1]
    \end{array}\right.
\end{align}
are also smooth, so that $(\bX \concat_h \bY), (\bX \concat_v \bZ) \in C^1([0,1]^2, V)$. 

\subsection{Thin Homotopy of Surfaces}
We begin by considering the 2-dimensional notion of thin homotopy defined by generalizing the rank condition~\ref{R1}. 

\begin{definition} \label{def:2d_thin_homotopy}
    Two smooth surfaces $\bX, \bY \in C^1([0,1]^2, V)$ are \emph{thin homotopy equivalent}, denoted $\bX \sim_{\thinhom} \bY$, if there exists a corner-preserving smooth homotopy $H \in C^1([0,1]^3, V)$ between $\bX$ and $\bY$ such that
    \begin{itemize}
        \item \textbf{(homotopy condition)} $H_{0,s,t} = \bX_{s,t}$ and $H_{1,s,t} = \bY_{s,t}$;
        \item \textbf{(thin homotopic boundaries)} the four sides of the homotopy $H_{u,s,0}, H_{u,s,1}, H_{u,0,t}, H_{u,1,t}$ are thin homotopies between the four boundary paths of $\bX$ and $\bY$,
        \item \textbf{(thinness condition)} $\rank(dH) \leq 2$, where $dH$ is the differential of $H$.
    \end{itemize}
\end{definition}

\begin{remark}
    In this article, we use $[0,1]^2$ to parametrize surfaces, which allows for simple definitions of horizontal and vertical concatenation, as in Equation~\eqref{eq:strict_surface_concatenation}. 
    In~\Cref{def:2d_thin_homotopy}, we use \emph{corner-preserving} maps. This will allow us, as in~\eqref{eq:horizontal_composition}, to retain the simple formulation of concatenation for thin equivalence classes. We emphasize that we have chosen to use corner preserving maps because it simplifies the definition of the algebraic operations, but that it should be possible to relax this requirement. 
\end{remark}

Similar to the case of paths, thin homotopy equivalence encodes reparametrizations and local cancellations such as folding:
\begin{enumerate}
    \item \textbf{Reparametrization.} Given a surface $\bX \in C^1([0,1]^2, V)$, and a corner and edge preserving reparametrization $\phi: [0,1]^2 \to [0,1]^2$, then $\bX \sim_{\thinhom} \bX \circ \phi$.
    \item \textbf{Folding.} Given a surface $\bX \in C^1([0,1]^2, V)$, we say that a \emph{fold} is a region of a surface of the form $\bX \concat_h \bX^{-h}$ or $\bX \concat_v \bX^{-v}$. Two surfaces which differ by folds are thin homotopy equivalent.
\end{enumerate}

\begin{figure}[!h]
    \includegraphics[width=\linewidth]{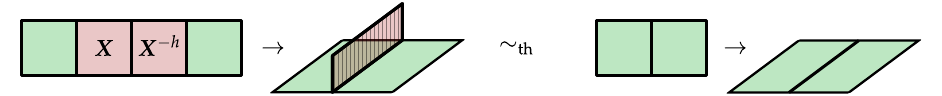}
\end{figure}

\noindent However, thin homotopy of surfaces may yield \emph{non-local} cancellations, as the following shows.

\begin{proposition} \label{prop:rp2_example}
    Consider a map $\bX^{RP} \in C^1([0,1]^2, \R^n)$ which factors as 
    \begin{align} \label{eq:rp2_example_factorization}
        \bX^{RP}: [0,1]^2 \xrightarrow{q} S^2 \xrightarrow{Y} \RP^2 \xrightarrow{g} \R^n,
    \end{align}
    where $q$ is a map which covers the sphere and sends the boundary of $[0,1]^2$ to a basepoint $* \in S^2$, $Y$ is the double cover obtained by identifying antipodal points, and $g$ is a smooth map. Then $\bX^{RP}$ is thinly null homotopic.
\end{proposition}
\begin{proof}
We represent $\mathbb{RP}^2 = D^2/\sim$ as the quotient of a disc $D^2$ by the antipodal identification on the boundary. We compose the antipodal projection map $Y$ with a map $c: \RP^2 \to S^2$ which collapses the boundary circle to a single point, to obtain $c \circ Y: S^2 \to S^2$ as in Figure \ref{fig:RP2fig}.

\begin{figure}[!h]   \includegraphics[width=\linewidth]{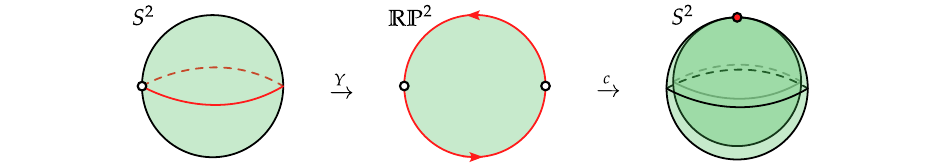}
 \caption{Collapsing map}
\label{fig:RP2fig}
\end{figure}

The map $c \circ Y$ has degree $0$ and hence is null-homotopic by the Hurewicz theorem. Therefore, given a map $g' : S^2 \to \mathbb{R}^n$, the map $g' \circ c \circ Y \circ q$ is thinly null-homotopic. Indeed, this homotopy is a `fold-cancellation' given by `opening' the hole at the top of the sphere in order to contract the map to a point. 
Thus, to show that $\bX^{RP}$ is thinly null-homotopic, it suffices to construct a thin homotopy between $g$ and $g' \circ c$, for some map $g'$. To this end, let $p: [0,1] \to D^2$ be the map parametrizing the top  boundary of the disc, and let $\gamma : [0,1] \to \mathbb{RP}^2$ be the corresponding loop in $\mathbb{RP}^2$, which generates the fundamental group $\pi_{1}(\mathbb{RP}^2) \cong \mathbb{Z}/2$. The loop $g \circ \gamma: [0,1] \to \mathbb{R}^n$ is contractible. Let $h : [0,1]^2 \to \mathbb{R}^n$ be a smooth homotopy that contracts $\gamma$ to a point. 

It is possible to modify $g$ on a collar neighbourhood of the boundary of $D^2$, so that its restriction to both the top half and to the bottom half of the disc agrees with $h$. This can be pictured as in Figure \ref{fig:homotopyg1}
\begin{figure}[!h]
\includegraphics[width=\linewidth]{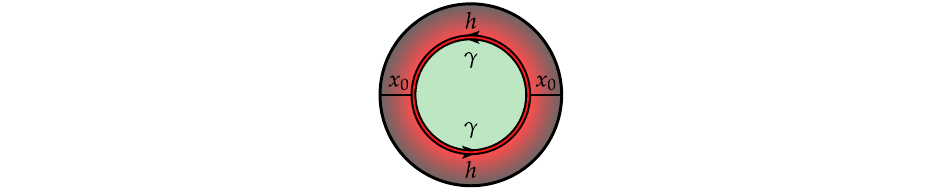}
\caption{The map $g_1$}
\label{fig:homotopyg1}
\end{figure}
\medskip

Let $g_{1} : \mathbb{RP}^2 \to \mathbb{R}^n$ denote the new map. It is obtained by adding a surface, namely the image of $h$, to the original map $g$. Since it sends the boundary of the disc $D^2$ to a single point $x_0 \in \mathbb{R}^n$ it factors as $g_{1} = g' \circ c$. Furthermore, the new map $g_{1}$ is homotopic to $g$ through a homotopy which is restricted to lie within the image of $g_{1}$. Hence $g$ and $g_{1}$ are thinly homotopic to each other. This completes the proof. 
 \end{proof}

Proposition \ref{prop:rp2_example} shows that, in contrast to the case of paths, it is possible for an \emph{immersed surface} to be thinly null homotopic. Indeed, it suffices to take the map $g$ from the Proposition to be an embedding. Hence, there are thin homotopies that do not only involve reparametrizations and fold cancellations. Furthermore, Proposition \ref{prop:rp2_example} also shows that, unlike in the case of paths, surfaces do not admit unique reductions via local cancellations. As a result, in order to see that a surface is thinly null-homotopic, it may be necessary to `backtrack' by introducing new surfaces. This means that techniques used to show the injectivity of the path signature, which usually involve taking the reduced form of a path, will not generalize to surfaces.  

 Finally, one upshot of the present discussion is that generalizing the equivalent definitions of thin homotopy for paths from~\Cref{thm:path_thin_homotopy} to the case of surfaces is not entirely straightforward. Indeed, the naive generalizations may not be true. For example, the straightforward analog of the image condition \ref{I1} fails, as the following corollary shows.

 \begin{corollary} \label{cor:rp2_nonexist}
    Let $X^{RP}: [0,1]^2 \to \R^n$ be defined as in~\Cref{prop:rp2_example}, with $g$ an embedding. There does not exist a null-homotopy $H: [0,1]^3 \to \R^n$ of $X^{RP}$ such that $\im(H) \subseteq \im(X^{RP})$. 
\end{corollary}
\begin{proof}
    Assume that such a homotopy $H$ exists. Since it is contained in the embedded surface $\im(X^{RP})$, it lifts to a null homotopy $\hat{H} : [0,1]^3 \to \RP^2$ of $f \circ q : [0,1]^2 \to \RP^2$. This is a contradiction since $f \circ q$ is a generator for $\pi_{2}(\RP^2) = \Z$. 
\end{proof}

In addition, it is not immediately clear how to generalize the word condition \ref{W1}, the tree condition \ref{F1}, and the analytic condition \ref{A1}. One of the contributions of this article is to suggest a way to adapt all the conditions from Theorem \ref{thm:path_thin_homotopy} to the two-dimensional setting, which is stated in~\Cref{thm:surface_thin_homotopy}.
However, our first task will be to study the generalization of the signature condition \ref{S1} by first introducing the \emph{surface signature}. We will begin by introducing the required algebraic structures and the concept of surface holonomy.

\subsection{Double Groupoids and Crossed Modules of Groups}
The algebraic structure inherent in the composition of surfaces is most naturally encoded using the formalism of double groupoids. However, in this paper, we prefer to use the equivalent concept of crossed modules, since they are more convenient to work with algebraically. In this section, we briefly recall this equivalence and construct the thin crossed module of surfaces.

\begin{definition}{\cite[Theorem 2.13]{martins_surface_2011}}
    The \emph{thin fundamental double groupoid} $\thingroupoid(V)$ is an edge symmetric double groupoid where
    \begin{align}
        \thingroupoid_0(V) \coloneqq V, \quad \thingroupoid_1(V) \coloneqq C^1([0,1], V)/\sim_{\thinhom}, \andd \thingroupoid_2(V) \coloneqq C^1([0,1]^2, V)/\sim_{\thinhom}.
    \end{align}
    The \emph{thin double group} $\thingroup^{\square}(V)$ is an edge symmetric double group where
    \begin{align}
        \thingroup^\square_0(V) = \{*\}, \quad \thingroup^\square_1(V) \coloneqq \thingroupoid_1(V)/\sim_{\trans}, \andd \thingroup_2^{\square}(V) \coloneqq \thingroupoid_2(V)/\sim_{\trans}.
    \end{align}
   
\end{definition}

As in the case of paths, our convention will be to use surfaces based at the origin as representatives of the translation equivalence classes. Two thin homotopy classes $[\bX], [\bY] \in \thingroup^\square_2(V)$ with representatives $\bX, \bY \in C^1([0,1]^2, V)$, where $\bX_{0,0} = 0$ and $\bY_{0,0} = \bX_{1,0}$, are \emph{horizontally composable} if $\partial_r \bX \sim_{\thinhom} \partial_l \bY$. Let $h \in C^1([0,1]^2, V)$ be a thin homotopy between $\partial_r \bX$ and $\partial_l \bY$. We define the \emph{horizontal composition} of $[\bX]$ and $[\bY]$ to be 
\begin{align} \label{eq:horizontal_composition}
    [\bX] \concat_h [\bY] \coloneqq [ \bX \concat_h h \concat_h \bY].
\end{align}
Next, we define the \emph{horizontal inverse} of a surface in $\thingroup^\square_2(V)$ by
\begin{align} \label{eq:horizontal_inverse}
    \bX^{-1}_{s,t} \coloneqq \bX_{1-s, t} - \bX_{1,0} \andd [\bX]^{-1} \coloneqq [\bX^{-1}].
\end{align}
The vertical operations can be defined in an analogous manner; see~\cite[Section 2.3.3]{martins_surface_2011} for details. 
In this article, we will primarily work with a related algebraic structure called a crossed module. 

\begin{definition}
    \label{def:GCM}
        A \emph{pre-crossed module of groups},
        \begin{align}
            \cmG = \left(\cmb: G_1\rightarrow G_0,\  \gt: G_0 \rightarrow \Aut(G_1) \right)
        \end{align}
        is given by two groups $(G_0, \cdot), (G_1, *)$, a group morphism $\cmb: G_1 \rightarrow G_0$ and a left action of $G_0$ on $G_1$ by group automorphisms, denoted elementwise by $g \gt \cdot : G_1 \rightarrow G_1$ for $g \in G$. These data are required to satisfy 
        \begin{align}
            \cmb(g \gt E) = g \cdot \cmb(E) \cdot g^{-1} \quad \text{for} \quad  g \in G_0 \andd E \in G_1.
        \end{align}
        We say $\cmG$ is a \emph{crossed module of groups} if it also satisfies the \emph{Peiffer identity}
        \begin{align}
            \cmb(E_1) \gt (E_2) = E_1 * E_2 *E_1^{-1} \quad \text{for} \quad E_1, E_2 \in G_1.
        \end{align}
        A \emph{(pre-)crossed module of Lie groups} is the same as above, except $G_0$ and $G_1$ are Lie groups, and all morphisms are smooth.
        Given another (pre-)crossed module $\cmH = (\cmb: H_1 \to H_0, \gt)$, a \emph{morphism of (pre-)crossed modules} $f = (f_1, f_0) : \cmG \to \cmH$ consists of group homomorphisms $f_0: G_0 \to H_0$ and $f_1 : G_1 \to H_1$ such that, for all $g \in G_0$ and $E \in G_1$, we have
        \begin{align}\nonumber
            \delta \circ f_1(E) = f_0 \circ \delta(E) \andd f_1(g \gt E) = f_0(g) \gt f_1(E).
        \end{align}
        The categories of crossed modules of groups and Lie groups are respectively denoted $\XGrp \label{pg:XGrp}$ and $\XLGrp \label{pg:XLGrp}$.
\end{definition}

Crossed modules of groups are equivalent to \emph{double groups}, which are defined to be edge-symmetric double groupoids with thin structure and with a single object. We denote the category of double groups by $\DGrp$.

\begin{theorem}{\cite[Section 6.6]{brown_nonabelian_2011}} \label{thm:equiv_xmod_dgrp}
    There is an equivalence of categories between $\DGrp$ and $\XGrp$. 
\end{theorem}

Here we will consider two examples of crossed modules which will be used later.

\begin{definition}
    \label{ex:thin_cm}
    The \emph{thin crossed module} of a vector space $V$ is the crossed module associated to the thin double group $\thingroup^\square(V)$. In particular, let %
    \begin{align} \label{eq:thin_group}
        \thingroup_1(V) \coloneqq \thingroup^\square_1(V) \andd 
        \thingroup_2(V) \coloneqq \left\{ \bX \in \thingroup^\square_2(V) \, : \, \bX_{s,0} = \bX_{0, t} = \bX_{1,t} \sim_{\thinhom} 0 \right\}.
    \end{align}
    The group operation $\concat$ in $\thingroup_2(V)$ is defined as the \emph{reversed}\footnote{The group operation is defined via the equivalence between $\DGrp$ and $\XGrp$ in~\cite[Section 6.6]{brown_nonabelian_2011}, and is determined by the convention (starting point and orientation) used for the boundary $\partial$ given in~\Cref{eq:surface_boundary}. We use the same convention as~\cite{martins_surface_2011,lee_random_2023, lee_surface_2024, chevyrev_multiplicative_2024-1}, which results in the reversed ordering. } horizontal composition \eqref{eq:horizontal_composition},  \label{pg:thin_cm_composition}
    \begin{align} \label{eq:thin_cm_composition}
        \bX \concat \bY \coloneqq \bY \concat_h \bX
    \end{align}
    and the inverse is given by the horizontal inverse \eqref{eq:horizontal_inverse}.
    The \emph{crossed module boundary map} $\delta: \thingroup_2(V) \to \thingroup_1(V)$ is given by the boundary $\partial$ of the surface from~\eqref{eq:surface_boundary}
    \begin{align}
        \delta(\bX) \coloneqq \partial(\bX) = \partial_t(\bX)^{-1}.
    \end{align}
    For a path $\bx \in C^1([0,1],V)$ with $\bx_0 = 0$, we define the degenerate surface $\sigma^\bx : [0,1]^2 \to V$ by
    \begin{align} \label{eq:degenerate_surface}
        \sigma^\bx_{s,t} \coloneqq \bx_{s}.
    \end{align}
    Finally, we define a left action of $\thingroup_1(V)$ on $\thingroup_2(V)$ by 
    \begin{align} \label{eq:thin_group_action}
        \bx \gt \bX \coloneqq  \sigma^\bx \concat_h (\bx_1 + \bX) \concat_h (\bx_1 + \sigma^{\bx^{-1}}),
    \end{align}
    where we must translate the surfaces such that they are composable.
\end{definition}
This indeed forms a crossed module by~\cite[Proposition 6.2.4]{brown_nonabelian_2011}. Furthermore, linear maps $\phi: V \to W$ preserve thin homotopy equivalence of maps, and thus we obtain a functor
\begin{align}
    \cmthingroup: \Vect \to \XGrp.
\end{align}

\begin{definition}{\cite[Section 2.2]{brown_nonabelian_2011}} \label{ex:fund_crossed_module}
    Let $(C, c_0)$ be a based 2-dimensional CW complex.
    The \emph{fundamental crossed module of $(C,c_0)$} is given by
    \begin{align}
        \bpi(C, C_1) \coloneqq \left( \partial: \pi_2(C, C_1), \to \pi_1(C_1), \gt\right),
    \end{align}
    where $\pi_1$ is the fundamental group of the $1$-skeleton $C_1$, and $\pi_2(C,C_1)$ is the relative homotopy group of $C$ with respect to the 1-skeleton $C_1$.
    The later is defined to be the group of homotopy classes of maps $\bX: [0,1]^2 \to C$ such that the boundary of $[0,1]^2$ is sent to $C_{1}$, the corners are sent to $c_{0}$, and such that the left, bottom and right boundary paths are null homotopic in $C_1$, following the convention of~\eqref{eq:thin_group}.
    The multiplication $\concat$ in $\pi_2(C,C_1)$ is given by reversed horizontal concatenation $\bX \concat \bY = \bY \concat_h \bX$, as in~\eqref{eq:thin_cm_composition}. This is equipped with the boundary map $\partial: \pi_2(C, C_1) \to \pi_1(C_1)$ from~\eqref{eq:surface_boundary} and an action of $\pi_1(C_1)$ on $\pi_2(C, C_1)$ given by
    \begin{align}
        \bx \gt \bX \coloneqq \sigma^{\bx} \concat_h \bX \concat_h \sigma^{\bx^{-1}},
    \end{align}
    where $\sigma^\bx$ is the degenerate surface defined in~\eqref{eq:degenerate_surface}.
\end{definition}

\subsection{Crossed Module of Lie Algebras and 2-Connections}

In order to define surface holonomy and the surface signature, we will require the infinitesimal version of crossed modules. 

\begin{definition}
    A \emph{pre-crossed module of Lie algebras}
    \begin{align}\nonumber
        \cmg = (\delta: \fg_1 \to \fg_0, \gt)
    \end{align}
    consists of Lie algebras $(\fg_0, [\cdot, \cdot]_0)$ and $(\fg_1, [\cdot, \cdot]_1)$, a morphism $\delta: \fg_1 \to \fg_0$, and an action $\gt$ of $\fg_0$ on $\fg_1$ by derivations. In other words, the action satisfies
    \begin{align}\nonumber
        x \gt [E, F]_1 = [x \gt E, F]_1 + [E, x \gt F]_1 \andd [x, y]_0 \gt E = x \gt (y \gt E) - y \gt (x \gt E)
    \end{align}
    for all $x,y \in \fg_0$ and $E, F \in \fg_1$. Furthermore, these data are required to satisfy 
    \begin{align}
        \delta(x \gt E) = [x, \delta(E)]_0 \quad \text{for all } \quad x \in \fg_0 \andd E \in \fg_1.
    \end{align}
    We say $\cmg$ is a \emph{crossed module of Lie algebras} if it also satisfies the \emph{Peiffer identity}
    \begin{align}
        \delta(E) \gt E' = [E, E']_1 \quad \text{for all} \quad E,E' \in \fg_1.
    \end{align}
    Suppose $\cmh = (\delta : \fh_1 \to \fh_0, \gt)$ is another (pre-)crossed module of Lie algebras. A \emph{morphism of (pre-)crossed modules} $f = (f_0, f_1) : \cmg \to \cmh$ consists of two Lie algebra morphisms $f_0 : \fg_0 \to \fh_0$ and $f_1 : \fg_1 \to \fh_1$ such that for all $x \in \fg_0$ and $E \in \fg_1$, 
    \begin{align} \label{eq:cmla_morphism}
        \delta \circ f_1(E) = f_0 \circ \delta(E)  \andd f_1 ( x \gt E) = f_0(x) \gt f_1(E).
    \end{align}
    The category of crossed modules of Lie algebras is denoted $\XLie$. \label{pg:XLie}
\end{definition}

The notion of a 2-connection is defined in terms of a crossed module of Lie algebras.

\begin{definition}{\cite[Definition 2.16, Proposition 2.17]{martins_two-dimensional_2010}} \label{def:Con}
    Let $\cmg = (\delta: \fg_1 \to \fg_0)$ be a crossed module of Lie algebras. A \emph{2-connection} valued in $\cmg$ over $V$ is a pair $(\con, \Con)$, such that $\con \in \Omega^1(V, \fg_0)$ is a $\fg_0$-valued $1$-form and $\Con \in \Omega^2(V, \fg_1)$ is a $\fg_1$-valued $2$-form. The 1-curvature $\curv^{\con, \Con} \in \Omega^2(V, \fg_0)$ and 2-curvature $\Curv^{\con, \Con} \in \Omega^3(V, \fg_1)$ are respectively defined as follows \label{pg:Curv}
    \begin{align} 
        \curv^{\con, \Con} = \curv^{\con} - \delta \Con \andd \Curv^{\con, \Con} \coloneqq d\Con + \con \wedge^\gt \Con,
    \end{align}
     where $\curv^{\con}$ is the curvature of $\con$ from~\eqref{eq:1_curvature} and $\con \wedge^\gt \Con \in \Omega^3(V, \fg_1)$ is defined by
    \begin{align} \label{eq:wedge_action}
        (\con \wedge^\gt \Con)(X,Y,Z) \coloneqq \con(X) \gt \Con(Y,Z) - \con(Y) \gt \Con(X,Z) + \con(Z) \gt \Con(X,Y).
    \end{align}
    The 2-connection $(\con, \Con)$ is called \emph{fake-flat} (or semi-flat) if its 1-curvature vanishes, $\curv^{\con, \Con} = 0$. We will always work with fake-flat 2-connections.     
\end{definition}

We say that a 2-connection $(\con, \Con)$ is \emph{translation-invariant} if it has the form
\begin{align}
    \con = \sum_{i=1}^d \con^i dz_i \andd \Con = \sum_{i < j} \Con^{i,j} dz_i \wedge dz_j,
\end{align}
where $\con^i \in \fg_0$, $\Con^{i,j} \in \fg_1$, and $z_i$ are linear coordinates on $V$. In this case, we can represent the 2-connection as a pair of linear maps 
\begin{align}
    \con \in L(V, \fg_0) \andd \Con \in L(\Lambda^2 V, \fg_1).
\end{align}
The curvature forms then simplify to
\begin{align}
    \curv^{\con, \Con} = \frac{1}{2}[\con, \con] - \delta\Con \andd \Curv^{\con, \Con} = \con \wedge^{\gt} \Con.
\end{align}

\subsection{Surface Holonomy}

Here, we discuss surface holonomy, the 2-dimensional generalization of parallel transport.
For a surface $\bX \in C^1([0,1]^2, V)$, the \emph{$(s,t)$-tail path of $\bX$} is
\begin{align} \label{eq:tail_path}
    \bx^{s,t}_u \coloneqq \left\{ \begin{array}{cl}
        \bX_{0,2ut} & : u \in [0,1/2] \\
        \bX_{(2u-1)s, t } & : u \in (1/2, 1]
    \end{array} \right.
\end{align}
\begin{figure}[!h]
    \includegraphics[width=\linewidth]{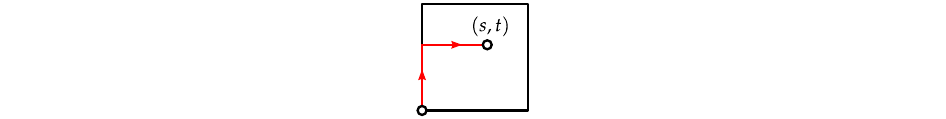}
\end{figure}

\begin{definition}{\cite[Equation 2.13]{martins_surface_2011}} \label{def:sh}
    Let $\cmG = (\delta: G_1 \to G_0, \gt)$ be a crossed module of Lie groups with associated crossed module of Lie algebras $\cmg = (\delta: \fg_1 \to \fg_0, \gt)$. Let $(\con, \Con)$ be a 2-connection valued in $\cmg$ and let $\bX \in C^1([0,1]^2, V)$. Consider the following differential equation for $\Hol^{\con, \Con}(\bX): [0,1]^2 \to G_1$
    \begin{align}\label{eq:general_sh}
        \frac{\partial \Hol^{\con, \Con}(\bX)_{s,t}}{\partial t} = dL_{\Hol^{\con, \Con}(\bX)_{s,t}} \int_0^s \hol^{\con}(\bx^{s',t}) \gt \Con\left( \frac{\partial \bX_{s',t}}{\partial s}, \frac{\partial \bX_{s',t}}{\partial t}\right) ds', \quad \Hol^{\con, \Con}(\bX)_{s,0} = e_{G_1},
    \end{align}
    where the tail path $\bx^{s,t}: [0,1] \to V$ is defined in~\eqref{eq:tail_path} and $\hol^{\con}(\bx^{s',t})$ is the parallel transport of $\con$ along $\bx^{s',t}$, as given in \Cref{def:ph}.
    We define the \emph{surface holonomy of $(\con, \Con)$ along $\bX$} to be
    \begin{align}
        \Hol^{\con, \Con}(\bX) \coloneqq \Hol^{\con, \Con}(\bX)_{1,1}.
    \end{align}
\end{definition}

In~\cite[Theorem 2.32]{martins_surface_2011}, it is shown that the surface holonomy gives rise to a morphism of double groupoids. When the 2-connection is translation-invariant, we apply ~\Cref{thm:equiv_xmod_dgrp} to state a version of the theorem in terms of the crossed module $\cmthingroup(V)$.

\begin{theorem}{\cite[Theorem 2.32]{martins_surface_2011}} \label{thm:sh_cm_morphism}
    Let $(\con, \Con)$ be a translation-invariant 2-connection valued in $\cmg$. Then the maps
    \begin{align}
        F^{\con, \Con} = (\Hol^{\con, \Con}, \hol^{\con}) : \cmthingroup(V) \to \cmG
    \end{align}
    define a morphism of crossed modules. %
\end{theorem}

This theorem has two main implications.
\begin{enumerate}
    \item Surface holonomy is invariant with respect to thin homotopies of surfaces.
    \item Surface holonomy respects concatenation of surfaces.
\end{enumerate}

\subsection{Free Crossed Module of Lie Algebras} \label{ssec:free_xlie}

For the next two sections, we discuss the \emph{surface signature}, which was originally introduced by Kapranov~\cite{kapranov_membranes_2015}, and recently studied from the analytic perspective of irregular surfaces in~\cite{lee_surface_2024,chevyrev_multiplicative_2024-1}. Generalizing the path signature, the surface signature is defined to be the surface holonomy of a universal $2$-connection valued in a free crossed module. Here, we begin with an overview of the essential properties of free crossed modules of Lie algebras, and refer the reader to~\Cref{apx:xlie} for further details and proofs. \medskip

Let $\VL$ \label{pg:VL} denote the comma category $(\id \downarrow \For)$ associated to the functors $\id: \Vect \to \Vect$ and $\For: \Lie \to \Vect$. An object of $\VL$ is given by the data of a vector space $V$, a Lie algebra $\fg$, and a linear map $s : V \to \fg$. A morphism $f = (f_1, f_0): (s : V \to \fg) \to (t: W \to \fh)$ consists of a linear map $f_1: V \to W$ and a Lie algebra morphism $f_0: \fg \to \fh$ such that $t \circ f_1 = f_0 \circ s$ as linear maps. There exists a natural forgetful functor
\begin{align}
\For : \XLie \to \VL.
\end{align}
The free crossed module functor 
\begin{align}
    \Fr: \VL \to \XLie
\end{align}
is defined to be the left adjoint. In the following, we describe this functor in detail. Given an object $s: V \to \fg$ of $\VL$, we first build the free $\fg$-representation on $V$, which is $U(\fg) \otimes V$, where $U(\fg)$ is the universal enveloping algebra of $\fg$, where the $\fg$ action is given by
\begin{align}
    x \gt (x_1 \otimes\ldots \otimes x_k) \otimes v = (x \otimes x_1 \otimes \ldots \otimes x_k) \otimes v,
\end{align}
and where the linear map $s$ is lifted to a morphism of $\fg$-representations as follows
\begin{align}
    \delta : U(\fg) \otimes V \to \fg, \quad (x_1 \otimes \ldots \otimes x_k) \otimes v \mapsto [x_1, \ldots [x_k, s(v)]\ldots].
\end{align}
We define the \emph{Peiffer subspace} of $U(\fg) \otimes V$ to be
\begin{align}
    \Pf(U(\fg) \otimes V) \coloneqq \SPAN\left\{\delta(X) \gt Y + \delta(Y) \gt X \, : \, X, Y \in U(\fg) \otimes V\right\}.
\end{align}
By taking the quotient with respect to this subspace, we obtain the \emph{free crossed module of Lie algebras generated by $s$}, 
\begin{align}
    \Fr(s) \coloneqq \left( \delta: (U(\fg) \otimes V)/\Pf(U(\fg) \otimes V) \to \fg, \gt \right),
\end{align}
where the Lie bracket of $X, Y \in (U(\fg) \otimes V)/\Pf(U(\fg) \otimes V)$ is defined by
\begin{align}
    [X,Y] = \delta(X) \gt Y = -\delta(Y) \gt X.
\end{align}
This satisfies the following universal property, which is proved in~\Cref{apxthm:free_xlie_global}.

\begin{theorem} \label{thm:free_xlie}
    Let $(s: V \to \fh) \in \VL$, $\cmg = (\delta: \fg_1 \to \fg_0, \gt) \in \XLie$, and 
    \begin{align}
    f = (f_1, f_0) : (s: V \to \fh) \to \For(\delta: \fg_1 \to \fg_0, \gt).
    \end{align}
    Then there is a unique morphism of crossed modules 
    \begin{align}
    F = (F_1, F_0): \Fr(s) \to \cmg,
    \end{align}
    such that $F_1 \circ \eta_s = f_1$ and $F_0 = f_0$.
\end{theorem}

Now consider the functor
\begin{align}
    \wedge: \Vect \to \VL,
\end{align}
which sends a vector space $V \in \Vect$ to
\begin{align}
    s_V: \Lambda^2 V \to \FL(V), \quad u \wedge v \mapsto [u, v].
\end{align}
Composing with the free crossed module functor gives 
\begin{align} \label{eq:free_cmla_vector_space_functor}
    \cmk = \Fr \circ \wedge : \Vect \to \XLie.
\end{align}
Given a vector space $V$, the \emph{free crossed module of Lie algebras generated by $V$} is the result of applying this functor, and is given by
\begin{align} \label{eq:free_cmla_vector_space}
    \cmk(V) \coloneqq \left( \delta: \fk_1(V) \to \fk_0(V), \gt\right),
\end{align}
where
\begin{align}
    \fk_1(V) = (T(V) \otimes \Lambda^2 V)/\Pf(T(V) \otimes \Lambda^2 V).
\end{align}

\subsection{The Surface Signature}

Following Kapranov~\cite{kapranov_membranes_2015}, we consider the \emph{universal translation-invariant 2-connection}. This is a $2$-connection $(\conk, \Conk)$ over $V$ valued in $\cmk(V)$. The translation invariant $\fk_0(V)$-valued 1-form $\conk \in \Omega^1(V, \fk_0(V))$ is the one defined in~\eqref{eq:univ_connection}, and hence corresponds to the identity endomorphism of $V$, $\conk = \mathrm{id}_V$. Similarly, the translation invariant $\fk_1(V)$-valued 2-form $\Conk$ can be understood as the identity endomorphism of $\Lambda^2V$ 
\begin{align}
    \Conk = \mathrm{id}_{\Lambda^2V} \in \Lambda^2V^* \otimes \Lambda^2V \subset \Omega^2(V, \fk_{1}(V)). 
\end{align} 
Alternatively, $\conk$ and $\Conk$ can be viewed as the linear maps given by including generators 
\begin{align}
    \conk \in L(V, \fk_0(V)), \qquad \Conk \in L(\Lambda^2V, \fk_1(V)).
\end{align}
In coordinates, they are given by 
\begin{align} \label{eq:univ_2con}
    \conk = \sum_{i=1}^n  dz_i e_i \in \Omega^1(V, \fk_0) \andd \Conk = \sum_{i < j}  dz_i \wedge dz_j \, (e_i \wedge e_j) \in \Omega^2(V, \fk_1).
\end{align}
We follow a similar procedure as for path signatures and consider truncations of $\fk_1$ via the $\fk_0$-lower central series of $\fk_1$, defined by 
\begin{align}
    \LCS_1(\fk_0, \fk_1) = \fk_1 \andd \LCS_r(\fk_0,\fk_1) = \fk_0 \gt \LCS_{r-1}(\fk_0, \fk_1).
\end{align}
This allows us to define the $n$-truncated free crossed module as follows
\begin{align}
    \cmk^{(n)} \coloneqq (\delta: \fk_1^{(n)} \to \fk_0^{(n)}, \gt) \ \ \text{ where }  \ \ \fk_1^{(n)} \coloneqq \fk_1 / \LCS_{n+1}(\fk_0, \fk_1).
\end{align}
We can integrate $\fk_1^{(n)}$ as the Lie group of formal exponentials $K_1^{(n)}$, and consider the projective limit $\com{K}_1$, where
\begin{align}
    K_1^{(n)} \coloneqq \{ \exp(x) \, : \, x \in \fk_1^{(n)}\} \andd \com{K}_1 \coloneqq \lim_{\longleftarrow} K_1^{(n)}.
\end{align}
This defines a crossed module of groups \label{pg:com_cmK}
\begin{align} 
    \com{\cmK}(V) = (\delta: \com{K}_1(V) \to \com{K}_0(V), \gt),
\end{align}
with a corresponding crossed module of Lie algebras \label{pg:com_cmk}
\begin{align} 
    \com{\cmk}(V) = (\delta: \com{\fk}_1(V) \to \com{\fk}_0(V), \gt), \quad \com{\fk}_1(V) \coloneqq \lim_{\longleftarrow} \fk_1^{(n)}(V).
\end{align}

\begin{remark}
    In~\cite{lee_surface_2024}, it is shown that by considering the free crossed module of associative algebras, one can construct the completion $\com{K}_1(V)$ in a way that is similar to the construction of $\com{K}_0(V)$ in the completed tensor algebra. Furthermore, one can also construct analytic completions in this setting by using Banach or topological algebras. As we do not require these details here, we will work with formal completions. 
\end{remark}

\begin{definition} \label{def:ssig}
    Let $\bX \in C^1([0,1]^2, V)$. For $n \in \N$, the \emph{$n$-truncated surface signature of $\bX$}, denoted $\Sig^{(n)}(\bX)$, is defined to be the surface holonomy, as defined in~\Cref{def:sh}, of the 2-connection obtained by projecting the universal 2-connection $(\conk, \Conk)$ from~\eqref{eq:univ_2con} onto $\cmk^{(n)}(V)$.
    The \emph{surface signature of $\bX$} is defined to be the projective limit
    \begin{align}
        \Sig(\bX) \coloneqq \lim_{\longleftarrow} \Sig^{(n)}(\bX) \in \com{K}_1(V). 
    \end{align}
\end{definition}

As in the case of the path signature, there is a universal property ~\cite[Theorem 4.29]{lee_surface_2024} which implies that the surface signature fits into a natural transformation
\begin{align} \label{eq:ssig_nt}
    \bS = (\Sig, \sig) : \cmthingroup \Rightarrow \com{\cmK}.
\end{align}
We note that the $n$-truncated group $K_1^{(n)}$ is equipped with a natural topology induced by the Lie algebra $\fk_1^{(n)}$ via the exponential. The continuity of the surface extension theorem from~\cite[Theorem 5.40]{lee_surface_2024} implies that the surface signature is continuous. 

\begin{proposition}{\cite[Theorem 5.40]{lee_surface_2024}} \label{prop:surface_signature_continuous}
    Let $C^1_0([0,1]^2, V)$ be the space of based smooth surfaces $\bX$ such that $\bX_{0,0} = 0$. The $n$-truncated surface signature
    \begin{align}
        \Sig^{(n)} : C^1_0([0,1]^2, V) \to K_1^{(n)}(V),
    \end{align}
    where $C^1_0([0,1]^2,V)$ is equipped with the Lipschitz norm, is continuous.
\end{proposition}

\section{Abelianization} \label{sec:abelianization}

In this section, we establish our first characterization of the kernel of the surface signature. This is done in two steps. First, we consider a gauge transformation which converts the universal 2-connection of~\eqref{eq:univ_2con} into an \emph{abelianized} 2-connection. This is a connection that is valued in the center of $\com{\fk}_1(V)$. Using the representation theory of $\GL(V)$, we produce an explicit expression for the $2$-curvature of this abelianized $2$-connection. Second, we relate the surface signature to an integral of this abelianized 2-curvature. Using these two results, we show in~\Cref{thm:ssig_equiv_to_abelianization} that the surface signature of a closed surface is encoded by integration over all polynomial 2-forms. 

\subsection{Gauge Transformations and Abelianization} \label{ssec:abelianization}
We begin with the general definition of gauge transformations for 2-connections.

\begin{definition}
    Let $\cmG \in \XLGrp$ %
    and $\cmg \in \XLie$ be its associated crossed module of Lie algebras. A \emph{gauge transformation} is a pair $(\gtrans, \Gtrans)$, where $\gtrans \in C^\infty(V, G_0)$ and $\Gtrans \in \Omega^1(V, \fg_1)$. The two components act on a 2-connection $(\con, \Con)$ valued in $\cmg$ in the following way
    \begin{align}
        \gtrans \cdot (\con, \Con) &= (\gtrans d(\gtrans^{-1}) + \gtrans \con \gtrans^{-1}, \gtrans \gt \Con)\\
        \Gtrans \cdot (\con, \Con) &= \Big(\con + \delta(\Gtrans), \Con - d\Gtrans - \con \wedge^{\gt} \Gtrans - \frac{1}{2} [\Gtrans, \Gtrans]\Big),
    \end{align}
    where $\wedge^{\gt}$ is defined in~\eqref{eq:wedge_action}. 
    As a convention, $(\gtrans, \Gtrans)$ acts first by $\gtrans$, and then by $\Gtrans$:
    \begin{align}
        (\gtrans, \Gtrans) \cdot (\con, \Con) = \Gtrans \cdot \big(\gtrans \cdot (\con, \Con)\big).
    \end{align}
\end{definition}

By direct computation, the two components of the curvature are given by
\begin{align}
    \curv^{\gtrans \cdot(\con, \Con)} = \gtrans \curv^{\con, \Con}\gtrans^{-1} \andd \Curv^{\gtrans \cdot(\con, \Con)} = \theta \gt \Curv^{\con, \Con}
\end{align}
and
\begin{align}
    \curv^{\Gtrans \cdot(\con, \Con)} = \curv^{\con, \Con} \andd \Curv^{\Gtrans \cdot(\con,\Con)} = \Curv^{\con, \Con} - \curv^{\con, \Con} \gt \Gtrans.
\end{align}
In particular, for a fake-flat 2-connection, where $\curv^{\con, \Con} = 0$, we obtain
\begin{align} \label{eq:gauge_trans_fake_flat_curvature}
    \curv^{(\gtrans, \Gtrans) \cdot (\con, \Con)} = 0 \andd \Curv^{(\gtrans, \Gtrans)\cdot (\con, \Con)} = \theta \gt \Curv^{\con, \Con}.
\end{align}

An interesting feature of higher gauge theory is that given a fake flat 2-connection $(\con, \Con)$ over a contractible space, one can always find a gauge transformation such that $(\gtrans, \Gtrans) \cdot (\con, \Con) = (0, \Con^{\ab})$, with $\Con^{\ab}$ valued in the kernel of $\delta: \fg_1 \to \fg_0$~\cite[Theorem 4.3]{samann_towards_2020} (see also ~\cite{demessie_higher_2015} and ~\cite{voronov_non-abelian_2012} for similar results). We consider this construction for the universal 2-connection. \medskip
 
Let $(\conk, \Conk)$ be the universal 2-connection from~\eqref{eq:univ_2con}. First, we define $\gtrans \in C^\infty(V, \com{K}_0)$ by exponentiating a Lie algebra valued function $\eta \in C^\infty(V, \fk_0)$ 
\begin{align} \label{eq:kap_gtrans}
    \gtrans(x) = \exp_\otimes(\eta(x)) \quad \text{where} \quad \eta(x) = \sum_{i=1}^d z_i e_i.
\end{align}
Note that $\conk = d\eta$. Then, we have
\begin{align}
    \gtrans \cdot \conk &= \Ad_{\exp(\eta)}(d\eta) + \exp(\eta) d\exp(-\eta) \\
    & = \left(\exp(\ad_\eta) + \frac{1- \exp(\ad_\eta)}{\ad_\eta}\right) (d\eta) \\
    & = \left(\sum_{k=0}^\infty \frac{\ad_\eta^k}{k!} - \sum_{k=1}^\infty \frac{\ad_\eta^{k-1}}{k!}\right)(d\eta) \\
    & = \sum_{k=0}^\infty \frac{k+1}{(k+2)!} \ad_\eta^k [\eta, d\eta] \\
    & = \delta \left(\sum_{k=0}^\infty \frac{k+1}{(k+2)!} \ad_\eta^k U\right),
\end{align}
where
\begin{align}
    U = \sum_{i < j} e_i \wedge e_j (z_i dz_j - z_j dz_i) \in \Omega^1(V, \fk_1).
\end{align}
Note that $U$ has the property that  
\begin{align}
    \delta(U) = [\eta, d\eta] \andd dU = 2 \Conk.
\end{align}
Now, if we let
\begin{align} \label{eq:kap_Gtrans}
    \Gtrans = -\sum_{k=0}^\infty \frac{k+1}{(k+2)!} \ad_\eta^k U,
\end{align}
then we have
\begin{align}
    (\gtrans, \Gtrans) \cdot \conk = 0.
\end{align}
Thus, the transformed 2-connection is
\begin{align} \label{eq:abelianized_2_con}
    \Conk^\ab \coloneqq (\gtrans, \Gtrans) \cdot \Conk = \gtrans \gt \Conk - d\Gtrans  + \frac{1}{2} [\Gtrans, \Gtrans].
\end{align} 
Note that because fake-flatness is preserved by gauge transformations, $\Conk^{\ab}$ is valued in the completion of $\fa_1 \coloneqq \ker(\delta: \fk_1 \to \fk_0)$, which is an abelian Lie algebra.\medskip

While the expression~\eqref{eq:abelianized_2_con} for the abelianized 2-connection $(0, \Conk^{\ab})$ is quite complicated, its 2-curvature can be computed explicitly. The 2-curvature of the universal 2-connection is
\begin{align} \label{eq:kapranov_2_curvature}
     \Curv^{\conk, \Conk} = \sum_{i < j < k} B_{i,j,k} dz_i \wedge dz_j \wedge dz_k \in \fa_1 \otimes \Omega^{3, \cl}(V)
\end{align}
where
\begin{align} \label{eq:jacobi_elements}
    B_{i,j,k} = [e_i, e_{j}\wedge e_{k}] - [e_j, e_{i}\wedge e_{k}] + [e_k, e_{i}\wedge e_{j}] \in \fa_1 \subset \fk_1.
\end{align}
Then, applying the gauge transformation to the 2-curvature using~\eqref{eq:gauge_trans_fake_flat_curvature}, we obtain
\begin{align} \label{eq:abelianized_2curv}
    \Curv^{\ab} \coloneqq \Curv^{0, \Conk^{\ab}} &= \gtrans \gt \Curv^{\conk, \Conk} = \sum_{m=0}^\infty \frac{1}{m!} \ad_{\eta}^m\left(\sum_{1 \leq i < j < k \leq n} B_{i,j,k} \, dz_i \wedge dz_j \wedge dz_k\right).
\end{align}
Furthermore, because the abelianized 1-connection is trivial, the Bianchi identity for the 2-connection~\cite[Proposition 2.20]{martins_surface_2011} shows that $d\Curv^{\ab} = 0$, and thus
\begin{align}
    \Curv^{\ab} \in \com{\fa}_1 \,\com{\otimes} \,\com\Omega^{3,\cl}(V). 
\end{align}

\subsection{Polynomial Differential Forms and Currents} \label{ssec:forms_and_currents}
To understand $\Curv^{\ab}$, we make use of the isomorphism between $\fa_1$ and the space of closed polynomial currents established by Kapranov~\cite{kapranov_membranes_2015}. In this section, we expand on the exposition in~\cite{kapranov_membranes_2015} and describe the spaces of polynomial differential forms and currents as $\GL(V)$-representations.
Here, we will consider $V$ to be a finite-dimensional $\K$-vector space, for $\K = \R$ or $\K = \C$. Polynomial differential forms and their completions are naturally graded by their total \emph{weight} $r$. We define  \label{eq:poly_forms}  
\begin{align}
    \poly{\Omega}^m(V)_r = S^{r-m}(V^*) \otimes \Lambda^m(V^*), \quad 
    \poly{\Omega}^m(V) \coloneqq \bigoplus_{r=m}^\infty \poly{\Omega}^m(V)_r , \quad  \com{\Omega}^m(V) \coloneqq \prod_{r=m}^\infty \poly{\Omega}^m(V)_r,
\end{align}
equipped with the usual exterior differential $d^m: \poly{\Omega}^m(V) \to \poly{\Omega}^{m+1}(V)$. The polynomial currents and their completions are defined by taking the graded dual of $\poly{\Omega}^\bullet$ as follows \label{eq:poly_currents}
\begin{align}
    \poly{\Gamma}_m(V)_r \coloneqq \poly{\Omega}^m(V)_r^* = S^{r-m}(V) \otimes \Lambda^m(V), \hspace{7pt}
    \poly{\Gamma}_m(V) \coloneqq \bigoplus_{r=m}^\infty \poly{\Gamma}_m(V)_r, \hspace{7pt}
    \com{\Gamma}_m(V) \coloneqq \prod_{r=m}^\infty \poly{\Gamma}_m(V)_r.
\end{align}
To simplify notation, we will often write $\poly{\Gamma}_\bullet \coloneqq \poly{\Gamma}_\bullet(V)$ and $\poly{\Omega}^\bullet \coloneqq \poly{\Omega}^\bullet(V)$. 
Differential forms and currents are naturally equipped with the structure of $\GL(V)$-representations. An element $g \in \GL(V)$ acts on a current $\alpha = u_1 \cdots u_r \otimes (v_1 \wedge \ldots \wedge v_k) \in \poly{\Gamma}_k(V)_{k+r}$ as follows
\begin{align}
    g \gtd \alpha = (g \cdot u_1) \cdots (g \cdot u_r) \otimes \left( (g\cdot v_1) \wedge \ldots \wedge (g \cdot v_k)\right),
\end{align}
and acts on $ \omega \in \poly{\Omega}^\bullet(V)$ by pullback as follows
\begin{align}
    g \gtd \omega = (g^{-1})^*\omega. 
\end{align}
The exterior differential is $\GL(V)$-equivariant. Note also that the action of the subgroup of scalars $\K^* \subseteq \GL(V)$ induces a grading on differential forms and currents which agrees with the weight grading defined above.

Although the polynomial forms $\poly{\Omega}^\bullet$ and currents $\poly{\Gamma}_\bullet$ are naturally dual to each other, the naive pairing between them is not $\GL(V)$-equivariant. Here, we will use the Cartan calculus to define the ``correct'' pairing which is implicitly used in~\cite{kapranov_membranes_2015}.
In the following, we view $V \subset \fX(V)$ as constant vector fields on $V$; we note that these vector fields commute. These vector fields act on differential forms in two ways. First, for any $v \in V$, we can take the Lie derivative $L_v: \poly{\Omega}^m(V) \to \poly{\Omega}^m(V)$, and for constant vector fields, this is commutative,
\begin{align}
    [L_v, L_u] = L_{[v,u]} = 0.
\end{align}
Second, we can take the interior product $\iota_v: \poly{\Omega}^m(V) \to \poly{\Omega}^{m-1}(V)$, which satisfies
\begin{align}
    \iota_v \iota_u = - \iota_u \iota_v.
\end{align}
For $k \leq m$, we use these operations to define a map $P_{k,m}: \poly{\Gamma}_k(V) \otimes \poly{\Omega}^m(V) \to \poly{\Omega}^{m-k}(V)$ where
\begin{align} \label{eq:Pkm_def}
    P_{k,m}(u_1\cdots u_r \otimes (v_1 \wedge \ldots \wedge v_k) \otimes \omega) \coloneqq L_{u_1} \ldots L_{u_r} \iota_{v_1} \ldots \iota_{v_k} \omega. 
\end{align}

\begin{lemma}
    For $k \leq m$, the map $P_{k,m}: \poly{\Gamma}_k(V) \otimes \poly{\Omega}^m(V) \to \poly{\Omega}^{m-k}(V)$ is $\GL(V)$-equivariant. In particular, for $g \in \GL(V)$, $\alpha \in \poly{\Gamma}_k(V)$, $\omega \in \poly{\Omega}^m(V)$, we have
    \begin{align}
        g \gtd P_{k,m}(\alpha \otimes \omega) = P_{k,m}((g \gtd \alpha) \otimes (g \gtd \omega)). 
    \end{align}
\end{lemma}
\begin{proof}
    First, given $g \in \GL(V)$ and $v \in V$, the pullback satisfies
    \begin{align}
        g^*d\omega = d(g^*\omega) \andd g^* (\iota_v \omega) = \iota_{g^{-1}(v)}(g^*\omega). 
    \end{align}
    Then, by the Cartan formula $L_v = d \iota_v + \iota_v d$, we have
    \begin{align}
        g^* (L_v \omega) = L_{g^{-1}(v)}(g^* \omega). 
    \end{align}
    Thus, by the definition of $P_{k,m}$ in~\eqref{eq:Pkm_def}, we have
    \begin{align}
        g^*P_{k,m}(u_1\cdots u_r \otimes (v_1 \wedge \ldots v_k) \otimes \omega) &= L_{g^{-1}(u_1)} \ldots L_{g^{-1}(u_r)} \iota_{g^{-1}(v_1)} \ldots \iota_{g^{-1}(v_k)} g^*\omega \\
        & = P_{k,m}\left(g^{-1} \gtd (u_1\cdots u_r \otimes (v_1 \wedge \ldots v_k)) \otimes (g^{-1} \gtd \omega)\right). \nonumber
    \end{align}
\end{proof}

We also record a formula relating the exterior derivative and $P_{k,m}$. 

\begin{lemma}
    For $u \otimes (v_1 \wedge \ldots \wedge v_k) \in \poly{\Gamma}_k(V)$, where $u \in S(V)$,  and $\omega \in \poly{\Omega}^m(V)$, we have
    \begin{align}
        dP_{k,m}(u \otimes (v_1 \wedge \ldots  \wedge v_k) \otimes \omega) & = \sum_{i=1}^k (-1)^{i-1} P_{k,m}(u v_i \otimes (v_1 \wedge \ldots \hat{v}_i  \ldots \wedge v_k) \otimes \omega) \\
        & \quad + (-1)^k P_{k,m}(u \otimes (v_1 \wedge \ldots \wedge v_k) \otimes d\omega) \nonumber
    \end{align}
\end{lemma}
\begin{proof}
    This is immediate by the properties $dL_v = L_v d$, $d\iota_v = -\iota_v d + L_v$, and $L_v \iota_u = \iota_u L_v$. 
\end{proof}

Next, we assemble all of the $P_{k,m}$ into a map 
\begin{align}
    P : \poly{\Gamma}_\bullet(V) \otimes \poly{\Omega}^\bullet(V) \to \poly{\Omega}^\bullet(V).
\end{align}
For the final step in defining the equivariant pairing, consider the inclusion of the origin $i : 0 \to V$, which is trivially $\GL(V)$-equivariant. This induces a trivially $\GL(V)$-equivariant pullback
\begin{align}
    i^* : \poly{\Omega}^\bullet(V) \to \poly{\Omega}^\bullet(0) = \K.
\end{align}

\begin{proposition} \label{prop:equivariant_pairing}
    The pairing
    \begin{align} \label{eq:pairing}
        \langle \cdot, \cdot\rangle \coloneqq i^* P_{k,m} : \poly{\Gamma}_k(V) \otimes \poly{\Omega}^m(V) \to \K,
    \end{align}
     is non-degenerate for $k = m$, vanishes for $k \neq m$, and is $\GL(V)$-equivariant. 
\end{proposition}

If necessary, we use $\langle \cdot , \cdot \rangle_m : \poly{\Gamma}_m \otimes \poly{\Omega}^m \to \K$ to specify the degree. We can describe this pairing explicitly in coordinates. Let $(z_{1}, ..., z_{n})$ be a basis of $V^*$ and let $(e_{1}, ..., e_{n})$ be the dual basis of $V$. Given $\alpha = (\alpha_1, \ldots, \alpha_n) \in \N^n$, we define $z^\alpha \coloneqq z_1^{\alpha_1} \cdots z_n^{\alpha_n} \in S(V^*)$ and similarly for $e^{\alpha} \in S(V)$. Given an increasing index set $I = (i_{1} < ... < i_{m})$, we define $dz_{I} = dz_{i_{1}} \wedge ... \wedge dz_{i_{m}} \in \wedge^{m}V^*$. Similarly, we define $e_{I} = e_{i_{1}} \wedge ... \wedge e_{i_{m}}\in \wedge^{m}V$. In this way, we obtain the bases 
\begin{align}
z^{\alpha} \otimes dv_{I} \andd e^{\alpha} \otimes e_{I}
\end{align}
of $\poly{\Omega}^m(V)$ and $\poly{\Gamma}_m(V)$, respectively.
In terms of these coordinates, the pairing satisfies
\begin{align} \label{eq:dualbasisfactorial}
\langle e^{\alpha} \otimes e_{I} , z^{\beta} \otimes dz_{J} \rangle_{m} = (-1)^{\frac{m(m-1)}{2}} \alpha!\, \delta_{\alpha, \beta} \delta_{I,J} \quad \text{where} \quad \alpha! = \prod \alpha_{i}!.
\end{align}

Using the pairing, we can now define the codifferential 
\begin{align} \label{eq:codifferential}
    \partial_m : \poly{\Gamma}_m(V) \to \poly{\Gamma}_{m-1}(V) \quad \text{by} \quad \langle \partial_m \alpha, \omega \rangle = (-1)^m\langle \alpha, d^{m-1}\omega\rangle
\end{align}
for any $\alpha \in \poly{\Gamma}_m(V)$ and $\omega \in \poly{\Omega}^{m-1}(V)$. It is $\GL(V)$-equivariant. 
The \emph{closed forms} and \emph{closed currents} are defined as usual to be
\begin{align}
    \poly{\Omega}^{m, \cl}(V) \coloneqq \ker(d^m: \poly{\Omega}^m \to \poly{\Omega}^{m+1}) \andd \poly{\Gamma}^{\cl}_m(V) \coloneqq \ker( \partial_m : \poly{\Gamma}_m \to \poly{\Gamma}_{m-1})
\end{align}
respectively. By the Poincare lemma, the cochain complex of differential forms is acyclic, and thus by duality, the chain complex of currents is also acyclic. Hence, using the pairing, we have  
\begin{align}
    \poly{\Gamma}^{\cl}_m \cong \im( \partial_{m+1}: \poly{\Gamma}_{m+1} \to \poly{\Gamma}_{m}) \cong \coker( \partial_{m+2} : \poly{\Gamma}_{m+2} \to \poly{\Gamma}_{m+1}) \cong (\poly{\Omega}^{m+1, \cl})^{\vee},
\end{align}
where $(\cdot)^{\vee}$ denotes the graded dual. 
Given $\alpha = \partial_{m+1}(\beta) \in \poly{\Gamma}^{\cl}_m$, where $\beta \in \poly{\Gamma}_{m+1}$, and $\omega = d^m(\eta) \in \poly{\Omega}^{m+1, \cl}$, where $ \eta \in \poly{\Omega}^m$, the pairing between closed forms and closed currents is 
\begin{align} \label{eq:closed_pairing}
    \langle \alpha, \omega \rangle_{\cl} \coloneqq (-1)^{m+1} \langle \alpha, \eta\rangle_{m} = \langle \beta, \omega\rangle_{m+1}.
\end{align} 

In what follows, we will need the duality between closed $2$-currents and closed $3$-forms and it will be useful to have explicit dual bases. Given index list $\bq = (q_1, \ldots, q_r)$, and indices $i, j, k$, such that $q_1 \geq \ldots \geq q_r \geq i < j < k$, define 
\begin{align}
\gamma_{\bq, ijk} = - e_{q_{1}}...e_{q_{r}} \otimes e_{i} \wedge e_{j} \wedge e_{k} \in (\poly{\Gamma}_{3}(V))_{r+3}, \qquad \omega_{\bq,ijk} = \frac{z^{\alpha}}{\alpha!} dz_{j} \wedge dz_{k} \in (\poly{\Omega}^{2})_{r+3},
\end{align}
where $\balpha = (\alpha_1, \ldots, \alpha_n)$, and $\alpha_s$ denotes the number of times the index $s$ appears in $(\bq, i)$. 

\begin{lemma} \label{lem:closed_current}
    There is a natural isomorphism 
\begin{align}
    \poly{\Gamma}^{\cl}_2 \cong (\poly{\Omega}^{3, \cl})^{\vee}
\end{align}
where, on the right-hand side, we are taking the graded dual. The closed currents $\partial_3(\gamma_{\bq, ijk})$ and closed forms $d^2(\omega_{\bq,ijk})$, for $q_1 \geq \ldots \geq q_r \geq i < j < k$, give dual bases in weight $r+3$. 
\end{lemma}
\begin{proof}
    The space of closed weight $r+3$ currents $(\poly{\Gamma}_2^{\cl}(V))_{r+3}$ is the irreducible representation of $\GL(V)$ corresponding to the hook Young diagram of size $(r+1, 1, 1)$. It follows that the currents $\partial_3(\gamma_{\bq, ijk})$ give a basis. Taking the pairing in~\eqref{eq:closed_pairing}, we have 
    \begin{align}
        \langle \partial_3(\gamma_{\bq, ijk}), d^2(\omega_{\bp,abc}) \rangle_{\cl} &= \langle \gamma_{\bq, ijk}, d^2(\omega_{\bp,abc}) \rangle_{3} \\
        &= \frac{-\alpha_{a}}{\alpha !} \langle e_{q_{1}}...e_{q_{r}} \otimes e_{i} \wedge e_{j} \wedge e_{k}, z_{p_1} ... z_{p_{r}} dz_{a} \wedge dz_{b} \wedge dz_{c} \rangle_{3} \\ 
        &= \delta_{(\bq, ijk),(\bp,abc)}.
    \end{align}
\end{proof}

Finally, we conclude this section with the following relationship between the free crossed module from~\eqref{eq:free_cmla_vector_space} and closed $2$-currents; see~\cite[Appendix D]{chevyrev_multiplicative_2024-1} for further exposition.

\begin{theorem}{\cite{kapranov_membranes_2015}} \label{thm:rho_isomorphism}
    The symmetrization map $\rho' : T(V) \otimes \Lambda^2 V \to S(V) \otimes \Lambda^2 V = \poly{\Gamma}_2(V)$ is a $\GL(V)$-equivariant map which descends to define a Lie algebra map 
    \begin{align}
        \rho : \fk_1(V) \to \poly{\Gamma}_2(V)
    \end{align}
    which identifies the abelianization 
    \begin{align}
        \frac{\fk_1(V)}{[\fk_1(V), \fk_1(V)]} \cong \poly{\Gamma}_2(V). 
    \end{align}
     Furthermore, the restriction of $\rho$ to $\fa_1(V) \coloneqq \ker(\delta: \fk_1(V) \to \fk_0(V))$ defines an isomorphism of $\GL(V)$-representations
    \begin{align} 
        \rho:  \fa_1(V) \to \poly{\Gamma}^{\cl}_2(V). 
    \end{align}
\end{theorem}

Given indices $\bq = (q_1, \ldots, q_r)$, and $i, j, k$, such that $q_1 \geq \ldots \geq q_r \geq i < j < k$, define 
\begin{align}
    B_{\bq,ijk} \coloneqq [e_{q_1}, \ldots, [e_{q_r}, B_{ijk}]\ldots] \in \fa_1(V).
\end{align}
Under the isomorphism $\rho$ from Theorem \ref{thm:rho_isomorphism}, we see that 
\begin{align}
\rho(B_{\bq,ijk}) = \partial_3(\gamma_{\bq, ijk}).
\end{align}
Hence, we immediately conclude that $B_{\bq,ijk}$ define a basis of $\fa_1(V)$.

\subsection{Computing the Abelianized Curvature} 
We will use the relationship between $\fa_1$ and closed polynomial currents, along with their representation-theoretic properties, to explicitly compute the abelianized curvature $\Curv^{\ab}$. As we will be using Schur's lemma for complex representations, it is convenient to immediately consider the complexification $V_\C \coloneqq V \otimes_\R \C$. Because $\rho$ is defined over $\R$, this will not affect our final calculation of $\Curv^{\ab}$. 

Denote the complexified space of currents and forms by $(\poly{\Gamma}_{\bullet})_{\C}$ and $\poly{\Omega}^\bullet_\C$, respectively. Because the isomorphisms from ~\Cref{lem:closed_current} and ~\Cref{thm:rho_isomorphism} are $\GL(V)$-equivariant, they preserve the weight grading, and therefore extend to the completions. Hence, we obtain an isomorphism of $\GL(V_\C)$-representations $\rho: \com{\fa}_1(V_\C) \to (\poly{\Omega}^{3,\cl}(V_\C))^*$, which can be extended to 
\begin{align}
    \hrho = \rho \otimes \id : \com{\fa}_1(V_\C) \,\com{\otimes}\, \com{\Omega}^{3,\cl}_\C \to (\poly{\Omega}^{3,\cl}_\C)^* \,\com{\otimes}\, \com{\Omega}^{3,\cl}_\C.
\end{align}

Under this isomorphism, the curvature from~\eqref{eq:kapranov_2_curvature}, which we denote by $\Curv = \Curv^{\conk, \Conk}$, is sent to
\begin{align}
    \hrho(\Curv) = \sum_{i < j < k} \rho(B_{i,j,k}) dz_i \wedge dz_j \wedge dz_k \in (\poly{\Omega}^{3,\cl}_\C)^* \,\com{\otimes}\, \com{\Omega}^{3,\cl}_\C.
\end{align}
By Lemma \ref{lem:closed_current}, the closed currents $\rho(B_{ijk}) = \partial_{3}(\gamma_{ijk})$ and closed $3$-forms $dz_{i} \wedge dz_{j} \wedge dz_{k} = d\omega_{ijk}$ form dual bases. Hence, $\hrho(\Curv)$ is identified with the identity map for the weight $3$ closed forms
\begin{align}
    \hrho(\Curv) = \id_{(\poly{\Omega}^{3,\cl}_\C)_3} \in \End_{\GL(V_\C)}((\poly{\Omega}^{3,\cl}_\C)_3) \subset (\poly{\Omega}^{3,\cl}_\C)^* \,\com{\otimes}\, \com{\Omega}^{3,\cl}_\C.
\end{align}
Here, $\End_{\GL(V_\C)}(W)$ denotes the space of equivariant morphisms of $W$, a $\GL(V_\C)$-representation. 
Recalling that this is precisely the $\GL(V_\C)$-invariant subspace of the representation $\End(W)$, we conclude that $\hrho(\Curv)$ (and hence $\Curv$) is $\GL(V_\C)$-invariant. \medskip

\begin{lemma} \label{curvaturesubspace}
    The abelianized curvature $\Curv^{\ab}$ is $\GL(V_\C)$-invariant. Therefore 
    \begin{align}
    \hrho(K^{\ab}) \in \prod_{r=3}^\infty \End_{\GL(V_{\C})}\left((\poly{\Omega}^{3,\cl}_\C)_r\right).
\end{align}
\end{lemma}
\begin{proof}
    We recall from Equation~\eqref{eq:abelianized_2curv} that 
    \begin{align}
        \Curv^{\ab} = \sum_{m=0}^\infty \frac{1}{m!} \ad_{\eta}^m(\Curv).
    \end{align}
    Consider $\ad$ as the action of $S(V^*) \otimes  \fk_0 \subset C^\infty(V, \fk_0)$ on $\fa_1(V_\C) \otimes \Omega^{3}_\C$. By ~\Cref{cor:free_cm_equivariant}, we note that $\GL(V_{\C})$ acts on $\cmk$ in a way which preserves the crossed module structure. Therefore, $\GL(V_{\C})$ acts on both $S(V^*) \otimes  \fk_0$ and $\fa_1(V_\C) \otimes \Omega^{3}_\C$, and the action $\ad$ is $\GL(V_{\C})$-equivariant: 
    \begin{align}
    g \gtd \ad_{s}(\omega) = \ad_{g \gtd s} (g \gtd \omega),
    \end{align}
    where $g \in \GL(V_{\C})$, $s \in S(V^*) \otimes  \fk_0$ and $\omega \in \fa_1(V_\C) \otimes \Omega^{3}_\C$. The element $\eta = \sum_{i = 1} ^{d} v_{i} e_{i}$ is $\GL(V_{\C})$-invariant, since it corresponds to the identity $\id_{V_{\C}} \in V^* \otimes V \subset S(V^*) \otimes  \fk_0$. Because $\Curv$ is also $\GL(V_{\C})$-invariant, the same is true for $\ad_{\eta}^m(\Curv)$ and thus for $\Curv^{\ab}$ as well. In other words,
    \begin{align}
        \Curv^{\ab} \in \prod_{r=3}^\infty \End_{\GL(V_{\C})}\left((\poly{\Omega}^{3,\cl}_\C)_r\right).
    \end{align}
\end{proof}

The upshot of Lemma \ref{curvaturesubspace} is that the abelianized curvature can be decomposed as 
\begin{align}
\hrho(K^{\ab}) = \sum_{r = 3}^{\infty} \hrho(K^{\ab})_{r},
\end{align}
where each $\hrho(K^{\ab})_{r} \in \End_{\GL(V_{\C})}\left((\poly{\Omega}^{3,\cl}_\C)_r\right)$. By~\cite[Section 8.2, Theorem 2]{fulton_young_1996}, each representation $(\poly{\Omega}^{3,\cl}_\C)_r$ is irreducible. Therefore, by Schur's lemma, 
\begin{align}
\hrho(K^{\ab})_{r} = \lambda_{r} \id_{(\poly{\Omega}^{3,\cl}_\C)_r},
\end{align}
for a constant $\lambda_{r} \in \C$. In the following theorem, we verify that this constant is $\lambda_{r} = 1$.

\begin{theorem} \label{thm:abelianized_curvature}
    The abelianized curvature is
    \begin{align}
        \hrho(\Curv^{\ab}) = \id \in \End(\poly{\Omega}^{3,\cl}).
    \end{align}
\end{theorem}
\begin{proof}
    The element $\hrho(K^{\ab})_{m}$ has weight $m$ in the form component. Hence, it is given by 
    \begin{align}
    \hrho\left(\frac{1}{m!} \ad_{\eta}^m(\Curv)\right) = \frac{1}{m!} \hrho \ad_{\eta}^m\left(\sum_{1 \leq i < j < k \leq n} B_{i,j,k} \, dz_i \wedge dz_j \wedge dz_k \right).
    \end{align}
    To determine the constant $\lambda_{m}$, it suffices to compute the coefficient of 
    \begin{align}
        d \omega_{1, ..., 1, 2, 3} = \frac{z_{1}^m}{m!} dz_{1} \wedge dz_{2} \wedge dz_{3}.
    \end{align}
    This is easily seen to be 
    \begin{align}
        \rho(\ad_{e_1}^{m}(B_{123}))= \rho(B_{1,...,1,2,3}) = \partial_{3}(\gamma_{1,...,1,2,3}).
    \end{align}
    By Lemma \ref{lem:closed_current}, this is the closed current dual to $d \omega_{1, ..., 1, 2, 3}$. Hence, $\lambda_{m} = 1$. 
\end{proof}

\begin{corollary} \label{cor:abelianized_curavture_explicit}
    The abelianized curvature is
    \begin{align}
        \Curv^{\ab} = \sum_{\bq\geq i < j < k} B_{\bq, ijk}  \otimes d(\omega_{\bq, ijk}) \in \com{\fa}_1(V) \com{\otimes} \com{\Omega}^{3,\cl}.
    \end{align}
\end{corollary}

\subsection{Abelianized Surface Signature}
Now, we use the surface holonomy with respect to the abelianized 2-connection to compute the signature of closed surfaces,
\begin{align}
    C^{1}_{\cl}([0,1]^2, V) \coloneqq \left\{ \bX \in C^1([0,1]^2, V) \, : \, \partial \bX = 0\right\}.
\end{align}
To this end, we make use of the relatively simple expression for the abelianized $2$-curvature from \Cref{cor:abelianized_curavture_explicit} in order to simplify this calculation. The following lemma tells us that it suffices to integrate this $2$-curvature over a volume.

\begin{lemma} \label{lem:abelian_sh_from_2curv}
    Let $\bX \in C^1_{\cl}([0,1]^2, V)$ and let $\cX \in C^1([0,1]^3, V)$ be a volume such that all boundaries are sent to $0$ except the top, which is equal to $\bX$,
    \begin{align}
        \cX_{0,t,u} = \cX_{1,t,u} = \cX_{s,0,u} = \cX_{s,1,u} = \cX_{s,t,0} = 0 \andd \cX_{s,t,1} = \bX_{s,t}.
    \end{align}
    Let $(0, \Con^{\ab})$ be an abelian 2-connection valued in $\cmg$, with 2-curvature $\Curv^{\ab}$. Then, the surface holonomy is given by 
    \begin{align} \label{eq:abelian_sh}
        \Hol^{0, \Con^{ab}}(\bX) = \int_{\cX} \Curv^{\ab} = \int_{[0,1]^3} \Curv^{\ab}\left(\frac{\partial \cX_{s,t,u}}{\partial s}, \frac{\partial \cX_{s,t,u}}{\partial t}, \frac{\partial \cX_{s,t,u}}{\partial u}\right)  \, ds \, dt \, du.
    \end{align}
\end{lemma}
\begin{proof}
    This is a direct consequence of~\cite[Theorem 2.30]{martins_surface_2011}. 
\end{proof}

Next, the following result shows how the gauge transformation acts on the surface holonomy. 

\begin{proposition}{\cite[Corollary 4.9]{martins_surface_2011}} \label{prop:sh_under_gtrans}
    Let $\bX \in C^1_{\cl}([0,1]^2, V)$ be a smooth closed surface based at the origin, let $(\con, \Con)$ be a $2$-connection and let $(\gtrans, \Gtrans)$ be a gauge transformation. Then,
    \begin{align}
        \Hol^{(\gtrans, \Gtrans) \cdot (\con, \Con)}(\bX) = (\gtrans(0))^{-1} \gt \Hol^{\con, \Con}(\bX).
    \end{align}
\end{proposition}

Putting this all together, we obtain the following result. 
\begin{theorem} \label{thm:ssig_equiv_to_abelianization}
    For any smooth closed surface $\bX \in C^1_{\cl}([0,1]^2, V)$ based at the origin, the surface signature is given by
    \begin{align}
        \Sig(\bX) = \sum_{\bq \geq i < j < k} B_{\bq, ijk} \int_{\bX} \omega_{\bq, ijk} \in \com{\fa}_1(V).
    \end{align}
    Furthermore, using $\rho$ to embed $\com{\fa}_1(V)$ into the space of formal $2$-currents $\com{\Gamma}_2(V)$, the surface signature is given by the following expression 
    \[
        \Sig(\bX) = \sum_{\alpha \in \N^n, \ i < j} \frac{1}{\alpha !} e^{\alpha} \otimes e_{i} \wedge e_{j} \int_{\bX} z^{\alpha} dz_{i} \wedge dz_{j} \in \com{\Gamma}_2(V). 
    \]
\end{theorem}
\begin{proof}
    Let $(\conk, \Conk)$ denote the universal $2$-connection and recall that its surface holonomy is denoted $\Sig(\bX) = \Hol^{\conk, \Conk}(\bX)$. Let $(\gtrans, \Gtrans)$ denote the abelianization gauge transformation from~\eqref{eq:kap_gtrans} and~\eqref{eq:kap_Gtrans} so that $\gtrans(0) = 1$. Given the surface $\bX$, let $\cX$ be a volume as in \Cref{lem:abelian_sh_from_2curv}. Then 
    \begin{align}
        \Sig(\bX) &= \Hol^{0, \Conk^{\ab}}(\bX) =  \int_{[0,1]^3} \Curv^{\ab}\left(\frac{\partial \cX_{s,t,u}}{\partial s}, \frac{\partial \cX_{s,t,u}}{\partial t}, \frac{\partial \cX_{s,t,u}}{\partial u}\right)  \, ds \, dt \, du  \\
        &= \sum_{\bq\geq i < j < k} B_{\bq, ijk}  \int_{\cX} d(\omega_{\bq, ijk}) = \sum_{\bq\geq i < j < k} B_{\bq, ijk}  \int_{\bX} \omega_{\bq, ijk},
    \end{align}
    where the first equality follows from ~\Cref{prop:sh_under_gtrans} and the fact that $\theta(0) = 1$, the second equality follows from \Cref{lem:abelian_sh_from_2curv}, the third equality follows from ~\Cref{cor:abelianized_curavture_explicit}, and the fourth equality follows from Stokes' theorem. Note that this expression for the signature has the form $(\mathrm{id} \otimes \int_{\bX})( \sum_{l} b^l \otimes b_{l})$, for $(b^{l}, b_{l})$ a dual pair of bases of $\com{\fa}_1(V)$ and $\com{\Omega}^{3,\cl}$. Because $\bX$ is closed, we have $\int_{\bX} d\mu = 0$ for all $\mu \in \Omega^1$ by Stokes' theorem. Therefore, using $\rho$ to embed $\com{\fa}_1(V)$ into $\com{\Gamma}_2(V)$, our expression for $\Sig(\bX)$ also has the same form where now $(b^l, b_{l})$ are a dual pair of bases of $\com{\Gamma}_2(V)$ and $\com{\Omega}^{2}$. Our final expression for the signature then follows by using the dual bases from Equation \ref{eq:dualbasisfactorial}. 
\end{proof}

\begin{corollary} \label{cor:trivial_ssig_current}
    Let $\bX \in C^1_{\cl}([0,1]^2, V)$. Then, $\Sig(\bX) = 0$ if and only if
    \begin{align}
        \int_\bX \omega = 0
    \end{align}
    for all compactly supported 2-forms $\omega \in \Omega^2_c(V)$. 
\end{corollary}
\begin{proof}
    By ~\Cref{thm:ssig_equiv_to_abelianization}, the condition that $\Sig(\bX) = 0$ is equivalent to $\int_{\bX} \omega = 0$ for all polynomial $2$-forms. Since polynomial 2-forms $\poly{\Omega}^2$ are dense in the compactly supported 2-forms $\Omega_c^2$, we obtain the desired result.
\end{proof}

\begin{remark} \label{rem:analyticcond}Let $\bx \in C^1([0,1], V)$ be a thinly null-homotopic path. The existence of a height function $h: [0,1] \to \R$ from the analytic condition \ref{A1} is used to detect cancellations: whenever $s < t$ with $h(s) = h(t) = \inf_{s \leq u \leq t} h(u)$, the restriction $\bx|_{[s,t]}$ is also thinly null homotopic. Due to the fact that cancellations in 2-dimensional thin homotopy can be non-local, a direct analogue of the height function does not exist. Instead, the condition in~\Cref{cor:trivial_ssig_current} uses the integration of compactly supported forms to detect both local and non-local cancellations which occur on a surface $\bX \in C^1_{\cl}([0,1]^2, V)$. Thus, we interpret this condition to be the generalization of \ref{A1}.
\end{remark}

\section{Piecewise Linear Surface Signature and Decompositions} \label{sec:PL}

In this section, we algebraically construct a crossed module of piecewise linear surfaces, extending $\PL_0(V)$ from~\Cref{ssec:pl_paths}. This leads to an algebraic definition of the surface signature, and furthermore, a decomposition of the signature into abelian and boundary components.

\subsection{Free Crossed Modules of Groups}
We will begin by discussing free crossed modules of groups. This will be used in~\Cref{sec:PL_thin_homotopy}, but it also serves as motivation for the construction of the piecewise linear crossed module in the following section.
Our aim is to show a global universal property, analogous to that for Lie algebras in~\Cref{thm:free_xlie}. \medskip%

Given a group $G$, let $\XGrp(G)$ be the subcategory of crossed modules
\begin{align}
    \cmG = (\delta: H \to G, \gt)
\end{align}
where $G$ is fixed. Let $\slice{\Set}{G}$ denote the category of set functions $\rho: L \to G$, where $L$ is a set and $G$ is the fixed group. There is a natural forgetful functor, along with a free functor as a left adjoint,
\begin{align}
    \For: \XGrp(G) \to \slice{\Set}{G} \andd \Fr: \slice{\Set}{G} \to \XGrp(G). 
\end{align}

The free crossed module generated by $\rho: L \to G$ can be constructed as follows~\cite[Proposition 3.4.3]{brown_nonabelian_2011}. Let $\tFr(\rho)$ be the free group generated by $G \times L$. It is equipped with a left $G$-action by automorphisms which is defined on the generators by $g' \gt (g, \lambda) = (g'g, \lambda)$. Define a homomorphism $\delta: \tFr(\rho) \to G$ by the following map on generators
\begin{align}
    \delta(g, \lambda) = g \cdot \rho(\lambda) \cdot g^{-1}.
\end{align}
This is the \emph{free pre-crossed module on $\rho$}.
Then, the free crossed module $\Fr(\rho) \coloneqq \tFr(\rho)/\sim_{\Pf}$ is the quotient by the Peiffer identity,
\begin{align}
    \delta(g, \lambda) \gt (g',\lambda') \sim_{\Pf} (g,\lambda) \cdot (g',\lambda') \cdot (g,\lambda)^{-1}.
\end{align}
The unit of the adjunction provides a map $\eta_\rho: L \to \Fr_1(\rho)$ given by sending $\lambda \in L$ to the equivalence class of $(1,\lambda)$. \medskip

Next, let $\SG$ \label{pg:SG} denote the comma category $(\id \downarrow \For)$ associated to the functors $\id: \Set \to \Set$ and $\For: \Grp \to \Set$. An object of $\SG$ is given by the data of a set $L$, a group $G$, and a map of sets $\rho: L \to G$. A morphism $f = (f_{L}, f_{G}) : (\rho_{1}: L_{1} \to G_{1}) \to (\rho_{2} : L_{2} \to G_{2})$ consists of a set map $f_L: L_1 \to L_2$ and a group homomorphism $f_G : G_1 \to G_2$ such that the following diagram commutes 
\begin{equation}
    \begin{tikzcd}
        L_1 \ar[r, "\rho_1"] \ar[d,swap, "f_L"] &  G_1\ar[d, "f_G"]  \\
        L_2 \ar[r, "\rho_2"]  & G_2.
    \end{tikzcd}
\end{equation}
There exists a natural forgetful functor $\For: \XGrp \to \SG$. Free crossed modules of groups satisfy the following universal property. 

\begin{theorem} \label{thm:free_xgrp}
    Let $(\rho: L \to H) \in \SG$, $\cmG = (\delta: G_1 \to G_0, \gt) \in \XGrp$, and 
    \begin{align}
    f = (f_1, f_0) : (\rho: L \to H) \to \For(\delta: G_1 \to G_0, \gt).
    \end{align}
    Then there is a unique morphism of crossed modules 
    \begin{align}
    F = (F_1, F_0): \Fr(\rho) \to \cmG,
    \end{align}
    such that $F_1 \circ \eta_\rho = f_1$ and $F_0 = f_0$.
\end{theorem}
\begin{proof}
    Given the group homomorphism $f_0: H \to G_0$, we consider the pullback crossed module~\cite[Definition 5.1.1]{brown_nonabelian_2011}
    \begin{align}
        f_0^* \cmG \coloneqq (\delta: f_0^* G_1 \to H, \gt), \quad f_0^* G_1  \coloneqq \{ (h,g) \in H \times G_1 \, : \, f_0(h) = \delta(g)\},
    \end{align}
    and define the map $\tf_1 : f_0^* G_1 \to G_1$ given by $\tf_1(h,g) = g$, which gives a morphism of crossed modules 
    \begin{align}
        \tf = (\tf_1, f_0): f_0^* \cmG \to \cmG.
    \end{align}
    We define a map $u_1(f): L \to f_0^* G_1$ by $u_1(f)(\lambda) = (\rho(\lambda), f_1(\lambda))$. This defines a map
    \begin{align}
        u(f) = (u_1(f), \id_H) : (\rho: L \to H) \to \For(f_0^* \cmG)
    \end{align}
    in $\slice{\Set}{H}$ such that $\For(\tf) \circ u(f) = f$. Then, using the universal property of $\Fr(\rho)$ in $\XGrp(H)$, there is a unique map
    \begin{align}
        g: \Fr(\rho) \to f_0^*\cmG
    \end{align}
    which satisfies $\For(g) \circ \eta_\rho = u(f)$ in $\slice{\Set}{H}$. We define
    \begin{align}
        F = \Fr(\rho) \xrightarrow{g} f_0^* \cmG \xrightarrow{\tf} \cmG,
    \end{align}
    which is a morphism of crossed modules and it satisfies
    \begin{align}
        \For(F) \circ \eta_\rho = \For(\tf) \circ \For(g) \circ \eta_\rho = \For(\tf) \circ u(f) = f,
    \end{align}
    so $F$ is the desired morphism. The map $F$ is unique because the map $g$ coincides with the unique factorization of the morphism $F$ through the pullback crossed module~\cite[Theorem 5.1.2]{brown_nonabelian_2011}.
\end{proof}

\subsection{Crossed Module of Piecewise Linear Surfaces}\label{ssec:pl_surfaces}
Here, we will generalize the construction in~\Cref{ssec:pl_paths} and define a linear algebraic version of the free crossed module construction. Consider $\PL_0(V)$, and define the group homomorphism
\begin{align}
    t: \PL_0(V) \to V, \quad (v_1, \ldots, v_k) \mapsto \sum_{i=1}^k v_i
\end{align}
which defines the \emph{endpoint} of a PL path. We define the \emph{group of piecewise linear loops} by $\PL_0^{\cl}(V) \coloneqq \ker(t)$. Furthermore, we define the set of \emph{planar loops} by \label{pg:planarloop}
\begin{align}
    \planarloop(V) \coloneqq \{ \bx\in \PL_0^{\cl}(V) \, : \, \dim( \SPAN(\bx)) \leq 2\}, 
\end{align}
where we use the notion of span from~\eqref{eq:span_def} using the minimal representative of $\bx$.
We note that the span of any nontrivial planar PL loop must be two-dimensional. There is a natural inclusion
\begin{align}
    \iota: \planarloop(V) \hookrightarrow \PL_0(V). 
\end{align}

\begin{definition} \label{def:kite}
    A \emph{kite} is a pair $(\bw, \bb)$, consisting of a \emph{tail path} $\bw \in \PL_0(V)$ and a planar loop $\bb \in \planarloop(V)$. We define the set of kites by
\begin{align} \label{eq:kites}
    \Kite(V) \coloneqq \PL_0(V) \times \planarloop(V).
\end{align}
\end{definition} 
The \emph{pre-crossed module of piecewise linear surfaces} is defined to be
\begin{align}
    \tPL(V) \coloneqq (\delta: \tPL_1(V) \to \PL_0(V), \gt) \quad \text{where} \quad \tPL_1(V) \coloneqq \FG\big(\Kite(V)\big)/\sim
\end{align} 
is the quotient of the free group generated by kites subject to the following relations:
\begin{enumerate}[label=(\textbf{PL1.\arabic*})]
    \item \label{PL1.1} $(\bw_1, \bb_1) \concat (\bw_2, \bb_2) \sim (\bw_1, \bb_1 \concat \bu \concat \bb_2 \concat \bu^{-1})$, where $\bu = \bw_1^{-1} \concat \bw_2$, if 
    \begin{align}\bb_1 \concat \bu \concat \bb_2 \concat \bu^{-1} \in \planarloop(V);\end{align}
    \begin{figure}[!h]
    \includegraphics[width=\linewidth]{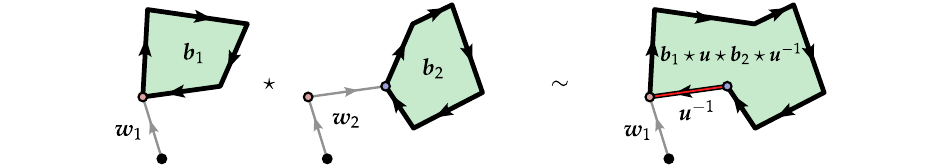}
    \end{figure}

    \item \label{PL1.2} $(\bw, \emptyset_0) \sim \emptyset_1$, where $\emptyset_0,\emptyset_1$ denote the empty words in $\PL_0(V)$ and $\FG\big(\Kite(V)\big)$ respectively; and
    \item \label{PL1.3} $(\bw, \bx \concat \bb \concat \bx^{-1}) \sim (\bw \concat \bx, \bb)$ for any $ \bx \in \PL_0(V)$ such that $\bx \concat \bb \concat \bx^{-1} \in \planarloop(V)$. 
\end{enumerate}

\begin{remark}
    Here, we do not prescribe the planes on which the kites and their compositions are defined. However, assuming that the loops are non-trivial, in~\Cref{cor:composition_of_kites}, we show that if \ref{PL1.1} holds, then $\SPAN(\bu) \subseteq \SPAN(\bb_1) = \SPAN(\bb_2)$. Furthermore, in~\Cref{lem:tail_of_planar_kite_is_planar}, we show that if \ref{PL1.3} holds, then $\SPAN(\bx) \subseteq \SPAN(\bb)$.
\end{remark}

With these relations, we note that $(\bw, \bb)^{-1} = (\bw, \bb^{-1})$, since
\begin{align}
    (\bw, \bb) \concat (\bw, \bb^{-1}) \sim (\bw, \bb \concat \bb^{-1}) = (\bw, \emptyset_0) \sim \emptyset_1,
\end{align}
where here, $\bu = \bw^{-1} \concat \bw = \emptyset_0$. 

\begin{remark}
    As with the 1-dimensional case, we can represent elements of $\tPL_1(V)$ as words in $\Kite(V)$, without considering the formal inverses. In particular, we can make the equivalent definition of $\tPL_1(V)$ in terms of the free monoid, which we will use interchangeably,
    \begin{align} \label{eq:tpl_free_monoid}
        \tPL_1(V) = \FMon(\Kite(V))/\sim.    
    \end{align}
\end{remark}
Next, we define the boundary map $\delta: \tPL_1(V) \to \PL_0(V)$ on the generators by
\begin{align}
    \delta(\bw, \bb) \coloneqq \bw \concat \bb \concat \bw^{-1},
\end{align}
and the action of $\PL_0(V)$ on the generators of $\tPL_1(V)$ by
\begin{align}
    \bx \gt (\bw, \bb) \coloneqq (\bx \concat \bw, \bb). 
\end{align}

\begin{lemma}
    The structure $\widetilde{\cmPL}(V) = (\delta: \tPL_1(V) \to \PL_0(V), \gt)$ is a pre-crossed module of groups.
\end{lemma}
In order to obtain a crossed module, we quotient by the standard Peiffer identity
\begin{align} \label{eq:PL_peiffer}
    \delta(\bw_1, \bb_1) \gt (\bw_2, \bb_2) \sim_{\Pf} (\bw_1, \bb_1) \concat (\bw_2, \bb_2) \concat (\bw_1, \bb_1^{-1}).
\end{align}

\begin{definition} \label{def:cmpl}
    The \emph{piecewise linear crossed module} is defined by
    \begin{align}
        \cmPL(V) \coloneqq (\delta: \PL_1(V) \to \PL_0(V), \gt),  \quad \text{where} \quad \PL_1(V) \coloneqq \tPL_1(V)/\sim_{\Pf}. 
    \end{align}
\end{definition}

There is a natural inclusion $\xi_{V,1}: \planarloop(V) \to \PL_1(V)$ defined by 
\begin{align} \label{eq:xi_V}
    \xi_{V,1}(\bb) \coloneqq (\emptyset_0, \bb). 
\end{align}
In fact, for 2-dimensional vector spaces, this is an isomorphism. 

\begin{lemma} \label{lem:dim2_planar_loop_iso}
    If $U \in \Vect$ is 2-dimensional, then $\planarloop(U) \cong \tPL_1(U) \cong \PL_1(U)$ are isomorphic as groups.
\end{lemma}
\begin{proof}
    Since $U$ is two-dimensional, all loops are planar, and so $\planarloop(U)= \PL_0^{\cl}(U)$, which is a group. The map $\txi_{U,1}: \planarloop(U) \to \tPL_1(U)$ is a group homomorphism by~\ref{PL1.1}. In fact, it is an isomorphism with inverse given by the boundary map $\delta : \tPL_1(U) \to \PL_0^{\cl}(U)$. Indeed, the identity $\delta \circ \txi_{U,1} = \id$ is immediate, and $\txi_{U,1} \circ \delta = \id$ follows by \ref{PL1.3}. Furthermore, the Peiffer identity~\eqref{eq:PL_peiffer} is already implied because it holds in $\PL_0^{\cl}(U)$.
\end{proof}

For $V \in \Vect$ where $\dim(V) \geq 3$, $\planarloop(V)$ is not a group. However, by~\Cref{lem:dim2_planar_loop_iso}, the restriction to a 2-dimensional subspace $U \subset V$ is a group $\planarloop(U) \cong \PL_1(U)$. For the remainder of the article, we will make the identification of crossed modules
\begin{align}
    (\iota: \planarloop(U) \to \PL_0(U), \gt) \cong \cmPL(U) = (\delta: \PL_1(U) \to \PL_0(U), \gt).
\end{align}
Furthermore, the restriction of the inclusion $\xi_{V,1}$ in~\eqref{eq:xi_V} yields a morphism of crossed modules
\begin{align} \label{eq:xi_inclusion}
    \xi^U_V = (\xi^U_{V,1}, \xi^U_{V,0}): \cmPL(U) \to \cmPL(V),
\end{align}
where $\xi^U_{V,0}: \PL_0(U) \to \PL_0(V)$ is the inclusion induced by $U \hookrightarrow V$.
This leads to the following universal property.

\begin{theorem} \label{thm:univ_prop_plcm_local}
    Let $V$ be a vector space, and suppose
    \begin{align}
        \cmG = (\delta: G \to \PL_0(V), \gt)
    \end{align}
    is a crossed module of groups. Let $f: \planarloop(V) \to G$ be a function such that $f(\emptyset_0) = 1$ and for any 2-dimensional subspace $U \subset V$, the restriction of $f$ to $U$ yields a morphism of crossed modules, 
    \begin{align}
        (f|_U, \xi^U_{V,0}): \cmPL(U) \to \cmG,
    \end{align}
    where $\xi^U_{V,0}: \PL_0(U) \hookrightarrow \PL_0(V)$ is the inclusion from~\eqref{eq:xi_inclusion}.
    Then, there exists a unique group homomorphism $F: \PL_1(V) \to G$ such that
    \begin{align}
        (F, \id): \cmPL(V) \to \cmG
    \end{align}
    is a morphism of crossed modules and $F \circ \xi_{V,1} = f$. 
\end{theorem}
\begin{remark}
    Note that if $\dim(V) \leq 2$ then $F = f$ and there is nothing to prove. 
\end{remark}
\begin{proof}
    First, we define a map $\tF: \FG(\Kite(V)) \to G$ on the generators by
    \begin{align}
        \tF(\bw, \bb) = \bw \gt f(\bb).
    \end{align}
    We verify that this descends to $\tPL_1(V)$ by checking the relations \ref{PL1.1}-\ref{PL1.3}. For \ref{PL1.2}, we have
    \begin{align}
        \tF(\bw, \emptyset_0) = \bw \gt f(\emptyset_0) = 1
    \end{align}
    by assumption. %
    Next, for \ref{PL1.3}, suppose $\bw \in \PL_0(V)$, $\bx \in \PL_0(V)$ and $\bb \in \planarloop(U)$ is a nontrivial planar loop with span $U$ such that $\bx \concat \bb \concat \bx^{-1} \in \planarloop(V)$. By~\Cref{lem:tail_of_planar_kite_is_planar}, $\bx \in \PL_0(U)$. 
    Then, since $f$ restricts to a morphism of crossed modules on $U$, 
    \begin{align} \label{eq:univ_property_showing_PL1.1}
        \tF(\bw, \bx \concat \bb \concat \bx^{-1}) = \bw \gt f(\bx \gt \bb) = \bw \gt (\bx \gt f(\bb)) = (\bw \concat \bx) \gt f(\bb) =  \tF(\bw \concat \bx, \bb).
    \end{align}
    Finally, we verify \ref{PL1.1}. Suppose $(\bw_1, \bb_1), (\bw_2, \bb_2) \in \Kite(V)$ are nontrivial kites such that $\bb_1 \concat \bu \concat \bb_2 \concat \bu^{-1} \in \planarloop(V)$, where $\bu = \bw_1^{-1} \concat \bw_2$. By~\Cref{cor:composition_of_kites}, there exists $U \subset V$ with $\dim(U) = 2$ such that $\bb_1, \bb_2, \bu \in \PL_0(U)$. 
    Then, using $\bw_2 = \bw_1 \concat \bu$, equation ~\eqref{eq:univ_property_showing_PL1.1}, and the assumption that $f$ restricts to a homomorphism on $\PL_{0}^{\cl}(U)$, we have
    \begin{align}
        \tF(\bw_1, \bb_1) \cdot \tF(\bw_2, \bb_2)  &= \tF(\bw_1, \bb_1) \cdot \tF(\bw_1 \concat \bu, \bb_2)\\
        &= (\bw_1 \gt f(\bb_1)) \cdot (\bw_1 \gt f(\bu \concat \bb_2 \concat \bu^{-1}))\\
        & = \bw_1 \gt f(\bb_1 \concat \bu \concat \bb_2 \concat \bu^{-1}) \\
        & = \tF(\bw_1, \bb_1 \concat \bu \concat \bb_2 \concat \bu^{-1}).
    \end{align}
    Therefore, the map $\tF$ descends to a map $\tF: \tPL_1(V) \to G$, which is the unique morphism of pre-crossed modules which satisfies $\tF \circ \txi = f$, where $\txi: \planarloop(V) \to \tPL_1(V)$ is defined by $\txi(\bb) = (\emptyset_0, \bb)$. The final step is applying the crossed module functor by quotienting out the Peiffer identity to obtain the desired unique morphism $F: \PL_1(V) \to G$. 
\end{proof}

We can also enhance this universal property by varying the base.

\begin{corollary} \label{cor:cmPL_enhanced_univ_property}
    Let $V$ be a vector space and suppose $\cmG = (\delta: G_1 \to G_0, \gt)$ is a crossed module of groups. Let $f_0: \PL_0(V) \to G_0$ be a group homomorphism and let $f_1: \planarloop(V) \to G_1$ be a function such that  $f_1(\emptyset_0) = 1$ and for any 2-dimensional subspace $U \subset V$, the restriction of $(f_1, f_0)$ to $U$ is a morphism of crossed modules, 
    \begin{align}
        (f_{U,1}, f_{U,0}): \cmPL(U) \to \cmG.
    \end{align}
    Then, there exists a unique group homomorphism $F_1: \PL_1(V) \to G_1$ such that
    \begin{align}
        (F_1, f_0): \cmPL(V) \to \cmG
    \end{align}
    is a morphism of crossed modules and $F_1 \circ \xi_{V,1} = f_1$.
\end{corollary}
\begin{proof}
    We consider the pullback crossed module of $\cmG$ with respect to $f_0$, defined by
    \begin{align}
        f_0^*\cmG \coloneqq (\delta: f_0^* G_1 \to \PL_0(V), \gt), \quad f_0^*G_1 \coloneqq \{ (\bw, g) \in \PL_0(V) \times G_1 \, : \, f_0(\bw) = \delta(g)\}.
    \end{align}
    It is equipped with a morphism of crossed modules $p = (p_1, f_0) : f_0^*\cmG \to \cmG$. There is a well-defined function $\tf: \planarloop(V) \to f_0^*G_1$ given by sending a planar loop $\bb$ to $\tf(\bb) = (\iota(\bb), f_1(\bb))$. To see that this is well-defined, note first that $\tf(\emptyset_0) = (\emptyset_0, 1)$. Furthermore, if $\bb \in \planarloop(V)$ is non-trivial, it determines a $2$-dimensional subspace $U$ such that $\bb \in \PL_{0}^{\cl}(U)$. Then, since $(f_{1}, f_0)$ restricts to a morphism of crossed modules on $\cmPL(U)$, we see that $f_0 (\iota(\bb)) = \delta(f_1 (\bb))$. Note also that $p_1 \circ \tf = f_1$. 

    We wish to apply the universal property of ~\Cref{thm:univ_prop_plcm_local}, and for this we must verify that for every $2$-dimensional subspace $U$, the restriction $\tf|_{U}$ defines a morphism of crossed modules $\cmPL(U) \to f_0^*\cmG$. Now, because $(f_{U,1}, f_{U,0})$ is assumed to be a morphism of crossed modules, by~\cite[Theorem 5.1.2]{brown_nonabelian_2011} there exists a unique factorization through the pullback crossed module $f_0^*\cmG$ 
    \begin{align}
        (f_{U,1}, f_{U,0}): \cmPL(U) \xrightarrow{(\tf_{U,1}, \xi_{V,0}^{U})} f_0^*\cmG \xrightarrow{p} \cmG.
    \end{align}
    But a simple calculation shows that $\tf|_{U} = \tf_{U,1}$, thus verifying the condition. As a result, by the universal property, there is a unique homomorphism $\tF: \PL_{1}(V) \to f_0^* G_1$ such that 
    \begin{align}
        (\tF, \id): \cmPL(V) \to f_0^*\cmG
    \end{align}
    is a homomorphism of crossed modules and $\tF \circ \xi_{V,1} = \tf$. Finally, we define $F_1 = p_1 \circ \tF$ to obtain our desired morphism $(F_1, f_0)$.

\end{proof}

\begin{corollary}
    The piecewise linear crossed module defines a functor
    \begin{align}
        \cmPL: \Vect \to \XGrp.
    \end{align}
\end{corollary}
\begin{proof}
    Let $\phi: V \to W$ be a linear map. Then, by~\Cref{cor:PL_functor}, there exists a homomorphism $\PL_0(\phi): \PL_0(V) \to \PL_0(W)$. By ~\Cref{cor:span_pres}, $\PL_0(\phi)$ preserves the span of paths. Hence, if a path $\bx$ is planar, then so is $\PL_0(\phi)(\bx)$. As a result, we can also define a map $f_{\phi, 1}: \planarloop(V) \to \PL_1(W)$ by
    \begin{align}
        f_{\phi,1} : \planarloop(V) \xrightarrow{\PL_0(\phi)} \planarloop(W) \xrightarrow{\xi_{W,1}} \PL_1(W).
    \end{align}
      Given a 2-dimensional subspace $U \subset V$, the restriction of the pair $(f_{\phi, 1}, \PL_0(\phi))$ to $U$ is a morphism of crossed modules $\cmPL(U) \to \cmPL(W)$. Hence, by~\Cref{cor:cmPL_enhanced_univ_property}, there is a unique homomorphism $\PL_1(\phi): \PL_{1}(V) \to \PL_{1}(W)$ such that
    \begin{align}
        \cmPL(\phi) = (\PL_1(\phi), \PL_0(\phi)) : \cmPL(V) \to \cmPL(W)
    \end{align}
    is a morphism of crossed modules and such that $\PL_1(\phi) \circ \xi_{V,1} = \xi_{W,1} \circ \PL_0(\phi)|_{\planarloop(V)}$. The component $\PL_0(\phi)$ is functorial by ~\Cref{cor:PL_functor}, and the component $\PL_1(\phi)$ is functorial because of the uniqueness.  
\end{proof}

\begin{remark}
    Let $\phi: U \hookrightarrow V$ be the inclusion of a 2-dimensional subspace into $V$. Then the restriction of $\xi_{V,1}^U$ from~\eqref{eq:xi_inclusion} is $\PL_1(\phi)$.
\end{remark}

\subsection{Piecewise Linear Surface Signature}
In this section, we continue the generalization of~\Cref{ssec:pl_paths} by extending the realization and piecewise linear signature natural transformations to the piecewise linear crossed module using the universal property in~\Cref{cor:cmPL_enhanced_univ_property}.
In fact, we will show that for a certain class of \emph{planar} functors, there is a unique natural transformation which extends the path constructions. 

\begin{definition}
    A functor $\bsF = (\sF_1, \sF_0): \Vect \to \XGrp$ is \emph{planar} if
    \begin{itemize}
        \item $\bsF(U)$ is trivial when $\dim(U) = 0$,
        \item $\sF_1(U)$ is trivial when $\dim(U) = 1$, and
        \item $\delta^{\sF}_U : \sF_1(U) \to \sF_0(U)$ is injective when $\dim(U) = 2$. 
    \end{itemize}
\end{definition}

\begin{lemma}
    The piecewise linear crossed module $\cmPL: \Vect \to \XMod$ is a planar functor.
\end{lemma}
\begin{proof}
    The cases of $\dim(U) = 0,1$ are immediate, and the condition for $\dim(U) = 2$ is given by~\Cref{lem:dim2_planar_loop_iso}.
\end{proof}

The following lemma is the main result used in this section. 

\begin{lemma} \label{lem:unique_extension_of_nt}
    Let $\bsF = (\sF_1, \sF_0): \Vect \to \XGrp$ be a planar functor 
    and let $\alpha_0 : \PL_0 \Rightarrow F_0$ be a natural transformation with the property that $\alpha_{U,0}(\PL_0^{\cl}(U)) \subseteq \im(\delta^{\sF}_{U})$ when $\dim(U) = 2$.
    Then, there exists a unique natural transformation which extends $\alpha_0$,
    \begin{align}
        \alpha = (\alpha_1, \alpha_0) : \cmPL \Rightarrow \bsF.
    \end{align}
\end{lemma}
\begin{proof}
    Let $\Vect_{\leq 2}$ denote the subcategory of vector spaces of dimension at most $2$. We will first construct $\alpha$ on this subcategory. When $\dim(U) \leq 1$, the components $\alpha_{U,1}$ are uniquely determined because $\PL_1(U)$ is trivial. Now assume that $\dim(U) = 2.$ In this case, the component $\alpha_{U,1}$ must make the following diagram commute
    \begin{equation}\label{eq:alpha_dim2_commutative}
    \begin{tikzcd}
        \PL_1(U) \ar[r, dashed,"\alpha_{U,1}"] \ar[d,swap, "\delta^{\PL}_U"] &  \sF_1(U)\ar[d, "\delta^{\sF}_U"]  \\
        \PL_0(U) \ar[r, "\alpha_{U,0}"]  & \sF_0(U).
    \end{tikzcd}
    \end{equation}
The image of the map $\alpha_{U,0} \circ \delta^{\PL}_U$ is contained in the image of the injective map $\delta^{\sF}_{U}$. Hence, the desired map $\alpha_{U,1}$ exists and it is unique. The pair $\alpha_{U} = (\alpha_{U,1}, \alpha_{U,0})$ is automatically a map of crossed modules because $\delta^{\PL}_U$ and $\delta^{\sF}_{U}$ are injective.

To check naturality, let $\phi: U \to V$ be a linear map. We need to verify that $\sF_{1}(\phi) \circ \alpha_{U,1} = \alpha_{V,1} \circ \PL_{1}(\phi)$. Because $\cmPL$ and $\bsF$ are planar functors, this is trivially satisfied when either $U$ or $V$ has dimension less than $2$. Hence, we assume that $\dim(U) = \dim(V) = 2$. Because $\alpha_0: \PL_0 \Rightarrow \sF_0$ is a natural transformation, we have $F_0(\phi) \circ \alpha_{U,0} = \alpha_{V,0} \circ \PL_0(\phi)$. Then, using the fact that $\cmPL(\phi), \bsF(\phi), \alpha_{U}$ and $\alpha_{V}$ are morphisms of crossed modules, we have 
\begin{align}
        \delta^{\sF}_V \circ (\sF_1(\phi) \circ \alpha_{U,1}) = (\sF_0(\phi) \circ \alpha_{U,0}) \circ \delta^{\PL}_U = (\alpha_{V,0} \circ \PL_0(\phi)) \circ \delta^{\PL}_U = \delta^{\sF}_V \circ (\alpha_{V,1} \circ \PL_1(\phi)).
    \end{align}
Hence, by injectivity of $\delta^{\sF}_V$, we have $\sF_{1}(\phi) \circ \alpha_{U,1} = \alpha_{V,1} \circ \PL_{1}(\phi)$.

Next, we extend $\alpha$ to vector spaces $V$ of dimension at least $3$. We will do this by applying the universal property of~\Cref{cor:cmPL_enhanced_univ_property}. Given the inclusion $\iota_U: U \hookrightarrow V$ of a 2-dimensional subspace $U$, we define the following morphism of crossed modules
    \begin{align}
        s_U = (s_{U,1}, s_{U,0}) : \cmPL(U) \xrightarrow{\alpha_U} \bsF(U) \xrightarrow{\bsF(\iota_U) = \xi^U_V} \bsF(V).
    \end{align}
    We assemble the $s_{U,1}$ over all 2-dimensional subspaces $U \subset V$ to obtain a map $s_1 : \planarloop(V) \to \sF_1(V)$. Along with the homomorphism $\alpha_{V,0} : \PL_0(V) \to \sF_0(V)$, this map satisfies the hypotheses of~\Cref{cor:cmPL_enhanced_univ_property}. Hence we obtain a unique homomorphism $\alpha_{V,1} : \PL_1(V) \to \sF_1(V)$ such that 
    \begin{align}
        \alpha_V = (\alpha_{V,1}, \alpha_{V,0}): \cmPL(V) \to \bsF(V)
    \end{align}
    is a morphism of crossed modules and such that $\alpha_{V,1} \circ \xi_{V,1} = s_1$. In particular, $\alpha_V \circ \cmPL(\iota_U) = \bsF(\iota_U) \circ \alpha_U$, meaning that $\alpha$ is natural with respect to inclusions of $2$-dimensional subspaces. Since this property characterizes $\alpha_{V,1}$, we see that any natural transformation extending $\alpha_{0}$ must be unique. Furthermore, it is straightforward to see that $\alpha$ is natural with respect to linear maps $\phi: U \to V$ where $\dim(U) \leq 2$. Indeed, this is immediate if $\dim(U) \leq 1$ because then $\PL_1(U)$ is trivial. Assuming that $\dim(U) = 2$, factor $\phi$ as  
     \begin{align}
        \phi : U \xrightarrow{\psi} W \xrightarrow{\iota_W} V,
    \end{align}
    where $W = \im(\phi)$. Therefore, 
    \begin{align}
        \ \alpha_{V,1} \circ \PL_{1}(\iota_{W}) \circ \PL_{1}(\psi) = \sF_1(\iota_W) \circ \alpha_{W,1} \circ \PL_{1}(\psi) = \sF_1(\iota_W) \circ \sF_{1}(\psi) \circ \alpha_{U,1},
    \end{align}
    where the first equality follows from naturality with respect to inclusions and the second equality follows from naturality in $\Vect_{\leq 2}$. 

The final step in the proof is to verify naturality of $\alpha$ for arbitrary linear maps $\phi: V \to W$. We may assume that $\dim(V) \geq 3$. For each 2-dimensional subspace $U \subset V$, let $\iota_U: U \hookrightarrow V$ be the inclusion and consider the following diagram
    \begin{equation}
        \begin{tikzcd}
        \cmPL(U) \ar[r, "\alpha_U"] \ar[d,swap, "\cmPL(\iota_U) = \xi^U_V"] &  \bsF(U)\ar[d, "\bsF(\iota_U)"]  \\
        \cmPL(V) \ar[r, "\alpha_V"]  \ar[d,swap, "\cmPL(\phi)"] & \bsF(V) \ar[d, "\bsF(\phi)"]\\
        \cmPL(W) \ar[r, "\alpha_W"]& \bsF(W).
        \end{tikzcd}
    \end{equation}
    We know that the top and outer squares commute, and we wish to show that the bottom square commutes. For each $U \subset V$, consider the morphism of crossed modules
    \begin{align}
        p_U = (p_{U,1}, p_{U,0}) : \cmPL(U) \xrightarrow{\alpha_U} \bsF(U) \xrightarrow{\bsF(\phi \circ \iota_U)} \bsF(W).
    \end{align}
    We assemble the $p_{U,1}$ over all 2-dimensional subspaces into a function $p_1: \planarloop(V) \to \sF_1(W)$. Along with the homomorphism $\sF_{0}(\phi) \circ \alpha_{V,0} : \PL_0(V) \to \sF_0(W)$, this map satisfies the hypotheses of~\Cref{cor:cmPL_enhanced_univ_property}. Hence, there is a unique homomorphism $P: \PL_1(V) \to \sF_1(W)$ such that $(P, \sF_{0}(\phi) \circ \alpha_{V,0})$ is a morphism of crossed modules, and $P \circ \xi_{V,1} = p_1$. Hence, because both $\sF_{1}(\phi) \circ \alpha_{V,1}$ and $\alpha_{W,1} \circ \PL_{1}(\phi)$ satisfy these conditions, they must be equal. 
\end{proof}

To use this result to extend the realization functor, we will require the following lemma.

\begin{lemma} \label{lem:unique_thin_surface_2d}
    Let $\bX, \bY \in C^1([0,1]^2, \R^2)$ be a pair of maps with equal corners $\bX_{i,j} = \bY_{i,j} = c_{i,j}$ for $i, j =\{0,1\}$ and such that each boundary path \eqref{eq:individual_boundary_paths} is thin homotopy equivalent $\partial_i \bX \sim_{\thinhom} \partial_i \bY$ for $i \in \{l,b,r,t\}$. Then $\bX \sim_{\thinhom} \bY$.
\end{lemma}
\begin{proof}
    Let $h_i : [0,1]^2 \to \R^2$ by the thin homotopy between $\partial_i \bX$ and $\partial_i \bY$ for $i \in \{l,b,r,t\}$. Then $\bX$ is thin homotopy equivalent to the surface $\tbX$, where we glue the four thin path homotopies $h_i$ along the boundary as follows.

    \begin{figure}[!h]
        \includegraphics[width=1.0\linewidth]{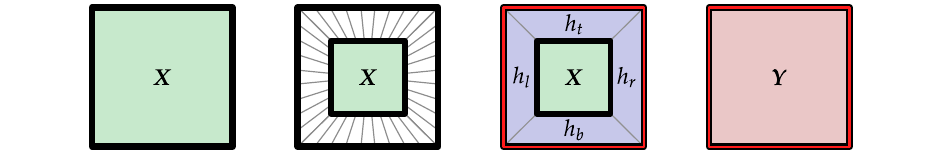}
    \end{figure}

    \noindent Then, the linear homotopy between $\tbX$ and $\bY$ is thin as it must have rank at most $2$ and does not change the boundaries. Hence $\tbX \sim_{\thinhom} \bY$. 
\end{proof}

\begin{lemma} \label{lem:surface_realization_planar_loop}
    For every planar loop $\bb = (b_1, \ldots, b_m)_{\min} \in \PL_0^{\cl}(U) \subset \planarloop(V)$, there exists a surface $\bX^\bb \in C^1([0,1]^2, U)$ whose boundary is in the thin homotopy class of $\realization_0(\bb)$.
\end{lemma}
\begin{figure}[!h]
    \includegraphics[width=\linewidth]{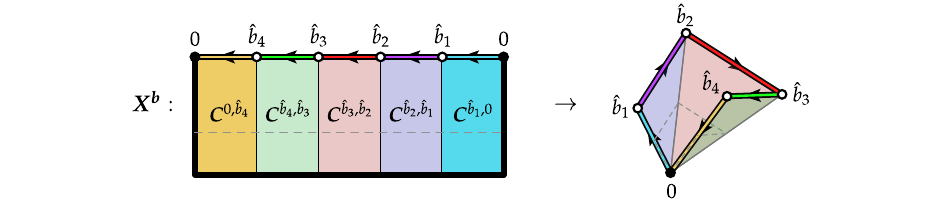}
\end{figure}
\begin{proof}
For $u,v \in U$, define the function
\begin{align}
    \bC^{u,v}: [0,1]^2 \to U \quad \text{by} \quad \bC^{u,v}_{s,t} \coloneqq \psi(t) \cdot (u + \psi(s) \cdot (v-u)),
\end{align}
where $\psi: [0,1] \to [0,1]$ is a reparametrization with sitting instants. 
Define $\bX^\bb \in C^1([0,1]^2, U)$ by
\begin{align} \label{eq:Xb}
    \bX^{\bb} = \bC^{0,\hb_{m-1}} \concat_h \bC^{\hb_{m-1}, \hb_{m-2}} \concat_h \ldots \concat_h \bC^{\hb_2, \hb_1} \concat_h \bC^{\hb_1,0},
\end{align}
where $\concat_h$ is defined in~\eqref{eq:strict_surface_concatenation}, and the concatenations are defined from left to right (as $\concat_h$ is not associative). 
Note that the left, bottom, and right boundaries of $\bX^{\bb}$ are trivial, and it can be verified that $\bX^\bb$ has the correct boundary. 
\end{proof}

\begin{proposition} \label{prop:cmrealization}
    There exists a unique natural transformation 
    \begin{align}
        \cmrealization = (\realization_1, \realization_0) : \cmPL \Rightarrow \cmthingroup,
    \end{align}
    called \emph{realization}, which extends $\realization_0: \PL_0 \Rightarrow \thingroup_1$ from~\eqref{eq:realization0_nt}. 
\end{proposition}
\begin{proof}
    First, $\cmthingroup(U)$ is trivial when $\dim(U) = 0$, and $\thingroup_2(U)$ is trivial when $\dim(U) = 1$, as every surface in $U$ is thinly null homotopic by the linear homotopy. 
    By~\Cref{lem:unique_thin_surface_2d}, the crossed module boundary $\delta: \thingroup_2(U) \to \thingroup_1(U)$ is injective when $\dim(U) = 2$. Therefore, $\cmthingroup$ is a planar functor. Furthermore, $\realization_0(\PL_0^{\cl}(U)) \subseteq \im(\partial: \thingroup_2(U) \to \thingroup_1(U))$ when $\dim(U) = 2$ as every loop in $U$ can be filled in to obtain a surface by~\Cref{lem:surface_realization_planar_loop}. 
    Thus, by~\Cref{lem:unique_extension_of_nt}, there exists a unique extension $\cmrealization = (\realization_1, \realization_0)$.
\end{proof}

Next, we will show that the path signature also extends uniquely to a natural transformation. 

\begin{theorem} \label{thm:unique_surface_signature}
    There exists a unique natural transformation
    \begin{align}
        \bS_{\PL} = (\SigPL, \sigPL): \cmPL \Rightarrow \com{\cmK},
    \end{align}
    called the \emph{piecewise linear signature}, which extends $\sigPL: \PL_0 \Rightarrow \com{K}_0$ from~\Cref{prop:psigpl_nt}.
\end{theorem} 
\begin{proof} 
    First, $\cmk(U)$ is trivial when $\dim(U) = 0$, and $\fk_1(U)$ is trivial when $\dim(U) = 1$ since $\Lambda^2 U = 0$. 
    Now, consider a 2-dimensional vector space $U$. By~\Cref{lem:closed_current}, we note that the kernel $\ker(\delta: \fk_1(U) \to \fk_0(U))$ is trivial. Thus
\begin{align}
    \fk_1(U) \cong [\fk_0(U), \fk_0(U)] = \LCS_2(\fk_0(U))
\end{align}
as Lie algebras and the boundary map on groups $\delta:\com{K}_1(U) \hookrightarrow \com{K}_0(U)$ is injective, so $\com{\cmK}$ is a planar functor. 
Furthermore, we claim that $\sigPL(\PL_0^{\cl}(U)) \subset \im(\delta: \com{K}_1(U) \to \com{K}_0(U))$. Indeed, the signature of a closed loop $\bx \in \PL_0^{\cl}(U)$ must be trivial in degree $1$, implying that $\log(\sigPL(\bx)) \in \widehat{\LCS}_2(\fk_0(U))$. But we have just seen that $\im(\delta: \com{\fk}_1(U) \to \com{\fk}_0(U)) = \widehat{\LCS}_2(\fk_0(U))$. Therefore, $\sigPL(\bx) \in \im(\delta: \com{K}_1(U) \to \com{K}_0(U))$. Thus, by~\Cref{lem:unique_extension_of_nt}, there exists a unique extension $\bS_{\PL} = (\SigPL, \sigPL)$.
\end{proof}

By the uniqueness of these constructions, these natural transformations factor in the same way as for paths in~\Cref{prop:psigpl_nt}.

\begin{proposition} \label{prop:factor_ssig_nt}
    The maps $\cmrealization$, $\bS$, and $\bS_{\PL}$ are natural transformations which factor as
    \begin{align}
        \bS_{\PL}: \cmPL \xRightarrow{\cmrealization} \cmthingroup \xRightarrow{\bS}  \com{\cmK}.
    \end{align}
\end{proposition}

This factorization implies that the smooth surface signature is uniquely determined.

\begin{theorem} \label{thm:unique_smooth_surface_signature}
    The smooth signature is the unique natural transformation
    \begin{align}
        \bS = (\Sig, \sig): \cmthingroup \Rightarrow \com{\cmK},
    \end{align}
    which extends $\sig: \thingroup_0 \Rightarrow \com{K}_0$ from~\eqref{eq:psig_nt} and such that $\Sig: \thingroup_2(V) \to \com{K}_1(V)$ is continuous with respect to the quotient topology on $\thingroup_2(V)$ induced by the Lipschitz topology on $C^1_0([0,1]^2, V)$.
\end{theorem}
\begin{proof}
    Let $\bT = (T_1, T_0): \cmthingroup \Rightarrow \com{\cmK}$ be a continuous natural transformation which extends $S_0$. Then $\bT \circ \cmrealization: \cmPL \to \com{\cmK}$ is a natural transformation which extends $\sigPL$. By ~\Cref{thm:unique_surface_signature}, $\bT \circ \cmrealization = \bS_{\PL}$. Therefore, $\bT$ and $\bS$ agree on piecewise linear surfaces. Now let $\bX \in C^1_0([0,1]^2, V)$. There exists a sequence of piecewise linear surfaces $\bX^k$ such that $\bX^k \xrightarrow{\Lip} \bX$ in the Lipschitz topology. Indeed, let $T^k$ be a triangulation of $[0,1]^2$ such that the mesh size approaches $0$. Define $\bX^k_{s,t} = \bX_{s,t}$ on all vertices $(s,t) \in T_k$, and extend linearly. Then, because $\bX$ is smooth and $[0,1]^2$ is compact, the derivatives of $\bX^{k}$ converge to those of $\bX$, and thus $\bX^k \xrightarrow{\Lip} \bX$.
    Because $T_{1}(\bX^k) = S_1(\bX^k)$ for all $k$, we must have $T_1(\bX) = \Sig(\bX)$ by continuity from~\Cref{prop:surface_signature_continuous}.
\end{proof}
\begin{remark}
    The uniqueness results of~\Cref{thm:unique_surface_signature} and~\Cref{thm:unique_smooth_surface_signature} can be strengthen by keeping track of basepoints as in ~\Cref{prop:groupoid_sig_uniqueness}. Indeed, by upgrading $\cmPL$, $\cmthingroup$, and $\com{\cmK}$ to crossed modules of \emph{groupoids} defined on the category $\Aff$ of affine spaces, we can show that the signature is uniquely characterized as the natural transformation extending the identity on objects.  
\end{remark}

\subsection{Computational Methods for the Surface Signature} \label{ssec:computational}
There are two main challenges in developing computational methods for the surface signature. 
\begin{enumerate}
    \item First, the standard definition of the surface signature via the surface holonomy equation in~\Cref{def:sh} requires the solution to a complicated differential equation.
    \item Second, the surface signature is valued in a group (or algebra) which lacks an evident natural choice of basis. This is in contrast to the path signature, which is valued in the completed tensor algebra and thus is equipped with a natural basis induced by a basis on the underlying vector space $V$. Indeed, in~\cite{lee_surface_2024} the surface signature is valued in a free crossed module of associative algebras. Even though this algebra is built out of tensor algebras, one must quotient by the Peiffer subspace. As a result, the choice of bases employed in ~\cite{lee_surface_2024} is non-canonical. 
\end{enumerate}

In the proof of~\Cref{thm:unique_surface_signature}, we observe through~\Cref{lem:unique_extension_of_nt} that the  signature of a planar surface is entirely determined by the path signature of its boundary. 
This already resolves the first problem mentioned above in the setting of piecewise linear surfaces: by leveraging the algebraic structure, we can compute the surface signature through composition of basic building blocks made up of the signatures of planar loops and their tail paths. In this section, we extend this approach by applying our previous results to develop a decomposition of the surface signature that enables tractable computational methods. Specifically, we demonstrate in~\Cref{thm:main_computation_result} how the signature can be decomposed into a boundary component, valued in a subset of the usual completed tensor algebra, and an abelian component, valued in a vector space of formal currents. This resolves the second problem, as we can equip these vector spaces with canonical bases. \medskip %

\subsubsection{Decomposition of the PL Crossed Module.}
We begin by constructing a decomposition of the piecewise linear crossed module $\cmPL(V)$. To do this, we will make use of a universal property for the group of piecewise linear loops $\PL_0^{\cl}(V)$, which we obtain by expressing $\PL_0^{\cl}(V)$ as a quotient of a free group generated by triangular loops. 

First, recall that the pair groupoid of $V$, $\Pair(V) \rightrightarrows V$, \label{pg:pair_gpd} is the groupoid whose set of objects is $V$ and such that there is a unique morphism between any two objects. Hence, the space of morphisms is $\Pair(V) = V \times V$. By sending the pair $(v, u) \in V \times V$ to the triangular loop with vertices $(0, v, u)$ we obtain a map of sets 
\begin{align} \label{eq:eta_V_loops}
   \eta_V: \Pair(V) \to \PL_0^{\cl}(V), \quad  \quad \eta_V(v,u) = (v, u-v, -u). 
\end{align}
A simple computation shows that $\eta_{V}$ evaluates to the identity on every pair $(v,u)$ for which $v$ and $u$ are colinear, and is a groupoid homomorphism when restricted to any affine line $\ell \subset V$. We use these two properties to define a monoid generated by the set of triangular loops. 
To this end, consider the quotient  \label{pg:loop}
\begin{align}
    \Loop(V) \coloneqq \FMon(V^2)/\sim
\end{align}
of the free monoid by the relations
\begin{enumerate}[label=(\textbf{L.\arabic*})]
    \item \label{L1} $(v, u) \concat (u, r) \sim (v, r)$ if $v, u, r$ lie on an affine line, and
    \item \label{L2} $(v,u) \sim \emptyset$ if $v$ and $u$ are linearly dependent.
\end{enumerate}
Note that with these relations, $\Loop(V)$ is a group. Indeed, the inverse of the generator $(v,u)$ is $(u,v)$. The following theorem is proved in~\Cref{apxsec:pl_loops}.

\begin{theorem} \label{thm:pl_loop_iso}
    There is a group isomorphism $\Loop(V) \cong \PL_0^{\cl}(V)$.
\end{theorem}

As a result, we obtain the following universal property for the group $\PL_{0}^{\cl}(V)$, which is analogous to ~\Cref{lem:PL0_univ_property} and~\Cref{thm:univ_prop_plcm_local}. 
\begin{proposition} \label{prop:univ_property_loops}
    Let $V$ be a vector space and let $G$ be a group. Let $f: \Pair(V) \to G$ be a map which 
    \begin{enumerate}
        \item restricts to a groupoid homomorphism on $\Pair(\ell)$ for every affine line $\ell \subset V$, and
        \item restricts to the trivial homomorphism on $\Pair(L)$ for every subspace $L \subset V$ with $\dim(L) \leq 1$.
    \end{enumerate}
    Then there exists a unique group homomorphism $F: \PL_0^{\cl}(V) \to G$ such that $F \circ \eta_V = f$. 
\end{proposition}

We can use this universal property to define a section of $\delta: \PL_1(V) \to \PL_0^{\cl}(V)$. Indeed, consider the map $c: \Pair(V) \to \PL_1(V)$ defined by $c(v,u) = (\emptyset_0, \eta_V(v,u)) \in \PL_1(V)$. This map verifies the equation $\delta \circ c = \eta_{V}$ and satisfies the two conditions in~\Cref{prop:univ_property_loops}.
Therefore, it induces a homomorphism $\Cone_{\PL}: \PL_0^{\cl}(V) \to \PL_1(V)$\label{pg:PLcone} which satisfies $\delta \circ \Cone_{\PL} = \id$. As a result, the following short exact sequence splits
\begin{align}
    1 \rightarrow \PL_1^{\cl}(V) \hookrightarrow \PL_1(V) \xrightarrow{\delta} \PL_0^{\cl}(V) \to 1.
\end{align}

\begin{corollary} \label{cor:PL_decomposition}
    The map
    \begin{align}
        \Phi : \PL_1(V) \to \PL_1^{\cl}(V) \times\PL_0^{\cl}(V), \quad \Phi(\bX) = (\bX \concat (\Cone_{\PL} \circ \delta(\bX))^{-1}, \delta(\bX))
    \end{align}
    is a group isomorphism.
\end{corollary}
\begin{proof}
    Because the above splitting, $\PL_1(V)$ is isomorphic to the semidirect product $\PL_1^{\cl}(V) \rtimes\PL_0^{\cl}(V)$. However, the conjugation action of $\PL_0^{\cl}(V)$ on $\PL_1^{\cl}(V)$ is trivial since $\PL_1^{\cl}(V)$ is in the center of $\PL_1(V)$. Thus $\PL_1(V)$ is isomorphic to the direct product.
\end{proof}

The crossed module action in $\cmPL(V)$ induces an action of $\PL_0(V)$ on $\PL_1^{\cl}(V) \times\PL_0^{\cl}(V)$ under the isomorphism $\Phi$. 
In particular, for $\bX \in \PL_1(V)$ and $\bx \in \PL_0(V)$, we have
\begin{align}
    \Phi(\bx \gt \bX) &= \Big( (\bx \gt \bX) \concat \Cone_{\PL} ( \delta(\bx \gt \bX)))^{-1}, \bx \concat \delta(\bX) \concat \bx^{-1}\Big),
\end{align}
and the induced action on $(\bZ, \bb) \in \PL_1^{\cl}(V) \times\PL_0^{\cl}(V)$, where $\bX = \Phi^{-1}(\bZ, \bb) = \bZ \concat \Cone_{\PL}(\bb)$, is
\begin{align}
    \bx \gt (\bZ, \bb) &= \Big( (\bx \gt \bZ) \concat (\bx \gt \Cone_{\PL}(\bb)) \concat \Cone_{\PL}(\bx \concat \bb \concat \bx^{-1})^{-1}, \bx \concat \bb \concat \bx^{-1}\Big).
\end{align}
Define the \emph{suspension}, $\Susp_{\PL} : \PL_0(V) \times \PL_0^{\cl}(V) \to \PL_1^{\cl}(V)$, by
\begin{align}
    \Susp_{\PL}(\bx, \bb) = (\bx \gt \Cone_{\PL}(\bb)) \concat \Cone_{\PL}(\bx \concat \bb \concat \bx^{-1})^{-1}.
\end{align}
Using the Peiffer identity, we can verify that, if $\bc \in \PL_{0}^{\cl}(V)$, then 
\begin{align}
    \Susp_{\PL}(\bc \concat \bx, \bb) = \Susp_{\PL}(\bx \concat \bc, \bb) = \Susp_{\PL}(\bx, \bc \concat \bb \concat \bc^{-1}) = \Susp_{\PL}(\bx, \bb) %
\end{align}
In particular, $\Susp_{\PL}(\bx, \bb)$ only depends on the path $\bx$ through its endpoint $t(\bx) \in V$. 
If $\bx = \delta(\Cone_{\PL}(\bx)) \in \PL_{0}^{\cl}(V)$, then by the Peiffer identity we have
\begin{align}
    \bx \gt (\bZ, \bb) = (\emptyset_1, \bx) \cdot (\bZ, \bb) \cdot (\emptyset_1, \bx)^{-1} = (\bZ, \bx \concat \bb \concat \bx^{-1}).
\end{align}
Hence, the action of $\PL_{0}^{\cl}(V)$ on the $\PL_1^{\cl}(V)$-component is trivial. As a result, we obtain the following expression for the action. 

\begin{lemma}
    The action of $\PL_{0}(V)$ on $\PL_1^{\cl}(V) \times\PL_0^{\cl}(V)$ is given by the following formula
    \begin{align}
        \bx \gt (\bZ, \bb) = \Big( (t(\bx) \gt \bZ) \concat \Susp_{\PL}(t(\bx), \bb), \bx \concat \bb \concat \bx^{-1}\Big),
    \end{align}
    where $(\bZ, \bb) \in \PL_1^{\cl}(V) \times\PL_0^{\cl}(V)$, $\bx \in \PL_0(V)$, and we view $t(\bx) \in \PL_0(V)$ as the linear path from the origin to the endpoint of $\bx$.
   \end{lemma}

\subsubsection{Decomposition of the Kapranov Crossed Module.} \label{ssec:kapranovdecompsection}
Next, we study a canonical splitting of Kapranov's Lie algebra $\fk_1(V)$ in order to obtain a decomposition of the group $\com{K}_1(V)$.
We use the definitions in~\Cref{ssec:forms_and_currents} for polynomial differential forms and currents. 
Let
\begin{align}
    E = \sum_{i} x_i \partial_{x_i}
\end{align}
be the Euler vector field on the vector space $V$. Let $\iota_E: \poly{\Omega}^{k+1}(V) \to \poly{\Omega}^k(V)$ be the derivation of the de Rham complex given by interior product with $E$. This operator is $\gl(V)$-equivariant and satisfies $\iota_E^2 = 0$. The Lie derivative $L_{E} = [d, \iota_E]$ acts on the weight $r$ subcomplex $\poly{\Omega}^k(V)_{r}$ by multiplication by $r$. Recall that we can extend the de Rham complex by putting a copy of the base field $\mathbb{R}$ in degree $-1$. We can then extend $\iota_{E}$ by sending an element $f \in \poly{\Omega}^{0}(V)$ to the value $f(0)$. As part of the proof of the Poincare lemma, we obtain the following decomposition.

\begin{lemma}
    For each $k \in \N$, we have a direct sum decomposition of $\gl(V)$-represenations,
    \begin{align}
        \poly{\Omega}^k(V) = \ker(d: \poly{\Omega}^k(V) \to \poly{\Omega}^{k+1}(V)) \oplus \ker(\iota_E: \poly{\Omega}^k(V) \to \poly{\Omega}^{k-1}(V)).
    \end{align}
\end{lemma}

We dualize this to get a decomposition for currents. Recall that $\partial: \poly{\Gamma}_k(V) \to \poly{\Gamma}_{k-1}(V)$ is the codifferential from~\eqref{eq:codifferential}. Similarly, the dual of $\iota_{E}$ is the operator $e: \poly{\Gamma}_k(V) \to \poly{\Gamma}_{k+1}(V)$ defined by 
\begin{align}
\langle e(\alpha), \omega \rangle = (-1)^k\langle \alpha, \iota_E \omega\rangle
\end{align}
for $\alpha \in \poly{\Gamma}_k(V)$ and $\omega \in \poly{\Omega}^{k+1}(V)$. It can be expressed explicitly as 
\begin{align}
    e(u_1 \cdots u_r \otimes v_1 \wedge \ldots \wedge v_k) = \sum_{i=1}^r u_1 \cdots \hat{u}_i \cdots u_r \otimes u_i \wedge v_1 \wedge \ldots \wedge v_k .
\end{align}
Because $e$ is the dual of $\iota_E$, we have $e^2 = 0$. Furthermore, $\ell = [e, \partial]$ acts on $\poly{\Gamma}_k(V)_{r}$ by multiplication by the (negative) weight $-r$. By the same argument as above, we obtain the following.
\begin{lemma}
    For each $k \in \N$, we have a direct sum decomposition of $\gl(V)$-representations,
    \begin{align}
        \poly{\Gamma}_k(V) = \ker(\partial: \poly{\Gamma}_k(V) \to \poly{\Gamma}_{k-1}(V)) \oplus \ker(e: \poly{\Gamma}_k(V) \to \poly{\Gamma}_{k+1}(V)).
    \end{align}
\end{lemma}
As a particular case of this Lemma, we have the decomposition $\poly{\Gamma}_2(V) = \poly{\Gamma}_{2}^{\cl}(V) \oplus \ker(e)$. Recall~\Cref{thm:rho_isomorphism}, which states that the abelianization map $\rho: \fk_1(V) \to \poly{\Gamma}_2(V)$ sends $\fa_1(V) = \ker(\delta)$ isomorphically to $\poly{\Gamma}_{2}^{\cl}(V)$. Therefore, the ideal $\rho^{-1}(\ker(e)) \subset \fk_1(V)$ is a complement to $\fa_1(V)$ and hence is isomorphic to $\LCS_2(\fk_0(V)) = [\fk_0(V), \fk_0(V)]$ via the map $\delta$ (cf. \cite{MR1252663}). Working with completions, we define 
\begin{align}
    \com{\cE}(V) \coloneqq \rho^{-1}\Big(\ker(e: \com{\Gamma}_2(V) \to \com{\Gamma}_{3}(V))\Big) \quad \text{where} \quad \rho: \com{\fk}_1(V) \to \com{\Gamma}_2(V).
\end{align}

Then $\delta: \com{\cE}(V) \to \widehat{\LCS}_2(\fk_0(V))$ is an isomorphism, and we denote the inverse morphism as
\begin{align}
    \fcone: \widehat{\LCS}_2(\fk_0(V)) \xrightarrow{\cong} \com{\cE}(V) \subset \com{\fk}_1(V).
\end{align}
In particular, the map $\fcone$ satisfies $\delta \circ \fcone = \id$, and should be viewed as the Lie algebraic analogue of the $\Cone_{\PL}$ construction. 
Thus, we obtain the following decomposition. %

\begin{corollary} \label{cor:kapranov_decomposition_E}
    There exists a canonical isomorphism of Lie algebras
    \begin{align}
        \Psi: \com{\fk}_1(V) \xrightarrow{\cong} \com{\Gamma}_2^{\cl}(V) \oplus \widehat{\LCS}_2(\fk_0(V)) \quad \text{defined by} \quad \Psi(A) = \Big(\rho(A - \fcone\circ \delta(A)), \delta(A) \Big).
    \end{align}
\end{corollary}

Under this canonical decomposition of $\com{\fk}_1(V)$, the crossed module action of $\com{\cmk}(V)$ induces an action of $\com{\fk}_0(V)$ on $\com{\Gamma}_2^{\cl}(V) \oplus \widehat{\LCS}_2(\fk_0(V))$, which we can write out explicitly. For $x \in \com{\fk}_0(V)$ and $A \in \com{\fk}_1(V)$, we have
\begin{align}
    \Psi(x \gt A) = \Big( \rho(x \gt A - \fcone \circ \delta (x \gt A)), \delta(x \gt A) \Big).
\end{align}
From the construction of $\com{\cmk}(V)$ outlined in ~\Cref{apx:xlie}, we observe that the abelianization map $\rho$ satisfies $\rho(x \gt A) = \pi(x) \rho(A)$, where $V$ acts via the product in the symmetric algebra $\com{S}(V)$ of $\com{\Gamma}_2 = \com{S}(V) \otimes \Lambda^2(V)$. Therefore, the induced action on $(\gamma, y) \in \com{\Gamma}_2^{\cl}(V) \oplus \widehat{\LCS}_2(\fk_0(V))$, where $A = \Psi^{-1}(\gamma, y) = \rho^{-1}(\gamma) + c(y)$, is
\begin{align}
    x \gt (\gamma, y) = \Big(\pi(x)\gamma + \rho(x \gt \fcone(y) - \fcone([x,y])), [x,y]\Big).
\end{align}
We define the Lie algebraic analogue of the suspension $\fsusp: \com{\fk}_0(V) \times \widehat{\LCS}_2(\fk_0(V)) \to \com{\Gamma}_2^{\cl}(V)$ by 
\begin{align}
    \fsusp(x,y) = \rho(x \gt \fcone(y) - \fcone([x,y])),
\end{align} 
which is the projection of $\rho(x \gt \fcone(y)) \in \com{\Gamma}_2$ onto $\com{\Gamma}_2^{\cl}$.
By applying the Peiffer identity, we can verify that for $r \in \widehat{\LCS}_2(\fk_0(V))$, we have
\begin{align}
    \fsusp(r+x, y) = \fsusp(x,y).
\end{align}
Therefore, $\fsusp(x,y)$ only depends on $v = \pi(x) \in V$. As a result, we obtain the following expression for the crossed module action.

\begin{lemma}
    The action of $\com{\fk}_0(V)$ on $\com{\Gamma}_2^{\cl}(V) \oplus \widehat{\LCS}_2(\fk_0(V))$ is given by the following formula 
    \begin{align}
        x \gt (\gamma, y) = (\pi(x) \gamma + \fsusp(\pi(x), y), [x,y]),
    \end{align}
    where $(\gamma, y) \in \com{\Gamma}_2^{\cl}(V) \oplus \widehat{\LCS}_2(\fk_0(V))$ and $x \in \com{\fk}_0(V)$.
\end{lemma}

\begin{remark}
    The suspension map $\fsusp$ in fact only depends on $y \in \widehat{\LCS}_2(\fk_0(V))$ through its abelianization. To see this, consider the operator $\ell = [e, \partial]$ acting on currents $\com{\Gamma}_2$. It satisfies
    \begin{align}
        \id = \ell^{-1} e \partial + \ell^{-1} \partial e. 
    \end{align}
    Since $\ell$ acts as multiplication by the negative weight $-r$, we can identify $\ell^{-1} e \partial$ as the projection onto $\ker(e)$, and $\ell^{-1} \partial e$ as the projection onto $\ker(\partial) = \com{\Gamma}_2^{\cl}$. Thus, we can express $\fsusp$ as
    \begin{align}
        \fsusp(x,y) = \ell^{-1} \circ \partial \circ e \big(\pi(x) \rho(\fcone(y))\big),
    \end{align}
    so that $\fsusp(x,y)$ depends only on the projection $\pi(x)$ and the abelianization $\rho(\fcone(y))$.
\end{remark}

Finally, the decomposition of ~\Cref{cor:kapranov_decomposition_E} induces a decomposition of the group $\com{K}_1(V)$
\begin{align} \label{eq:kapranov_decomposition}
    \com{K}_1(V) = \com{K}_1^{\Gamma}(V) \times \com{K}_1^{\cE}(V),
\end{align}
where $\com{K}_1^{\Gamma}(V) \coloneqq \ker(\delta: \com{K}_1(V) \to \com{K}_0(V))$. Note that because $\com{K}_1^{\Gamma}(V)$ is abelian, we may identify it with $\com{\Gamma}_2^{\cl}(V)$ with its additive group structure. Finally, because the constructions in this section are functorial, given a linear map $\phi: V \to W$, the induced group homomorphism $\com{K}_1(\phi) : \com{K}_1(V) \to \com{K}_1(W)$ preserves decomposition in \Cref{eq:kapranov_decomposition}. \medskip

\subsubsection{Decomposition of PL Surface Signature.}
Next, we study how the surface signature factors with respect to the decompositions of both $\PL_1(V)$ and $\com{K}_1(V)$. We consider
\begin{align} \label{eq:decomposed_sigpl}
    \SigPL \circ \Phi^{-1}: \PL_1^{\cl}(V) \times \PL_0^{\cl}(V) \to \com{K}_1^{\Gamma}(V) \times \com{K}_1^{\cE}(V),
\end{align}
where $\Phi^{-1}: \PL_1^{\cl}(V) \times \PL_0^{\cl}(V) \to \PL_1(V)$ is the isomorphism from~\Cref{cor:PL_decomposition} given by $\Phi^{-1}(\bX, \bb) = \bX \concat \Cone_{\PL}(\bb)$.

\begin{lemma}
    The homomorphism $\SigPL \circ \Phi^{-1}$ preserves the decompositions of both $\PL_1(V)$ and $\com{K}_1(V)$. 
\end{lemma}
\begin{proof}
    The restriction of $\SigPL $ to $\PL_{1}^{\cl}(V)$ must be valued in $\com{K}_1^{\Gamma}(V)$ since $\bS_{\PL}$ is a morphism of crossed modules.
    To understand the restriction of $\SigPL \circ \Phi^{-1}$ to $\PL_0^{\cl}(V)$, consider the function
\begin{align} \label{eq:sv_definition}
    s_V: \Pair(V) \xrightarrow{\eta_V} \PL_0^{\cl}(V) \xrightarrow{\Cone_{\PL}} \PL_1(V) \xrightarrow{\SigPL} \com{K}_1(V).
\end{align}
Let $(v,u) \in \Pair(V)$, let $U \subset V$ be a 2-dimensional subspace such that $v,u \in U$, and let $\phi_U : U \hookrightarrow V$ denote the inclusion. Because the PL surface signature is a natural transformation, we have
\begin{align}
    s_V(v,u) = \com{K}_1(\phi_U) \circ s_U(v,u). 
\end{align}
Furthermore, the factor $\com{K}_1^{\Gamma}(U) = 0$ is trivial since $U$ is 2-dimensional, and this implies that $s_U(v,u) \in \com{K}_1^{\cE}(U)$. Then, because $\com{K}_1(\phi_U)$ preserves the decomposition in~\eqref{eq:kapranov_decomposition}, we conclude that $s_V(v,u) \in \com{K}_1^{\cE}(V)$. The map $s_V$ satisfies the hypotheses of~\Cref{prop:univ_property_loops} and so by the universal property, there is a unique map 
\begin{align}
    F : \PL_0(V) \to \com{K}_1^{\cE}(V) \subset \com{K}_1(V)
\end{align}
such that $F \circ \eta_V = s_V$. Recalling the definition of $s_V$ from ~\eqref{eq:sv_definition}, we see that $(\SigPL \circ \Cone_{\PL}) \circ \eta_V = s_{V}$ and hence 
\begin{align}
    F = \SigPL \circ \Cone_{\PL} = (\SigPL \circ \Phi^{-1})|_{\PL_0^{\cl}(V)},
\end{align}
implying that $\SigPL \circ \Cone_{\PL}$ is valued in $\com{K}_1^{\cE}(V)$. As a result, we conclude that $\SigPL \circ \Phi^{-1}$ respects the decompositions of both $\PL_1(V)$ and $\com{K}_1(V)$.
\end{proof}

The signature satisfies the following equation
\begin{align}
   \delta \circ \SigPL \circ \Cone_{\PL} = \sigPL \circ \delta \circ \Cone_{\PL} = \sigPL. 
\end{align}
Hence, using the identification between $\com{K}_1^{\cE}(V)$ and its image under $\delta$ in $\com{K}_0(V)$, we conclude that $(\SigPL \circ \Phi^{-1})|_{\PL_0^{\cl}(V)} = \sigPL$. This implies the following decomposition of the surface signature.

\begin{proposition} \label{prop:PLsig_decomposition}
    The piecewise linear surface signature decomposes as 
    \begin{align}
        \SigPL = (\SigPL^{\Gamma}, \SigPL^{\cE}): \PL_1(V) \to   \com{K}_1^{\Gamma}(V) \times \com{K}_1^{\cE}(V),
    \end{align}
    where
    \begin{align}
        \SigPL^{\Gamma}(\bX) = \SigPL(\bX \concat \Cone_{\PL}(\delta(\bX)^{-1})) \andd \SigPL^{\cE}(\bX) = \sigPL(\delta(\bX)).
    \end{align}
\end{proposition}

\subsubsection{Decomposition of Smooth Surface Signature}

We can also consider a decomposition for the smooth thin crossed module
\begin{align}
    \cmthingroup(V) = (\partial: \thingroup_2(V) \to \thingroup_1(V), \gt).
\end{align}
We denote the kernel by $\thingroup_2^{\cl}(V) \coloneqq \ker(\partial: \thingroup_2(V) \to \thingroup_1(V))$, and the group of thin homotopy classes of loops by $\thingroup_1^{\cl}(V) \coloneqq \im(\partial: \thingroup_2(V) \to \thingroup_1(V))$. Recall our convention that these groups consist of thin homotopy classes of paths and surfaces based at the origin, so $\bx_0 = \bX_{0,0} = 0$. Now, we define a section of $\partial: \thingroup_2(V) \to \thingroup^{\cl}_1(V)$
\label{pg:smooth_cone} by
\begin{align} \label{eq:smooth_cone}
    \Cone: \tau_1^{\cl}(V) \to \tau_2(V), \quad \Cone(\bx) = \Big( (s,t) \mapsto t \bx_{s}\Big).
\end{align}
This is well-defined under thin homotopy. Indeed, suppose $\bx \sim_{\thinhom} \by$, and let $h: [0,1]^2 \to V$ be a thin homotopy between $\bx$ and $\by$ where $h_{0,s} = \bx_s$ and $h_{1,s} = \by_s$. Define a map $H: [0,1]^3 \to V$ by
\begin{align}
    H_{u,s,t} = t h_{u,s}, \quad \text{where} \quad H_{0,s,t} = \Cone(\bx)_{s,t} \andd H_{1,s,t} = \Cone(\by)_{s,t}.
\end{align}
Because $\rank(dh) \leq 1$, we must have $\rank(dH) \leq 2$. Hence, $\Cone(\bx) \sim_{\thinhom} \Cone(\by)$.
Then, by the same arguments as the PL setting in~\Cref{cor:PL_decomposition}, we have the following.

\begin{lemma}
    The map
    \begin{align} \label{eq:smooth_surface_decomposition}
        \Phi: \thingroup_2(V) \to \thingroup^{\cl}_2(V) \times \thingroup^{\cl}_1(V), \quad \Phi(\bX) = \Big(\bX \concat \Cone(\partial(\bX))^{-1}, \partial(\bX)\Big)
    \end{align}
    is a group isomorphism.
\end{lemma}

The map $\Cone$ extends the piecewise linear cone in the sense that $\Cone \circ \realization_0 = \realization_1 \circ \Cone_{\PL}$. This can be verified by checking equality after precomposing with $\eta_V$ and applying~\Cref{prop:univ_property_loops}. Using this fact, we can show that the surface signature also preserves the decompositions of $\thingroup_2(V)$ and $\com{K}_1(V)$. The main property to show is that cones are mapped to $\com{K}_1^{\cE}(V)$ under the surface signature.

\begin{lemma}
    The map $\Sig \circ \Cone : \thingroup_1^{\cl}(V) \to \com{K}_1(V)$ is valued in $\com{K}_1^{\cE}(V)$.
\end{lemma}
\begin{proof}
    Recall that $\fk_1(V) = \Gamma_2^{\cl}(V) \oplus \cE(V)$. By projection, this decomposition also holds for truncations $\fk_1^{(n)}(V) = \Gamma_2^{\cl, (n)}(V) \oplus \cE^{(n)}(V)$. Then, by~\cite[Section 2.7, Problem 4]{rossmann_lie_2002}, the exponential group $\com{K}_1^{\cE, (n)}(V)$ is a closed subgroup in $\com{K}_1^{(n)}(V)$. \medskip
    
    Next, let $\bx \in C^1([0,1], V)$ be a smooth loop such that $\bx_0 = \bx_1 = 0$. Then, by the same reasoning as the proof of~\Cref{thm:unique_smooth_surface_signature}, there exists a sequence $\{\bx^k\}_{k=1}^\infty$ of piecewise linear loops such that $\bx^k \xrightarrow{\Lip} \bx$ as $k \to \infty$. This implies that $\Cone(\bx^k) \rightarrow \Cone(\bx)$ in Lipschitz norm as $k \to \infty$ as well. Then, since $\Cone(\bx^k)$ is the cone of a piecewise linear loop, \Cref{prop:PLsig_decomposition} shows that $\Sig(\Cone(\bx^k)) \in \com{K}_1^{\cE}(V)$. Finally, since the (truncated) surface signature is continuous with respect to the Lipschitz topology (\Cref{prop:surface_signature_continuous}), we have $\Sig(\Cone(\bx)) \in \com{K}_1^{\cE}(V)$. 
\end{proof}

Thus, the analogue of~\Cref{prop:PLsig_decomposition} holds also in the setting of smooth surfaces. Combining this with the description of $\com{K}_1^{\cE}(V)$ from~\Cref{ssec:kapranovdecompsection}, and~\Cref{thm:ssig_equiv_to_abelianization} on the signature of closed surfaces, we obtain the following explicit formula for the surface signature, which is the main result of this section. 

\begin{theorem} \label{thm:main_computation_result}
    There is a canonical embedding 
    \begin{align}
        \com{K}_1(V) \to \com{\Gamma}_2(V) \times T\ps{V},
    \end{align}
    from the group $\com{K}_1(V)$ of formal surfaces to the product of $\com{\Gamma}_2(V)$, the vector space of formal $2$-currents, and $T\ps{V}$, the algebra of formal non-commutative power series. With respect to this embedding, the smooth surface signature decomposes as follows
    \begin{align}
        \Sig = (\Sig^{\Gamma}, \Sig^{\cE}): \thingroup_2(V) \to   \com{\Gamma}_2(V) \times T\ps{V}.
    \end{align}
    Given a smooth surface $\bX \in \thingroup_2(V)$, the two components of the signature are given by
    \begin{itemize}
        \item the path signature of the boundary path of $\bX$
         \begin{align}
        \Sig^{\cE}(\bX) = \sig(\partial(\bX)) \in T\ps{V},
        \end{align}
        \item the sum of surface integrals 
        \begin{align}
            \Sig^{\Gamma}(\bX) = \sum_{\alpha \in \N^n, \ i < j} \frac{1}{\alpha !} \left( \int_{\mathcal{C}(\bX)} z^{\alpha} dz_{i} \wedge dz_{j} \right) e^{\alpha} \otimes e_{i} \wedge e_{j}  
        \end{align}
        where, given linear coordinates $z_{i}$ on $V$, the sum runs over the set of all monomial $2$-forms $z^{\alpha} dz_{i} \wedge dz_{j}$ and where $e^{\alpha} \otimes e_{i} \wedge e_{j}$ are the dual polynomial $2$-currents. Furthermore, the integrals are taken over 
        \begin{align}
            \mathcal{C}(\bX) = \bX \concat \Cone(\partial(\bX))^{-1} \in \thingroup^{\cl}_2(V),
        \end{align} 
        the closed surface obtained by coning off the boundary of $\bX$.
    \end{itemize}
\end{theorem}

\begin{remark} \label{rem:canonicalexp}
    The expression for the abelian component $\Sig^{\Gamma}(\bX)$ of the signature of a surface $\bX$ from ~\Cref{thm:main_computation_result} has the following simple coordinate-free description. Observe that the linear structure on $V$ allows us to write 
    \[
    \Omega^2(V) = C^{\infty}(V) \otimes \Lambda^2V^*,
    \]
    which in turn allows us to view the surface integral as a map 
    \[
        \int_{\mathcal{C}(\bX)} : C^{\infty}(V) \to \Lambda^2V.
    \]
    The identity map on $V$ and its powers can be viewed as functions $(\mathrm{id}_V)^{k} \in C^{\infty}(V) \otimes S^{k}(V)$. Therefore 
    \[
    \int_{\mathcal{C}(\bX)} (\mathrm{id}_V)^{k} \in S^{k}(V) \otimes \Lambda^2V = \poly{\Gamma}_2(V)_{k+2}. 
    \]
    Then, using the multinomial theorem, the abelian component of the surface signature can be shown to have the following expression 
    \[
    \Sig^{\Gamma}(\bX) = \sum_{k \geq 0} \frac{1}{k!} \int_{\mathcal{C}(\bX)} (\mathrm{id}_V)^{k}.
    \]
\end{remark}

\section{Thin Homotopy of Piecewise Linear Surfaces} \label{sec:PL_thin_homotopy}
In this section, our aim is to show the following injectivity result.

\begin{theorem} \label{thm:PL_ssig_injectivity}
    The map $\SigPL: \PL_1(V) \to \com{K}_1(V)$ is injective. 
\end{theorem}
Because of the factorization of $\SigPL$ in~\Cref{prop:factor_ssig_nt}, this result has two main implications. First, it shows that our algebraic construction of $\PL_{1}(V)$ is rich enough to faithfully encode piecewise linear surfaces modulo thin homotopy. 
\begin{corollary}
    The realization map $\realization_1: \PL_1(V) \to \tau_2(V)$ is injective. 
\end{corollary}
In particular, this implies that there are two types of cancellations that can occur via thin homotopy of PL surfaces. 
\begin{enumerate}
    \item \textbf{Local Cancellations (Folds).} When two linear components with opposite orientation are adjacent, they are cancelled via the equivalence relation in~\ref{PL1.1}.
    \item \textbf{Non-local Cancellations.} When two linear components with opposite orientation are not adjacent, they may sometimes be cancelled via the Peiffer identity in~\eqref{eq:PL_peiffer}.
\end{enumerate}
Second, it shows that the surface signature can detect whether two piecewise linear surfaces are thinly homotopic to each other. Hence, this generalizes the signature condition \ref{S1}. 
\begin{corollary} \label{cor:smooth_sig_injective_on_PL}
    The restriction of the surface signature $\Sig: \thingroup_{2}(V) \to \com{K}_1(V)$ to $\im(\realization_1)$ is injective. 
\end{corollary}

\begin{remark} 
Consider the 2-truncated polynomial de Rham complex of $V$
\begin{align}
     \poly{\Omega}^0(V) \to \poly{\Omega}^1(V) \to \poly{\Omega}^2(V).
\end{align}
In order to prove a de Rham theorem for this complex, the appropriate replacement for the singular chains in degree $2$ appears to be the thin homotopy group of closed piecewise linear surfaces, $\PL_1^{\cl}(V) \coloneqq \ker(\delta: \PL_1(V) \to \PL_0(V))$ (or its smooth analogue). Indeed, there is a well-defined pairing given by integration 
\begin{align}
    \langle \cdot , \cdot \rangle : \PL_1^{\cl}(V) \times \frac{\poly{\Omega}^{2}(V)}{d\poly{\Omega}^1(V)} \to \R, \quad \quad \langle \bX, \omega \rangle = \int_\bX \omega.
\end{align}
Recalling the isomorphism \begin{align}
    \coker(d: \poly{\Omega}^1(V) \to \poly{\Omega}^2(V))^* \cong \com{\Gamma}_2^{\cl}(V),
\end{align}
where $\com{\Gamma}_2^{\cl}(V)$ denotes the closed formal 2-currents on $V$, the pairing induces a map $\PL_1^{\cl}(V) \to \com{\Gamma}_2^{\cl}(V)$, which we identify as the surface signature by ~\Cref{thm:main_computation_result}. Therefore, ~\Cref{thm:PL_ssig_injectivity} implies that this map is injective. Showing that this map is surjective, at least up to a finite degree in the weight grading, amounts to proving a generalization of Chow's theorem~\cite{chow_uber_1940} (see also~\cite[Theorem 7.28]{friz_multidimensional_2010}). While we do not pursue this direction further, we leave this as an interesting avenue for future work.
\end{remark}

To begin the proof of~\Cref{thm:PL_ssig_injectivity}, we reduce the problem to the case of closed surfaces by a general result on crossed modules.

\begin{lemma} \label{lem:reduce_to_kernel}
    Let $\cmG = (\delta_G: G_1 \to G_0)$ and $\cmH = (\delta_H: H_1 \to H_0)$ be crossed modules of groups, and let $F = (F_1, F_0) : \cmG \to \cmH$ be a morphism of crossed modules. Suppose $F_0: G_0 \to H_0$ and $F_1: \ker(\delta_G) \to H_1$ are injective. Then $F_1: G_1 \to H_1$ is injective.
\end{lemma}
\begin{proof}
    Let $x \in G_1$ such that $F_1(x) = 1$. Because $F$ is a morphism of crossed modules, we have $\delta_H \circ F_1 = F_0 \circ \delta_G$, so $F_0( \delta_G(x)) = 1$. But since $F_0$ is injective, this implies that $x \in \ker(\delta_G)$. By injectivity of $F_1$ on the kernel, we have $x = 1$. 
\end{proof}

By ~\Cref{lem:realization0_injective} and the injectivity of the path signature, $S_{\PL,0}$ is injective. As a result, ~\Cref{lem:reduce_to_kernel} reduces the proof of~\Cref{thm:PL_ssig_injectivity} to showing that 
\begin{align}
    \SigPL: \PL_1^{\cl}(V) \to \com{K}_1(V)
\end{align}
is injective, where $\PL_1^{\cl}(V) = \ker(\delta: \PL_1(V) \to \PL_0(V))$. Given $\bX \in \PL^{\cl}_1(V)$ such that $\SigPL(\bX) = 0$, our strategy will be to consider various representatives $r(\bX)$ of $\bX$ in order to show that $\bX = 0$ in $\PL_1(V)$.

\subsection{Piecewise Linear Simplicial Complexes}
In this section, we consider simplicial complexes $C$ which can be realized in a vector space $V$, and explain how to construct maps from their fundamental crossed modules to $\cmPL(V)$.

\begin{definition} \label{def:plsc}
    A \emph{piecewise linear simplicial complex (PLSC) in $V$} is a finite ordered simplicial complex $C$ whose vertices lie in $V$. In particular, the set of vertices $C_0 \subset V$ is equipped with a total ordering.
    Each $n$-simplex $\sigma$ of $C$ is an ordered set\footnote{Our convention is to use square brackets to denote simplices. This is to distinguish between the \emph{point} representations used for simplices as opposed to the \emph{edge} representations for piecewise linear paths.} $\sigma = [p_0, \ldots, p_n]$ of points $p_i \in V$, and hence, is equipped with a linear characteristic map $\Phi_\sigma: \Delta^n_{\sigma} \to V$. Let $|\sigma| \subset V$ denote the image of $\Phi_{\sigma}$. It is the convex hull of the points of $\sigma$ in $V$. We say that an $n$-simplex $\sigma$ is \emph{degenerate} if there exists some $(n-1)$-dimensional affine hyperplane $U \subset V$ such that $|\sigma| \subset U$, and we say that $\sigma$ is \emph{non-degenerate} otherwise. We say that $C$ is \emph{non-degenerate} if all its simplices are non-degenerate. 
    
    We will abuse notation and use $C$ to refer both to the abstract simplicial complex as well as its geometric realization as a topological space, and we let $C_n$ denote the $n$-skeleton of $C$. Furthermore, we use $|C|$ to denote the \emph{piecewise linear realization} of $C$ in $V$, defined as the union of the subsets $|\sigma|$ 
     \begin{align}
        |C| \coloneqq  \bigcup_{\sigma \in C} |\sigma| \subset V. 
    \end{align}
    The complex is equipped with a continuous realization map $\Phi: C \to |C|.$

    \begin{remark} \label{rem:nondegone}
        Note that because a 1-simplex is a pair of distinct points $[p_0, p_1]$, $p_{i} \in V$, it is automatically non-degenerate.
    \end{remark}

\end{definition}

Furthermore, we will require a notion of compatible complexes.

\begin{definition} \label{def:compatible}
    Suppose $C$ is a PLSC in $V$. We say that $C$ is \emph{compatible} if for every pair of simplices $\sigma, \sigma' \in C$, their intersection in $V$ satisfies
    \begin{align} \label{eq:simplex_intersection}
        |\sigma| \cap |\sigma'| = |\tau| 
    \end{align}
    where $\tau \subseteq \sigma \cap \sigma'$ is a common subsimplex of both $\sigma$ and $\sigma'$, or $\tau = \emptyset$ in which case $|\sigma|$ and $|\sigma'|$ are disjoint. 
\end{definition}
\begin{figure}[!h]
    \includegraphics[width=1.0\linewidth]{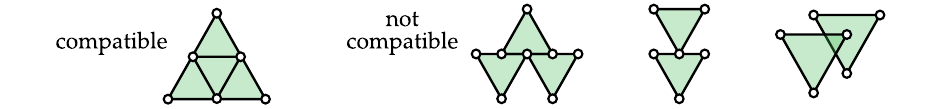}
\end{figure}
\begin{lemma} \label{lem:compatible_embedding}
    Let $C$ be a compatible non-degenerate PLSC. Then the realization map $\Phi: C \to |C|$ is a homeomorphism. 
\end{lemma}
\begin{proof}
    Because $C$ is compact and $|C|$ is Hausdorff, it suffices to prove injectivity. For each simplex $\sigma$, the map $\Phi: \Delta_\sigma \to |\sigma|$ is injective since $\sigma$ is nondegenerate. Hence, let $p \in \Delta_\sigma$, $q \in \Delta_\sigma'$ be points in different simplices such that $x = \Phi(p) = \Phi(q)$. Hence, $|\sigma|$ and $|\sigma'|$ intersect non-trivially. Therefore, there is a subsimplex $\tau$ of both $\sigma$ and $\sigma'$ such that $|\tau| = |\sigma| \cap |\sigma'|$. Thus, there is a point $r \in \Delta_\tau$ such that $\Phi(r) = x$. 
    Because $\tau \subseteq \sigma$ is a subsimplex we have $r \in \Delta_\sigma$. By injectivity of $\Phi|_{\Delta_\sigma}$, we conclude that $r = p$. Similarly, $r = q$. Therefore, $p = q$ and $\Phi$ is injective. 
\end{proof}

Let $C$ be a 2-dimensional PLSC. We denote the set of vertices by $C_0$, the set of $1$-simplices by $E$, and the set of $2$-simplices by $L$. The fundamental groupoid of the $1$-skeleton, $\Pi_1(C_1, C_0)$, is free on the set of edges $E$. We define a homomorphism 
\begin{align} \label{eq:tW_0}
\widetilde{W}_0 : \Pi_1(C_1, C_0) \to \PL_0(V)
\end{align}
by sending an edge $\epsilon = [p_0, p_1]$ to the element $\eta_V(p_1 - p_0) \in \PL_0(V)$. Composing with the realization map $\realization_0 : \PL_{0}(V) \to \thingroup_1(V)$, the edge $\epsilon$ gets sent to $\realization_0 \circ \widetilde{W}_0(\epsilon)$, the straight line path in $V$ from $0$ to $p_{1} - p_{0}$. Let $c_0 \in V$ be a basepoint. Restricting $\widetilde{W}_0$ to the fundamental group based at $c_0$ gives rise to a group homomorphism 
\begin{align} \label{eq:W0_def}
    W_0 : \pi_1(C_1, c_0) \to \PL_0(V).
\end{align}

\begin{lemma} \label{lem:W0_injective}
    Suppose $C$ is a PLSC whose $1$-skeleton $C_1$ is compatible and non-degenerate. Then the map $W_0: \pi_1(C_1,c_0) \to \PL_0(V)$ is injective.
\end{lemma}
\begin{proof}
    Let $\gamma \in \pi_1(C_1,c_0)$ be such that $W_0(\gamma) = 0$. Then the realization $\realization_0 \circ W_0(\gamma) \in \thingroup_1(V)$ is a loop in $|C_{1}|$ which is thinly null homotopic in $V$. By the image condition \ref{I1}, there exists a null homotopy $h$ whose image is contained in the image of $\realization_0 \circ W_0(\gamma)$, and therefore in $|C_{1}|$. Because the map $\Phi: C_1 \to |C_1|$ is a homeomorphism by \Cref{lem:compatible_embedding}, the null homotopy $h$ can be lifted to $C_1$. As a result, $\gamma = 0$.  
  \end{proof}

Free crossed modules were originally developed by Whitehead~\cite{whitehead_combinatorial_1949} to provide an algebraic description of second relative homotopy groups. We consider a special case of Whitehead's result for 2-dimensional CW-complexes, also discussed in~\cite{brown_identities_1982, brown_second_1980}. 

\begin{theorem}{\cite{whitehead_combinatorial_1949, brown_identities_1982}} \label{thm:whitehead_isomorphism}
    Let $(C,c_0)$ be a connected based 2-dimensional CW complex with 1-cells $E$ and 2-cells $L$. Fix a spanning tree $T \subset C_1$ of the 1-skeleton. For each 2-cell $\lambda \in L$, choose a basepoint $c_\lambda \in C_{0}$, and let $\phi_\lambda \in \pi_1(C_1,c_\lambda)$ be the attaching map of the cell. Furthermore, let $\omega_\lambda \in \Pi_1(C_1,C_0)$ be the unique path in $T$ connecting $c_0$ to $c_\lambda$. 
    Define $\rho: L \to \pi_1(C_1, c_0)$ by $\rho(\lambda) = \omega_\lambda \concat \phi_\lambda \concat \omega_\lambda^{-1}$. Then, the fundamental crossed module of the pair $(C,C_1)$
    \begin{align}
        \bpi(C, C_1) = \Big(\partial: \pi_2(C, C_1, c_0) \to \pi_1(C_1, c_0)\Big)
    \end{align}
    is the free crossed module on $\rho$. Let $\eta: L \to \pi_2(C, C_1, c_0)$ denote the map which sends each 2-cell $\lambda \in L$ to the corresponding \emph{generator} $\eta(\lambda)$ of $\pi_2(C, C_1, c_0)$.
\end{theorem}

Let $C$ be a $2$-dimensional PLSC. We will use Whitehead's theorem to construct a map $\bW: \bpi(C, C_1) \to \cmPL(V)$. First, choose a basepoint $c_{0} \in V$ and fix a spanning tree $T \subset C_{1}$ of the $1$-skeleton. Let $W_0$ be the homomorphism from ~\eqref{eq:W0_def}. By ~\Cref{thm:whitehead_isomorphism}, $\bpi(C, C_1)$ is free on the set of 2-cells $L$ of $C$, and so by the universal property of~\Cref{thm:free_xgrp}, the map $\bW$ is specified by a function $f: L \to \PL_1(V)$ such that $\delta \circ f = W_{0} \circ \rho$. 

Each 2-simplex $\lambda \in L$ has the form $\lambda = [p_0, p_1, p_2]$. Let $\phi_{\lambda} \in \pi_{1}(C_1, p_0)$ be its boundary loop. Under $\widetilde{W}_{0}$ it gets sent to the following planar loop 
\begin{align}
    \bb_\lambda = \widetilde{W}_{0}(\phi_{\lambda}) = (p_1 - p_0, p_2 - p_1, p_0 - p_2) \in \planarloop(V).
\end{align}
Let $\omega_\lambda \in \Pi(C_{1},C_0)$ be the unique path in $T$ connecting $c_{0}$ to $p_0$, and let $\bw_{\lambda} = \widetilde{W}_{0}(\omega_{\lambda}) \in \PL_0(V)$. Then $(\bw_\lambda, \bb_\lambda) \in \Kite(V)$ is a kite and we define 
\begin{align}
    f: L \to \PL_1(V), \qquad \lambda \mapsto (\bw_\lambda, \bb_\lambda). 
\end{align}
It is then clear from the construction that $\delta \circ f(\lambda) = W_{0}(\omega_\lambda \concat \phi_\lambda \concat \omega_\lambda^{-1}) = W_{0} \circ \rho(\lambda)$. We therefore obtain the desired morphism $\bW$. 
\begin{corollary} \label{cor:def_W}
    Let $C$ be a 2-dimensional connected PLSC in $V$. The maps $f$ and $W_0$ above induce a morphism of crossed modules $\bW = (W_1, W_0): \bpi(C, C_1) \to \cmPL(V)$ such that $f = W_1 \circ \eta$, where $\eta: L \to \pi_2(C, C_1)$ is the map sending each 2-cell to its generator. 
\end{corollary}

\subsection{Proof of Injectivity} \label{Injectivitysection}
In this section, we will prove~\Cref{thm:PL_ssig_injectivity}. Throughout this section, we fix $\bX \in \PL_1^{\cl}(V)$ such that $\SigPL(\bX) = 0$. The proof consists of three main steps. 
\begin{enumerate}
    \item Construct an appropriate representative $r(\bX) = (X_1, \ldots, X_k)$ which induces a \emph{compatible} PLSC $C$ (\Cref{def:compatible}). Then construct an element $Y \in \pi_2(C)$ such that $W_1(Y) = \bX$. 
    
    \item The trivial signature condition will provide a \emph{matching} of kites of opposite orientation, implying that under the Hurewicz map $\hurewicz: \pi_2(C) \to H_2(C)$, we have $\hurewicz(Y) = 0$.

    \item Construct a simply connected PLSC $\hC$ such that $C \hookrightarrow \hC$, and apply the Hurewicz theorem to show that $W_1(Y) = 0$. 
\end{enumerate}
\medskip

\subsubsection{Construct Simplicial Model}
Our first step is to construct a simplicial model for $\bX$. This is a PLSC $C$ whose fundamental crossed module contains an element $Y$ which gets sent to $\bX$ under the map $W_{1}$ constructed in \Cref{cor:def_W}. Such a simplicial model will be induced by a \emph{triangulated representative} of $\bX$. 
In order to define this representative, we begin by defining the set of \emph{marked kites} \label{pg:marked_kites}
\begin{align}
    \Kite^\times(V) \coloneqq \FMon(V) \times \planarloop(V),
\end{align}
where $\FMon(V)$ is the free monoid generated by $V$. In other words, these are kites for which we have additionally chosen the data of a representative of the tail path. There are naturally defined surjective monoid homomorphisms 
\begin{align}
    \FMon(\Kite^\times(V)) \to \FMon(\Kite(V)) \to \PL_{1}(V). 
\end{align}

\begin{definition} \label{def:marked_kite}
A \emph{marked representative} (or simply \emph{representative}) of an element $\bX \in \PL_{1}(V)$ is a lift of this element to $\FMon(\Kite^\times(V))$. A marked kite $(\bw, \bb) \in \Kite^\times(V)$ is \emph{triangular} if $\bb = (b_1, b_2, b_3)_{\min}$. Therefore, a representative $(X_{1}, ..., X_{k}) \in \FMon(\Kite^\times(V))$ of $\bX$ is said to be \emph{triangulated} if each marked kite $X_{i}$ is triangular. %
\end{definition}

We now explain how to construct a PLSC $C$ given the data of a triangulated representative. First, let $X = (\bw, \bb) \in \Kite^{\times}(V)$ be a marked triangular kite. Therefore $\bw = (w_{1}, ..., w_{m}) \in \FMon(V)$ and $\bb = (b_1, b_2, b_3)_{\min} \in \PL_{0}^{\cl}(V)$. Define $\hw_0 = 0 \in V$, and given $1 \leq k \leq m$, define $\hw_{k} = \sum_{i = 1}^{k} w_{i}.$ Finally, define 
\begin{align}
    \hb_{1} = \hw_m + b_1 \andd \hb_2 = \hb_1 + b_2.
\end{align}
The \emph{piecewise linear simplicial complex associated to $X$}, denoted $\Delta(X)$, is given by 
\begin{align}
        \Delta_0(X) &= \{\hw_0, \hw_1, \ldots, \hw_m, \hb_1, \hb_2\}\\
        \Delta_1(X) & =  \{[\hw_i, \hw_{i+1}]\}_{i=0}^{m-1} \cup \{ [\hw_m, \hb_1], [\hb_1, \hb_2], [\hb_2, \hw_m]\} \\
        \Delta_2(X) & = \{ [\hw_m, \hb_1, \hb_2]\}.
\end{align}
More precisely, we take the set of vertices to be $C_{0} = \Delta_0(X) \subset V$, with the induced order. Since this is a subset of $V$, any repetitions are automatically deleted. The simplicial complex is the union of $\Delta_0(X), \Delta_1(X),$ and $\Delta_2(X)$ in the power set of $C_{0}$. As a result, any repetitions are removed, regardless of orientation or ordering in the above description. The \emph{2-simplex associated to $X$} is the unique 2-simplex $\sigma_X = [\hw_m, \hb_1, \hb_2]$ of $\Delta(X)$. Note that $\sigma_X$ is non-degenerate and that $\Delta(X)$ is connected.

Next, let $r(\bX) = (X_{1}, ..., X_{k}) \in \FMon(\Kite^\times(V))$ be a triangulated representative. The \emph{PLSC associated to $r(\bX)$}, denoted $\Delta(r(\bX))$, is given by the union \label{pg:DeltabX}
\begin{align}
\Delta(r(\bX)) \coloneqq \bigcup_{j=1}^k \Delta(X_j).
\end{align}
More precisely, the set of vertices $C_{0} \subset V$ is the union of the vertices from each $\Delta(X_{j})$, with the ordering such that the new vertices of $\Delta(X_{j+1})$ come after those of $\Delta(X_{j})$. The simplicial complex is then the union of each $\Delta(X_{j})$ in the power set of $C_{0}$. Note that $\Delta(r(\bX))$ is connected because each complex $\Delta(X_{i})$ is connected and contains the origin as a vertex. 

The following definition will play an important role when we `match' the kites in a representative of an element $\bX \in \PL_{1}(V)$. 
\begin{definition}
    Let $\bX \in \PL_1(V)$. 
    A \emph{compatible representative} is a triangulated representative $r(\bX) \in \FMon(\Kite^\times(V))$ such that $\Delta(r(\bX))$ is compatible in the sense of~\Cref{def:compatible}.
\end{definition}
\begin{remark}
    Recall from~\Cref{rem:nondegone} that a $1$-simplex in a PLSC is always non-degenerate. Furthermore, as noted above, the 2-simplex associated to a triangular kite is also non-degenerate. Therefore, the simplicial complex $\Delta(r(\bX))$ is automatically non-degenerate. 
\end{remark}

We prove the following in~\Cref{apx:triangulations_matching} by carefully considering subdivisions.

\begin{theorem} \label{thm:existence_compatible_representative}    
    There exists a compatible  representative $r(\bX) \in \FMon(\Kite^{\times}(V))$ of every $\bX \in \PL_1(V)$.
\end{theorem}

Now let $\bX \in \PL_{1}(V)$, fix a compatible representative $r(\bX)$, and let $C \coloneqq \Delta(r(\bX))$ be the associated PLSC. Equip $C$ with the basepoint $c_{0} = 0 \in V$. All homotopy groups will be taken with respect to this basepoint. By~\Cref{cor:def_W}, we may construct a morphism of crossed modules $\bW = (W_1, W_0): \bpi(C, C_1) \to \cmPL(V)$.
Since $C$ is compatible and non-degenerate, the map $W_{0}$ is injective by~\Cref{lem:W0_injective}. 
Next, we lift $\bX$ to an element of $\bpi(C, C_1)$. 

\begin{proposition} \label{prop:def_Y}
    There exists an element $Y \in \pi_2(C, C_1)$ such that $W_1(Y) = \bX$. Furthermore, if $\SigPL(\bX) = 0$, then $Y \in \pi_2(C)$. 
\end{proposition}
\begin{proof}
    The representative $r(\bX)= (X_{1}, ..., X_{k})$ allows us to factor $\bX$ into a product of kites in $\PL_{1}(V)$. Hence, it suffices to show that each $X_{i} = (\bw_{i}, \bb_{i})$ is in the image of $W_{1}$. Each marked triangular kite $X_{i}$ has an associated 2-simplex $\sigma_{X_{i}} \in C$ and hence determines a generator $\eta(\sigma_{X_{i}}) \in \pi_{2}(C,C_1)$. Under the map $W_{1}$ this generator is sent to 
    \begin{align}
        W_{1}(\eta(\sigma_{X_{i}})^{s_i}) = \bv_{i} \gt \bb_i,
    \end{align}
    where $s_{i} = \pm 1$ is an appropriate sign and where $\bv_{i} = \widetilde{W}_{0}(\nu_{i})$, for $\nu_{i} \in \Pi_1(C_1, C_0)$ a path in $C_1$ connecting $c_0$ to the basepoint $p_i$ of $\bb_i$. The element $\bw_{i}$ can also be realized as $\bw_{i} = \widetilde{W}_0(\omega_{i})$, for $\omega_{i} \in \Pi_1(C_1, C_0)$ a path in $C_1$ connecting $c_0$ to $p_i$. Therefore, $\omega_i \concat \nu_i^{-1} \in \pi_1(C_1)$ and 
    \begin{align}
        W_{1}( (\omega_i \concat \nu_i^{-1}) \gt \eta(\sigma_{X_{i}})^{s_i}) = (\bw_i \concat \bv_{i}^{-1}) \gt ( \bv_i \gt \bb_i) = X_i. 
    \end{align}
    This shows that a preimage $Y$ exists. By the injectivity of the path signature and $\SigPL(\bX) = 0$, we have $\delta(\bX) = 0$. Thus, $W_0(\partial Y) = 0$, and so $\partial Y = 0$ by the injectivity of $W_{0}$. Hence $Y \in \pi_2(C)$.
\end{proof}

\subsubsection{Matching kites}
Consider the Hurewicz map
\begin{align} \label{eq:hurewicz}
    \hurewicz : \pi_2(C) \to H_2(C).
\end{align}
Applying it to the element $Y \in \pi_2(C)$ constructed in~\Cref{prop:def_Y}, we obtain an element $\hurewicz(Y) \in H_2(C)$. The next step is to show that this element is $0$. 

Let $L$ denote the set of 2-simplices of $C$. The cellular chain complex of $C$ is given by 
\[
C_2(C) = \bigoplus_{\lambda \in L} \mathbb{Z}\lambda. 
\]
Given the representative $r(\bX) = (X_{1}, ..., X_{k})$, each marked triangular kite $X_{i}$ has an associated 2-simplex $\sigma_{X_{i}}$ which is an element of $L$. This defines a function $\alpha: [k] \to L$ sending the index $i$ to $\sigma_{X_{i}}$. Since the ordering on the triangular loop in $X_i$ might not agree with the ordering of simplices in $L$, we must also record this information as orientation data. This is a function $s: [k] \to \{ \pm 1 \}$ such that $s(i) = +1$ if and only if the ordering on the triangular loop in $X_{i}$ matches the fixed order on $\alpha(i)$. We call the pair of functions $(\alpha, s)$ the \emph{simplex mapping} of $r(\bX)$ and we note that \label{pg:simplex_mapping}
\begin{align}
    \hurewicz(Y) = \sum_{i = 1}^{k} s(i) \alpha(i). 
\end{align}
The pair $(\alpha, s)$ is a \emph{simplex matching} if $\hurewicz(Y) = 0$, since this means that the simplices in $r(\bX)$ are ``matched up'' in pairs of opposite orientation.
By hypothesis, the signature of the realization of $\bX$ is trivial, $\Sig(\realization_1(\bX)) = 0$, and by~\Cref{cor:trivial_ssig_current}, this implies that 
\begin{align} \label{eq:realization_trivial_integrals}
    \int_{\realization_1(\bX)} \omega = 0
\end{align}
for all $\omega \in \Omega^2_c(V)$. We will use this fact to show that $(\alpha,s)$ is indeed a simplex matching. 

\begin{proposition} \label{prop:value_Hur_vanishes}
    Let $Y \in \pi_2(C)$ be the element from~\Cref{prop:def_Y}. Then $\hurewicz(Y) = 0$.
\end{proposition}
\begin{proof}
    Let $\omega \in \Omega^2_{c}(V)$ be a compactly supported form. Then 
    \begin{align}
        \int_{\realization_1(\bX)} \omega = \int_{\realization_1(W_1(Y))} \omega = \sum_{i = 1}^{k} s(i) \int_{\alpha(i)} \omega. 
    \end{align}
    Given a 2-simplex $\sigma \in L$, let $\omega_{\sigma} \in \Omega_{c}^2(V)$ be a compactly supported form with the property that $\int_{\sigma} \omega_{\sigma} = 1$ and whose support is disjoint from all other simplices. Such a form exists by~\Cref{lem:disjoint_2_forms}. Then  
    \begin{align}
        \sum_{i = 1}^{k} s(i) \int_{\alpha(i)} \omega = \sum_{i \ : \ \alpha(i) = \sigma} s(i). 
    \end{align}
    This sum vanishes by ~\eqref{eq:realization_trivial_integrals}. Therefore $\hurewicz(Y) = 0$. 
\end{proof}

\subsubsection{Show that $W_1(Y) = 0$.}

Our objective now is to apply the Hurewicz theorem to show that $W_1(Y) = 0$. While $C$ may not be simply-connected in general, we can add 2-cells to kill off the fundamental group.

\begin{lemma} \label{lem:cw_complex_with_trivial_pi1}
    Let $C$ be a 2-dimensional, connected PLSC in $V$. There exists a 2-dimensional PLSC $\hC$ such that $C \subset \hC$ and $\pi_1(\hC) = 0$.
\end{lemma}
\begin{proof}
    Let $C_{0} = \{p_0, ..., p_{n}\}$ be the set of vertices of $C$ and let $E$ be the set of 1-simplices. Choose a point $x \in V - C_{0}$ such that for each edge $\epsilon = [p_i, p_j] \in E$ the triple $\{ x, p_{i}, p_{j} \}$ is not contained in a line. For each vertex $p_i \in C_{0}$, define a new 1-simplex $\hepsilon(p_{i}) = [x, p_i]$, and for each edge $\epsilon = [p_i, p_j] \in E$, define a new 2-simplex $\hlambda(\epsilon) = [x, p_i, p_j]$. Now define a new PLSC $D$ with vertex set $D_0 = C_{0} \cup \{ x \}$, set of 1-simplices $E \cup \{ \hepsilon(p_{i}) \}_{p_i \in C_0}$, and set of 2-simplices $\{ \hlambda(\epsilon) \}_{\epsilon \in E }$. Extend the order on $C_{0}$ to $D_0$ so that $x$ comes after all other vertices. By construction, $D$ is contractible. 
    
    Now define $\hC$ to be the union of $C$ and $D$. This is a PLSC that contains $C$ and $D$ as subcomplexes. Furthermore, $C$ and $D$ intersect along $C_1$, the $1$-skeleton of $C$. Therefore, by the van Kampen theorem, $\pi_1(\hC) = 0$. 
\end{proof}

Next, we prove a general relationship between the kernels of $W_1$ and the Hurewicz map which will imply our main injectivity result.

\begin{proposition} \label{prop:kerW1_equal_hurewicz}
    Let $C$ be a 2-dimensional connected, compatible, and non-degenerate PLSC in $V$. Let $W_1: \pi_2(C, C_1) \to \PL_1(V)$ be the homomorphism defined in~\Cref{cor:def_W}, and let $\hurewicz : \pi_2(C) \to H_2(C)$ be the Hurewicz map from~\eqref{eq:hurewicz}. Then %
    \begin{align}
        \ker(\hurewicz) = \ker(W_1). 
    \end{align}
\end{proposition}
\begin{proof}
    First, we show $\ker(\hurewicz) \subset \ker(W_1)$. Let $\hC$ be the PLSC from~\Cref{lem:cw_complex_with_trivial_pi1} with inclusion map $\iota: C \hookrightarrow \hC$.
    The map $\iota$ induces a map $\bpi(\iota): \bpi(C, C_1) \to \bpi(\hC, \hC_1)$. 
    Applying~\Cref{cor:def_W} to $\hC$, we have another morphism of crossed modules
    $\widehat{\bW} = (\hW_1, \hW_0) : \bpi(\hC, \hC_1) \to \cmPL(V)$ such that the following diagram commutes,

    \begin{equation}
        \begin{tikzcd}
            \bpi(C, C_1) \ar [rr, "\bpi(\iota)"] \ar[dr, swap,"\bW"] & &\bpi(\hC, \hC_1) \ar[dl, "\widehat{\bW}"]\\
            & \cmPL(V) &.
        \end{tikzcd}
    \end{equation}
    Indeed, examining the construction of $W_0$ in \eqref{eq:W0_def}, it is clear that $\hW_0 \circ \pi_1(\iota) = W_0$. The construction of $W_1$ leading up to~\Cref{cor:def_W} depends on the choice of a spanning tree $T \subset C_1$. Therefore, to ensure that $\hW_1 \circ \pi_2(\iota) = W_1$, we choose the spanning tree $\widehat{T} \subset \hC_1$ to be an extension of $T$. This implies that the chosen generators of $\pi_2(C, C_1)$ are sent to the chosen generators of $\pi_2(\hC, \hC_1)$ under the map $\pi_2(\iota)$.

    Next, by naturality of the Hurewicz map, the following diagram commutes.
    \begin{equation}
        \begin{tikzcd}
            \pi_2(C) \ar [r, "\pi_2(\iota)"] \ar[d, swap,"\hurewicz"] & \pi_2(\hC) \ar[d,"\hurewicz"]\\
            H_2(C) \ar[r, swap, "H_2(\iota)"]& H_2(\hC),
        \end{tikzcd}
    \end{equation}
    Let $Y \in \ker(\hurewicz) \subset \pi_2(C)$. Because $\hurewicz: \pi_2(\hC) \to H_2(\hC)$ is an isomorphism by the Hurewicz theorem, and $\hurewicz(Y) = 0$, we have $\pi_2(\iota)(Y) = 0$. Therefore, this implies that $W_1(Y) = \hW_1 \circ \pi_2(\iota) (Y) = 0$. \medskip

    It remains to show that $\ker(W_1) \subset \ker(\hurewicz)$. 
    Suppose $Y \in \ker(W_1)$. Then, we have
    \begin{align}
        W_0 \circ \partial(Y) = \delta \circ W_1(Y) = 0.
    \end{align}
    Since $C$ is compatible and non-degenerate, $W_{0}$ is injective by~\Cref{lem:W0_injective}. Hence it follows that $Y \in \pi_2(C)$. Furthermore, the signature $\SigPL \circ W_1(Y) = 0$ is trivial. Then, the same argument as~\Cref{prop:value_Hur_vanishes} shows that $\hurewicz(Y) = 0$.

\end{proof}

\begin{remark}
    In light of~\Cref{prop:embedding_free_group}, a natural question is whether one can embed free crossed modules into the piecewise linear crossed module $\cmPL(V)$. In contrast to~\Cref{prop:embedding_free_group}, \Cref{prop:kerW1_equal_hurewicz} shows that this is not possible in general for crossed modules, as the kernel $\ker(W_1)$ may be nontrivial. 
\end{remark}

Now, we can prove our main injectivity result.

\begin{proof}[Proof of~\Cref{thm:PL_ssig_injectivity}]
    By~\Cref{lem:reduce_to_kernel}, it suffices to show that $\SigPL : \PL_1^{\cl}(V) \to \com{K}_1(V)$ is injective. Let $\bX \in \PL_1^{\cl}(V)$ such that $\SigPL(\bX) = 0$. Let $r(\bX)$ be a compatible representative of $\bX$, $C = \Delta(r(\bX))$, and $Y \in \pi_2(C)$ be element from~\Cref{prop:def_Y} such that $W_1(Y) = \bX$. By~\Cref{prop:value_Hur_vanishes}, we have $\hurewicz(Y) = 0$, and finally by~\Cref{prop:kerW1_equal_hurewicz}, $W_1(Y) = 0$, so $\bX = 0$. 
\end{proof}

\subsection{Equivalent Conditions}

In this section, we complete the generalization of~\Cref{thm:path_thin_homotopy} for piecewise linear surfaces, and consider appropriate extensions of the remaining definitions of thin homotopy. We begin by defining the class of geometric surfaces that we will consider. 

\begin{definition} \label{def:smooth_pl_surface}
    A \emph{smooth piecewise linear surface} is a surface $\bX \in C^1([0,1]^2, V)$ such that 
    \begin{itemize}
        \item only the top boundary is non-trivial, $\partial_l \bX = \partial_b \bX = \partial_r \bX = 0$, and
        \item there exists a compatible PLSC $T$ in $\R^2$ such that $|T| = [0,1]^2$, and the restriction of $\bX$ to any $2$-simplex $\sigma \in T_2$ is the composition of a smooth reparametrization $\psi_\sigma: |\sigma| \to |\sigma|$ with sitting instants, and an affine linear function $f_\sigma: |\sigma| \to V$,
        \begin{align}
            \bX|_{|\sigma|} : |\sigma| \xrightarrow{\psi_\sigma} |\sigma| \xrightarrow{f_\sigma} V.
        \end{align}
    \end{itemize}
    We denote the space of smooth piecewise linear surfaces by $C_{\PL}^1([0,1]^2, V)$. 
\end{definition}

The following lemma shows that $C^1_{\PL}([0,1]^2, V)$ consists of surfaces which lie in the thin homotopy classes defined by the realization of $\PL_1(V)$. 

\begin{lemma} \label{lem:smooth_pl_representative}
    Let $\bX \in C^1_{\PL}([0,1]^2, V)$. There exists some $\bZ \in \PL_1(V)$ such that $\bX$ is in the thin homotopy class of $\realization_1(\bZ)$. 
\end{lemma}
\begin{proof}
    Let $P = \Big( (1,0), (0, 1), (-1, 0), (0, -1)\Big) \in \PL_0(\R^2)$ be the boundary of the unit square. Let $T$ be a compatible PLSC in $\R^2$ which satisfies the conditions of~\Cref{def:smooth_pl_surface}, and note that it is a compatible triangulation of $[0,1]^2$ in the sense of~\Cref{def:triangulation}. Then, viewing $(\emptyset, P) \in \Kite^\times(\R^2)$ as a kite, the proof of~\Cref{lem:triangulation_decomposition_of_kite} implies there exists a compatible representative $\bY= (Y_1, \ldots, Y_k) \in \FMon(\Kite^\times(\R^2))$ of $(\emptyset, P) \in \PL_1(\R^2)$. Then, applying the maps $f_\sigma$ from~\Cref{def:smooth_pl_surface} to $\bY$, we obtain an element $\bZ = (Z_1, \ldots, Z_k) \in \PL_1(V)$. 
    Finally, by~\Cref{lem:unique_thin_surface_2d}, the identity map $I: [0,1]^2 \to \R^2$ is in the thin homotopy class of $\realization_1(\bY) \in \tau_2(\R^2)$, so $\bX$ is in the thin homotopy class of $\realization_1(\bZ) \in \tau_2(V)$.
\end{proof}

We now propose generalizations of the various definitions of thin homotopy for paths to the setting of smooth piecewise linear surfaces. The rank condition \ref{R1} was generalized in~\Cref{def:2d_thin_homotopy}, the analytic condition \ref{A1} was generalized in~\Cref{rem:analyticcond}, and the signature condition \ref{S1} is generalized using the surface signature from~\Cref{def:ssig}. Here, we comment on the remaining definitions before stating our generalization of ~\Cref{thm:path_thin_homotopy}. 

First, we generalize the word condition \ref{W1}. Given a surface $\bX \in C^1_{\PL}([0,1]^2, V)$, there exists a representative $r(\bX) \in \FMon(\Kite^\times(V))$ by \Cref{lem:smooth_pl_representative}. We say that the surface $\bX$ is \emph{word reduced} if the representative $r(\bX)$ is trivial in $\PL_1(V)$. This is analogous to the condition that the transfinite word associated to a path is reducible to the trivial word. Next, to generalize the holonomy condition \ref{H1}, we must consider a class of crossed modules that are analogous to semisimple Lie groups. We use the following condition.

\begin{definition}
    A crossed module $\cmg = (\delta: \fg_1 \to \fg_0, \gt) \in \XLie$ is \emph{holonomy nondegenerate} if
    \begin{enumerate} 
        \item there exists a semisimple Lie algebra $\fh$ such that $\fh \subset \im(\delta)$, and
        \item $\ker(\delta)$ is nontrivial. 
    \end{enumerate}
    A crossed module of Lie groups $\cmG \in \XLGrp$ is \emph{holonomy nondegenerate} if its associated crossed module of Lie algebras $\cmg \in \XLie$ is holonomy nondegenerate.
\end{definition}
\begin{remark}
    In the above definition, the first condition is used to ensure that the surface holonomy is rich enough to distinguish boundary paths, while the second condition allows us to distinguish between closed surfaces.
\end{remark}

To formulate the image condition \ref{I1}, the naive generalization of the definition for paths suggests that two surfaces $\bX$ and $\bY$ are equivalent if there is a homotopy that is constrained to lie in the union of the images of $\bX$ and $\bY$. However, as we saw in~\Cref{cor:rp2_nonexist}, this definition does not work because of nonlocal cancellations. Instead, we propose to use the transitive closure of this relation. Finally, our generalization of the factorization condition \ref{F1} will make use of factorizations through 2-dimensional CW complexes. However, simply requiring that a surface $\bX$ factors through such a complex is not enough. We must also require that the factorization is trivial in homology.  

\begin{theorem} \label{thm:surface_thin_homotopy}
    Let $\bX \in C^1_{\PL}([0,1]^2, V)$ be a piecewise linear surface. It is \emph{thinly null-homotopic} if any of the following equivalent definitions hold:
    \begin{enumerate}
        \item[\mylabel{W2}{(\textbf{W2})}] \textbf{Word Condition.} There is a marked representative $r(\bX) \in \FMon(\Kite^\times(V))$ of $\bX$ which is trivial in $\PL_1(V)$. 
        \item[\mylabel{H2}{(\textbf{H2}$_{\cmG}$)}] \textbf{Holonomy Condition.} For a holonomy nondegenerate $\cmG = (\delta: G_1 \to G_0, \gt)$, the surface holonomy along $\bX$ of every smooth fake-flat $\cmg$-connection $(\con, \Con)$ is trivial. 
        \item[\mylabel{R2}{(\textbf{R2})}] \textbf{Rank Condition.} There exists a smooth thin homotopy $H: [0,1]^3 \to V$, as defined in~\Cref{def:2d_thin_homotopy} from $\bX$ to the constant path at $0$.
        \item[\mylabel{I2}{(\textbf{I2})}] \textbf{Image Condition.} There exists a smooth homotopy $H: [0,1]^3 \to V$ between $\bX$ and the constant surface at the origin such that it satisfies \ref{I1} on the boundary of $[0,1]^2$ and such that
        \begin{align}
            \im(H) \subset \bigcup_{i=0}^k \im(H([0,1]^2 \times \{t_i\})),
        \end{align}
        for a collection of times $0 = t_0 < t_1 \ldots < t_k \leq 1$.
        \item[\mylabel{F2}{(\textbf{F2})}] \textbf{Factorization Condition.} The surface $\bX$ has trivial boundary and, after modifying the boundary using a thin homotopy, there exists a 2-dimensional CW complex $C$ and a map $\Theta: C \to V$, which is smooth when restricted to each cell, such that $\bX$ factors as
        \begin{align}
            \bX: [0,1]^2 \xrightarrow{q} S^2 \xrightarrow{f} C \xrightarrow{\Theta} V
        \end{align}
        and such that $H_2(f)([S^2]) = 0$, where $q: [0,1]^2 \to S^2$ is a map which covers the sphere and sends the boundary to a basepoint $* \in S^2$.
        \item[\mylabel{A2}{(\textbf{A2})}] \textbf{Analytic Condition.} The surface $\bX$ has trivial boundary, and for all compactly supported $\omega \in \Omega^2_c(V)$, we have $\int_\bX \omega = 0$. 
        \item[\mylabel{S2}{(\textbf{S2})}] \textbf{Surface Signature Condition.} The surface signature $\Sig$ of $\bX$ is trivial, $\Sig(\bX) = 1$.
    \end{enumerate}
\end{theorem}
\begin{proof}
    We have proved the equivalence of \ref{W2}, \ref{R2}, and \ref{S2} in~\Cref{thm:PL_ssig_injectivity}. Moreover, \Cref{cor:trivial_ssig_current} shows the equivalence between \ref{A2} and \ref{S2} for the case of closed surfaces. However, since \ref{S2} implies that $\bX$ is closed, the equivalence holds in general. \medskip

    Now we prove the equivalence of \ref{H2} for a fixed holonomy nondegenerate $\cmG = (\delta: G_1 \to G_0, \gt) \in \XLGrp$ with associated $\cmg = (\delta: \fg_1 \to \fg_0, \gt) \in \XLie$. Because surface holonomy is invariant with respect to thin homotopies~\cite{baez_higher_2004,schreiber_smooth_2011,martins_surface_2011} as defined by \ref{R2}, it follows that \ref{R2} $\Rightarrow$ \ref{H2}. Now, suppose \ref{H2} holds. 
    Let $\fh \subset \im(\delta)$ be the semi-simple subalgebra. Choose a linear section $s: \fh \to \fg_1$. Given any 1-connection $\con \in \Omega^1(V, \fh)$, we can consider the 2-connection $\Con = s(\curv^{\con}) \in \Omega^2(V, \fg_1)$ which is fake-flat by definition. Because surface holonomy is a morphism of crossed modules by~\Cref{thm:sh_cm_morphism}, we have
    \begin{align}
        \delta\Hol^{\con, \Con}(\bX) = \hol^{\con}(\partial \bX).
    \end{align}
    Hence, by the assumption of \ref{H2}, this implies that $\hol^{\con}(\partial\bX)$ is trivial for all $\con \in \Omega^1(V, \fh)$. Thus, \ref{H1} holds for $G = H$, where $H$ is a Lie group with Lie algebra $\fh$. This implies that the boundary of $\bX$ is trivial, so $\bX$ is closed. Next, let $\fa = \ker(\delta: \fg_1 \to \fg_0)$, which is nontrivial by assumption. Let $\mathbb{R} \subseteq \fa$ be a $1$-dimensional subalgebra and consider an abelian 2-connection $(0, W^{\ab})$, where $W^{\ab} \in \Omega^2(V, \mathbb{R}) \subseteq \Omega^2(V, \fa)$ is valued in the subalgebra. Let $A$ be the integration of $\fa$. The subalgebra $\mathbb{R}$ integrates to a subgroup which is either $\mathbb{R}$ or $S^1$. Because the connection is abelian, the surface holonomy (\Cref{def:sh}) is given by either
    \begin{align}
        \Hol^{0, W^{\ab}}(\bX) = \int_{\bX} W^{\ab} \in \mathbb{R} \,\,  \quad \text{or} \quad \Hol^{0, W^{\ab}}(\bX) = \exp\left(i\int_{\bX} W^{\ab}\right) \in S^1.
    \end{align}
    Because $W^{\ab}$ is an arbitrary 2-form, either of these will imply \ref{A2}.\medskip

    Next we consider the factorization condition \ref{F2}. Suppose \ref{S2} holds. Our proof of~\Cref{thm:PL_ssig_injectivity} that \ref{S2} $\Rightarrow$ \ref{W2} provided the desired factorization of $\bX$ through a PLSC $C$, where $\Theta$ is linear when restricted to each simplex. Furthermore in~\Cref{prop:def_Y}, we constructed a $Y \in \pi_2(C)$ such that $W_1(Y) = \bZ$, and thus $H_2(f)([S^2])=\hurewicz(Y) = 0$ by~\Cref{prop:value_Hur_vanishes}. Now, suppose \ref{F2} holds. This implies that $\bX$ is closed. Given a $2$-form $\omega \in \Omega^2(V)$, integrating the pullback $\Theta^{*}(\omega)$ along the cells of $C$ defines a map $\phi: H_2(C, \mathbb{Z}) \to \mathbb{R}$. Hence, the fact that $H_{2}(f)([S^2]) = 0$ implies that \ref{A2} holds. \medskip

    Next, we consider the image condition \ref{I2}, where it is immediate that \ref{I2} $\Rightarrow$ \ref{R2}. Now, suppose \ref{S2} holds, and once again we consider the proof of~\Cref{thm:PL_ssig_injectivity}. There, we show that $\bX$ factors through a PLSC $C$, and add 2-simplices $\hL = \{\hlambda_1, \ldots, \hlambda_m\}$ to build a simply-connected CW complex $\hC$ in~\Cref{lem:cw_complex_with_trivial_pi1}. The result of~\Cref{thm:PL_ssig_injectivity} implies that there is a smooth thin homotopy $h$ satisfying \ref{R2} which is contained in the image of $\hC$. Suppose $\lambda_i = [p_0, p_1, p_2]$. By choosing some $\bw_i \in \PL_0(V)$ which connects the basepoint $p_0$ of $\hlambda_i$ to the origin, we can interpret each $\hlambda_i \in \hL$ as a kite $\hlambda_i \in \Kite(V)$. Using~\Cref{lem:surface_realization_planar_loop} and translating the basepoint, we can realize each $2$-simplex $\hlambda_i$ as a surface $\bY_i \in C^1([0,1]^2, V)$. We can incorporate the tail paths by using the formula for the action given in~\eqref{eq:thin_group_action} to obtain a surface $\bX_i \in C^1([0,1]^2, V)$ defined by
    \begin{align}
        \bX_i \coloneqq \sigma^{\bw_i} \concat_h \bY \concat_h ((\bw_i)_1 + \sigma^{\bw_i^{-1}}).
    \end{align}
    Let $\bx_i \concat \bx_i^{-1}$ denote the top boundary of $\bX_i \concat_h \bX_i^{-1}$, and let $H_i: [0,1]^2 \to V$ be a thin homotopy of paths from $\bx_i \concat \bx_i^{-1}$ (on the bottom face of $H$) to $0$ (on the top face of $H$). 
    Then, we define
    \begin{align}
        \bF_i = (\bX_i \concat_h \bX_i^{-1}) \concat_v H_i \andd \hbX = \bX \concat_h \bF_1 \concat_h \ldots \concat_h \bF_m,
    \end{align}
    where concatenations are performed from left to right. Here, each $\bF_i$ is a surface which represents the fold $\bX_i \concat_h \bX_i^{-1}$ such that the boundary of $\bF_i$ is trivial. Note that the image of $\hbX$ is $|\hC|$ by definition, and since $\bX$ and $\hbX$ differ only by folds, there exists a homotopy between them which is contained in the image of $\hbX$, and whose boundary is contained in the image of $\partial \bX$. Then, the image of the sequence of homotopies
    \begin{align}
        \bX \to \hbX \to \bX \xrightarrow{h} 0,
    \end{align}
    is contained in the image of $\hbX$, and thus \ref{I2} is satisfied.     
\end{proof}

\section{Thin Null Homotopy and Group Homology} \label{sec:group_homology}

In this section, we build a connection between thinly null homotopic surfaces and the group homology $H_3(G)$. On the one hand, this allows us to geometrically interpret $H_3(G)$ in terms of surfaces; on the other, it allows us to further classify thinly null homotopic behavior. 
Given a compatible PLSC $C$ in $V$, we use~\Cref{cor:def_W} to obtain a morphism
\begin{align}
    \bW = (W_1, W_0): \cmfundgroup(C) \to \cmPL(V).
\end{align}
Recall from~\Cref{prop:kerW1_equal_hurewicz} that $\ker(\hurewicz)= \ker(W_1)$. 
Now, we aim to characterize $\ker(\hurewicz)$ in terms of group homology. We provide a brief exposition to the required background on group homology, and refer the reader to~\cite{brown_cohomology_1982} for further details. \medskip

\begin{definition}
    Let $G$ be a group and $M$ be a $G$-module. Let $\epsilon: \Z G \to \Z$ be the augmentation map defined by $\epsilon(g) = 1$ for all $g \in G$. The group of \emph{co-invariants} of $M$ is defined by $M_G \coloneqq M/IM$, where $I = \ker(\epsilon)$ is the augmentation ideal generated by $g-1$ for $g \in G$. In other words, $M_G$ is the quotient of $M$ by elements of the form $m - g\cdot m$ for $g\in G$ and $m \in M$. 
\end{definition}

\begin{definition}
    Let $G$ be a group and let $(F_\bullet, d)$ be a projective resolution of $\Z$ over $\Z G$, where
    \begin{align}
        \ldots \to F_n \to \ldots \to F_0 \to \Z \to 0
    \end{align}
    is exact. Then, we define the \emph{group homology of $G$} by
    \begin{align}
        H_n(G) \coloneqq H_n(F_G).
    \end{align}
\end{definition}

The following is an exact sequence, originally due to Hopf~\cite{hopf_fundamentalgruppe_1941}, which relates the homology and homotopy groups of a CW-complex $C$ with the homology of its fundamental group $G = \pi_1(C)$ through the Hurewicz map. We state a version from~\cite{brown_cohomology_1982}.

\begin{theorem}{\cite[Exercise 1, Section 2.5]{brown_cohomology_1982}}
    Let $C$ be an $n$-dimensional CW-complex such that $\pi_i(C) = 0 $ for $ 1 < i < n$ for some $n \geq 2$. Let $G = \pi_1(C)$. Then, the sequence
    \begin{align} \label{eq:hopf_exact_sequence}
        0 \to H_{n+1}(G) \to (\pi_n(C))_G \xrightarrow{\hurewicz_G} H_n(C) \to H_n(G) \to 0
    \end{align}
    is exact, where $\hurewicz_G: \pi_n(C)_G \to H_n(C)$ is the co-invariants functor applied to the Hurewicz map and $H_n(C)$ is equipped with the trivial $G$-action. 
\end{theorem}

In our current setting of a 2-dimensional CW complex $Z$, the connectivity hypothesis is trivially satisfied, and thus we obtain the exact sequence
\begin{align} \label{eq:hopf_exact_seq_our_case}
    0 \to H_3(G) \xrightarrow{\phi} (\pi_2(Z))_G \xrightarrow{\cH_G} H_2(Z) \to H_2(G) \to 0.
\end{align}
The kernel of $\cH_G : (\pi_2(Z))_G \to H_2(Z)$ is exactly $\phi(H_3(G))$. Thus, we obtain the following characterization of $\ker(W_1)$. 

\begin{proposition} \label{prop:W1_kernel_decomposition}
    The kernel of $W_1: \pi_2(C) \to \PL_1(V)$ is
    \begin{align} \label{eq:W1_kernel_decomposition}
        \ker(W_1) = I\cdot \pi_2(C) + \Z G \im(\phi),
    \end{align}
    where $I = \ker(\epsilon: \Z G \to \Z)$ is the augmentation ideal.
\end{proposition}
\begin{proof}
    By~\Cref{prop:kerW1_equal_hurewicz}, $\ker(W_1) = \ker(\hurewicz)$. Then, we can factor the Hurewicz map by
    \begin{align}
        \hurewicz : \pi_2(C) \xrightarrow{q} \pi_2(C)_G \xrightarrow{\hurewicz_G} H_2(C).
    \end{align}
    By~\eqref{eq:hopf_exact_seq_our_case}, $\ker(\hurewicz_G) = \im(\phi)$. Then, $\ker(\hurewicz) = q^{-1}(\im(\phi)) = I \cdot \pi_2(C) + \Z G \im(\phi)$. 
\end{proof}

This result allows us to classify thinly null homotopic surfaces whose image lies in an embedding of $C$. Suppose $\bX : [0,1]^2  \to V$ 
factors as 
\begin{align}
    \bX : [0,1]^2 \xrightarrow{q} S^2 \xrightarrow{Y} C \xrightarrow{|\cdot|} V. 
\end{align}
If $\bX$ is thinly null-homotopic, then $W_1(Y) = 0$. Then, depending on how $Y$ is killed based on~\eqref{eq:W1_kernel_decomposition}, the thin homotopy exhibits different behaviors.
\begin{enumerate}
    \item \textbf{(Folds)} $Y = 0 \in \pi_2(C)$: The map $Y$ is null homotopic within its image in $C$. This is the case if $Y$ exhibits only folds.
    \item \textbf{(Nonlocal Path Conjugation)} $Y = 0 \in \pi_2(C)_G$: The map $Y$ is null-homotopic within its image in $C$ up to conjugation by paths $\gamma \in \pi_1(C)$. This is an example of a \emph{nonlocal} cancellation because the thin null-homotopy must move through the simply connected extension $\hC$. 
    \item \textbf{(Nonlocal Surface Cancellation)} $\hurewicz_G (Y) = 0$: The map $Y$ is null homotopic within the simply connected extension $\hC$, and is classified by $H_3(G)$.
\end{enumerate}

\begin{example}
    Consider the example of the group $G = \Z/2$ with the presentation $\langle x \, | \, x^2 \rangle$, where the associated CW complex\footnote{While this CW complex is not a PLSC, we can refine it to a simplicial complex, and construct a piecewise linear map to $\R^n$ for sufficiently large $n$.} is $C = \RP^2$. In this case, we have
    \begin{align}
        H_2(\Z/2) =0, \quad H_3(\Z/2) = \Z/2, \quad H_2(\RP^2) = 0, \quad \pi_2(\RP^2) = \Z.
    \end{align}
    By~\eqref{eq:hopf_exact_seq_our_case}, this implies that $H_3(\Z/2) \cong \pi_2(\RP^2)_{\Z/2} = \Z/2$, and must represent a thinly null-homotopic surface which factors through $\RP^2$. Consider the surface $\bX$ from~\Cref{prop:rp2_example}. Note that the map $Y: S^2 \to \RP^2$ from the factorization~\eqref{eq:rp2_example_factorization} represents a generator of $\pi_2(\RP^2)$, and is also nontrivial when we pass to the co-invariants $\pi_2(\RP^2)_{\Z/2}$. Thus, the nontrivial element in $H_3(\Z/2)$ represents the thinly null homotopic surface $\bX$. 
\end{example}

\appendix
\clearpage
\section{Notation and Conventions}

{\small
\begin{longtable}
    {ccc}
     \toprule
    Symbol & Description & Page\\ \midrule
    \multicolumn{3}{c}{Categories} \\ \midrule
    $\Vect$ & category of finite-dimensional vector spaces  & \pageref{pg:vect}\\
    $\Lie$ & category of Lie algebras & \\
    $\XGrp, \XLGrp$ & category of crossed modules of groups (resp. Lie groups) & \pageref{pg:XGrp}\\
    $\XLie$ & category of crossed modules of Lie algebras & \pageref{pg:XLie} \\
    $\VL$ & comma category $(\id \downarrow \For)$ associated to $\id: \Vect \to \Vect$ and $\For: \Lie \to \Vect$ & \pageref{pg:VL} \\
    $\SG$ & comma category $(\id \downarrow \For)$ associated to  $\id: \Set \to \Set$ and $\For: \Grp \to \Set$ & \pageref{pg:SG} \\
    \midrule \multicolumn{3}{c}{Crossed Modules} \\ \midrule
    $\cmPL(V)$ & piecewise linear crossed module & \pageref{def:cmpl}\\
    $\cmthingroup(V)$ & thin crossed module (thin homotopy and translation equivalence classes) & \pageref{ex:thin_cm}\\
    $\bpi(C, C_1)$ & fundamental crossed module of CW-complex $C$ relative to the 1-skeleton $C_1$ & \pageref{ex:fund_crossed_module} \\
    $\cmk(V), \com{\cmk}(V)$ & Kapranov's free crossed module of Lie algebras and its completion & \pageref{eq:free_cmla_vector_space}, \pageref{pg:com_cmk}\\
    $\com{\cmK}(V)$ & formal integration of $\com{\cmk}(V)$ as a crossed module of groups &\pageref{pg:com_cmK}\\
    \midrule \multicolumn{3}{c}{Differential Forms and Currents} \\ \midrule
    $\Omega^k(V)$ & smooth differential $k$-forms on $V$ & \\
    $\Omega^k_c(V)$ & smooth compactly supported differential $k$-forms on $V$ & \\
    $\poly{\Omega}^k(V), \poly{\Gamma}_k(V)$ & polynomial differential $k$-forms and $k$-currents on $V$ & \pageref{eq:poly_forms}, \pageref{eq:poly_currents} \\
    $\com{\Omega}^k(V), \com{\Gamma}_k(V)$ & completion of polynomial differential $k$-forms and $k$-currents on $V$ & \pageref{eq:poly_forms}, \pageref{eq:poly_currents} \\
    \midrule \multicolumn{3}{c}{Holonomy and Signatures} \\ \midrule
    $\con, \Con$ & connection/2-connection & \pageref{def:con}, \pageref{def:Con}\\
    $\conk, \Conk$ & universal translation-invariant connection/2-connection& \pageref{eq:univ_connection}, \pageref{eq:univ_2con}\\
    $\curv^{\con}$ & curvature of the connection $\con$ & \pageref{eq:1_curvature} \\
    $\curv^{\con, \Con}, \Curv^{\con, \Con}$ & curvature/2-curvature of the 2-connection $(\con, \Con)$ & \pageref{pg:Curv}\\
    $\hol, \Hol$ & path/surface holonomy & \pageref{def:ph}, \pageref{def:sh}\\
    $\sig, \Sig$ & smooth path/surface signature & \pageref{def:psig}, \pageref{def:ssig}\\
    $\sigPL, \SigPL$ & piecewise linear path/surface signature & \pageref{prop:psigpl_nt}, \pageref{thm:unique_surface_signature}\\
    $\realization_0, \realization_1$ & realization of PL paths/surfaces as thin homotopy equivalence classes & \pageref{eq:realization0_nt}, \pageref{prop:cmrealization}\\
    \midrule \multicolumn{3}{c}{Misc} \\ \midrule
    $\planarloop(V)$ & planar piecewise linear loops in $V$ & \pageref{pg:planarloop} \\
    $\Kite(V)$ & set of kites $\Kite(V) \coloneqq \PL_0(V) \times \planarloop(V)$ in $V$ & \pageref{def:kite} \\
    $\Kite^\times(V)$ & set of marked kites $\Kite^\times(V) \coloneqq \FMon(V) \times \planarloop(V)$ & \pageref{pg:marked_kites} \\
    $\Loop(V)$ & piecewise linear loops generated by triangular loops & \pageref{pg:loop}\\
    $\Pair(V)$ & pair groupoid of $V$ & \pageref{pg:pair_gpd}\\
    $\Cone_{\PL}, \Cone$ & piecewise linear and smooth cone map & \pageref{pg:PLcone}, \pageref{pg:smooth_cone}\\
    $\hurewicz$ & Hurewicz map  & \\
    $\bW = (W_1, W_0)$ & morphism of crossed modules $\bW: \bpi(C, C_1) \to \cmPL(V)$ & \pageref{cor:def_W} \\
    $\Delta(r(\bX))$ & PLSC associated to a representative $r(\bX) \in \FMon(\Kite^\times(V))$ of $\bX$ & \pageref{pg:DeltabX} \\
    $(\alpha, s)$ & simplex mapping associated to compatible $r(\bX) \in \FMon(\Kite^\times(V))$ & \pageref{pg:simplex_mapping}\\
    $\concat$ & concatenation of paths / group operation in $\thingroup_2(V)$ & \pageref{pg:path_concat},\pageref{pg:thin_cm_composition}\\
    $\concat_h$, $\concat_v$ & horizontal/vertical concatenation of (thin homotopy classes) of surfaces & \pageref{eq:strict_surface_concatenation}, \pageref{pg:thin_cm_composition}\\
    \bottomrule
\end{longtable}
}
\clearpage

\section{Piecewise Linear Paths and Loops} \label{apx:pl_paths_loops}

\subsection{Minimal Representatives of PL Paths} \label{apxsec:minimal_pl_paths}
In this section, we prove~\Cref{prop:minimal_pl_paths}, which shows that elements in $\PL_0(V)$ have a unique minimal representative. 
This will be done by developing a rewriting theory on the free monoid $\FMon(V)$ generated by $V$. To simplify the notation, we will omit the symbol $\concat$ for monoid multiplication in this section. 
Now we consider rewriting on the free monoid $\FMon(V)$ generated by $V$. We define two rewriting steps corresponding to the relations for $\PL_0(V)$. For arbitrary words $\ba, \bb \in \FMon(V)$ and linearly dependent $v, w \in V$, we define
\begin{align} \label{eq:pl_paths_rewriting}
    \ba(v,w)\bb \xrightarrow{\ref{PL0.1}} \ba(v+w)\bb \andd \ba(0)\bb \xrightarrow{\ref{PL0.2}} \ba\bb.
\end{align}
Given two words $\ba, \bb \in \FMon(V)$, we write $\ba \rightarrow \bb$ for applying exactly one of the rewriting steps above and write $\ba\xrightarrow{*}\bb$ if we can get from $\ba$ to $\bb$ by a sequence of rewriting steps (possibly zero steps). 
Note that removing a $0$ can always be realized by the \ref{PL0.1} rewriting step, unless both $\ba$ and $\bb$ are empty words. 
Let $l(\ba) \in \mathbb{N}$ be the length of the word $\ba$. Note that if $\ba\xrightarrow{*}\bb$ then $l(\bb) \leq l(\ba)$. If this involves a non-zero number of rewrite steps, then $l(\bb) < l(\ba)$, so that rewriting strictly decreases the length of words. The only words for which we cannot apply a rewriting step are those described in~\Cref{prop:minimal_pl_paths}. We will call these \emph{minimal words}. 
It is clear that because rewriting decreases the length of words, the process must eventually terminate. Our goal is to show that it always terminates at the same minimal word. 

\begin{lemma} \label{lem:paths_rewritestep}
    Let $\ba = (v_1, \ldots, v_n) \in \FMon(V)$ and suppose that $\ba \to \bb$ and $\ba \to \bb'$ are two rewriting steps. Then there is a word $\bc$ such that $\bb\xrightarrow{*}\bc$ and $\bb'\xrightarrow{*}\bc$.
\end{lemma}
\begin{proof}
    Since the rewriting step \ref{PL0.2} only applies if the length of $\ba$ is $1$, we can assume that we are applying \ref{PL0.1}. Hence, the two rewriting steps have the form 
    \begin{align}
    (v_{1}, ..., v_{i}, v_{i+1}, ..., v_{j}, v_{j+1}, ..., v_{n}) \to (v_{1}, ..., v_{i}+ v_{i+1}, ..., v_{j}, v_{j+1}, ..., v_{n})
    \end{align}
    and 
    \begin{align}
     (v_{1}, ..., v_{i}, v_{i+1}, ..., v_{j}, v_{j+1}, ..., v_{n}) \to (v_{1}, ..., v_{i}, v_{i+1}, ..., v_{j}+ v_{j+1}, ..., v_{n}).
    \end{align}
    We can assume that $i \leq j$. If $i = j$, then the two rewrite steps are identical. If $i+1<j$ then we take $\bc = (v_{1}, ..., v_{i}+ v_{i+1}, ..., v_{j}+ v_{j+1}, ..., v_{n})$. If $j = i+1$, then each pair $v_{i}, v_{i+1}$ and $v_{i+1}, v_{i+2}$ is contained in a common line. There are two cases here. First, we could have $v_{i+1} = 0$. But then the two rewrites $\bb$ and $\bb'$ are the same. Second, if $v_{i+1} \neq 0$, then $v_{i}, v_{i+1}, v_{i+2}$ are all contained in the line $\mathrm{span}(v_{i+1})$. Then we take $c = (v_{1}, ..., v_{i}+ v_{i+1} + v_{i+2}, ..., v_{n})$.
\end{proof}

\begin{lemma} \label{lem:path_strongrewrite}
    Let $\ba \in \FMon(V)$ and suppose that $\ba\xrightarrow{*}\bb$ and $\ba\xrightarrow{*}\bb'$. Then there is a word $\bc$ such that $\bb\xrightarrow{*}\bc$ and $\bb'\xrightarrow{*}\bc$.
\end{lemma}
\begin{proof}
    We prove this by induction on the length of the word $\ba$. We can assume that both $\ba\xrightarrow{*}\bb$ and $\ba\xrightarrow{*}\bb'$ involve a non-zero number of rewriting steps, since otherwise the claim is obvious. Furthermore, if the first rewrite step is the same, meaning that $\ba \to \ba_{1} \xrightarrow{*}\bb$ and $\ba \to \ba_{1}\xrightarrow{*}\bb'$, then the result follows by induction since $l(\ba_{1}) < l(\ba)$, so that we can apply the induction hypothesis to $\ba_1$. As a result, we may assume the rewrites to have the form 
    \begin{align}
    \ba \to \ba_{1} \xrightarrow{*}\bb, \qquad \text{ and } \qquad \ba \to \ba_{1}' \xrightarrow{*}\bb'.
    \end{align}
    and $l(\ba_{1})< l(\ba)$ and $l(\ba_{1}')< l(\ba)$. By Lemma \ref{lem:paths_rewritestep}, there is a word $\bd$ such that $\ba_{1} \xrightarrow{*} \bd$ and $\ba'_{1} \xrightarrow{*} \bd$. Now $l(\ba_{1}) < l(\ba)$, and we have $\ba_{1} \xrightarrow{*} \bb$ and $\ba_{1} \xrightarrow{*} \bd$. So by the induction hypothesis, there is a word $\be$ such that $\bb \xrightarrow{*} \be$ and $\bd \xrightarrow{*} \be$. Similarly, $l(\ba_{1}') < l(\ba)$, $\ba'_{1} \xrightarrow{*} \bb'$ and $\ba'_{1} \xrightarrow{*} \bd$, so by the induction hypothesis, there is a word $\be'$ such that $\bb' \xrightarrow{*} \be'$ and $\bd \xrightarrow{*} \be'$. Finally, $l(\bd) \leq l(\ba_{1}) < l(\ba)$ and $\bd \xrightarrow{*} \be$ and $\bd \xrightarrow{*} \be'$. So by the induction hypothesis, there is a word $\bc$ such that  $\be \xrightarrow{*} \bc$ and  $\be' \xrightarrow{*} \bc$. Combining the rewriting steps we have 
    \begin{align}
    \bb \xrightarrow{*} \be \xrightarrow{*} \bc \qquad \text{ and } \qquad \bb' \xrightarrow{*} \be' \xrightarrow{*}\bc.
    \end{align}
\end{proof}

\begin{corollary} \label{cor:path_rewrite_minimal}
    If $\ba \xrightarrow{*} \bb$ and $\ba \xrightarrow{*} \bb'$ are rewritings such that $\bb$ and $\bb'$ are minimal words, then $\bb = \bb'$. In other words, any word reduces to a unique minimal word. 
\end{corollary}
\begin{proof}
    By Lemma \ref{lem:path_strongrewrite}, there is a word $\bc$ such that $\bb\xrightarrow{*}\bc$ and $\bb'\xrightarrow{*}\bc$. But $\bb$ and $\bb'$ are minimal, implying that $\bb = \bc = \bb'$. 
\end{proof}

\begin{proof}[Proof of~\Cref{prop:minimal_pl_paths}]
Let $g \in \PL_0(V)$. We claim that there is a unique \emph{minimal} representative, namely a representative $\ba = (v_1, ..., v_{n})$ such that each $v_i \neq 0$ and each consecutive pair $v_{i}, v_{i+1}$ are linearly independent. Note first that such a representative must exist: given an arbitrary representative of $g$, the process of applying the rewriting steps must terminate, since each step shortens the word. But rewriting can only terminate at a minimal word. To prove uniqueness, suppose that $\ba = (v_{1}, ..., v_{n})$ and $\bb = (u_{1}, ..., u_{m})$ are two minimal representatives. Since $\ba$ and $\bb$ are both words representing $g$, there is a sequence of rewriting steps, and their inverses, going from $\ba$ to $\bb$. By Lemma \ref{lem:path_strongrewrite}, there is a word $\bc$ such that $\ba\xrightarrow{*}\bc$ and $\bb\xrightarrow{*}\bc$. But because $\ba$ and $\bb$ are minimal, it is not possible to apply any non-trivial rewriting step. Hence $\ba = \bc = \bb$. 
\end{proof}
\begin{remark}
    Note that the minimal representative of an element $g \in \PL_0(V)$ is also the unique representative with minimal length. This is because a shortest-length word must be minimal.  
\end{remark}
\subsection{Piecewise Linear Loops} \label{apxsec:pl_loops}
The goal of this section is to prove~\Cref{thm:pl_loop_iso}. Let $\tilde{\eta}_{V}: \Pair(V) \to \Loop(V)$ be the natural map sending a generator $(v,u)$ to the corresponding element in the group. This map satisfies the conditions of ~\Cref{prop:univ_property_loops}. In fact, the universal property of this proposition holds for $\tilde{\eta}_{V}: \Pair(V) \to \Loop(V)$ because of the defining relations. We will need to make use of this property in what follows.

The first step in proving ~\Cref{thm:pl_loop_iso} is to show that $\Loop(V)$ has minimal representatives, analogous to~\Cref{prop:minimal_pl_paths}. 

\begin{lemma} \label{lem:loop_minimal}
    An element $\bx \in \Loop(V)$ has a unique minimal representative
    \begin{align}
        \bx = \Big( (u_1, v_1), \ldots, (u_n, v_n)\Big)_{\min}.
    \end{align}
    This is a word such that the following conditions are satisfied
    \begin{enumerate}
        \item for each $i$, the vectors $u_i$ and $v_i$ are linearly independent, and
        \item for each $i < n$, either $v_i \neq u_{i+1}$ or $v_i - u_i$ and $v_{i+1} - u_{i+1}$ are linearly independent. 
    \end{enumerate}
\end{lemma}

Similar to~\eqref{eq:pl_paths_rewriting}, we define two rewriting steps in $\FMon(V^2)$. Let $\ba, \bb \in \FMon(V^2)$, let $(v_1, v_2), (v_2, v_3) \in V^2$ be such that $v_1, v_2, v_3$ lie on a common affine line, and let $(u_1, u_2) \in V^2$ be such that $u_1, u_2$ are linearly dependent. The two rewriting steps are
\begin{align}
    \ba \concat (v_1, v_2) \concat (v_2, v_3) \concat \bb \xrightarrow{\ref{L1}} \ba \concat (v_1, v_3) \concat \bb \andd \ba \concat (u_1, u_2) \concat \bb \xrightarrow{\ref{L2}} \ba \concat \bb.
\end{align}
Note that the minimal words defined in ~\Cref{lem:loop_minimal} are precisely the words which cannot be shortened using these rewriting steps. 
We follow the notation from~\Cref{apxsec:minimal_pl_paths} and use $\ba \rightarrow \bb$ and $\ba \xrightarrow{*} \bb$ to denote a single and arbitrary (possibly zero) rewriting steps, respectively. The following lemma is proved in the same manner as~\Cref{lem:paths_rewritestep} by checking all cases, so we omit the details.

\begin{lemma} \label{lem:loop_rewritestep}
    Let $\ba \in \FMon(V^2)$ and suppose $\ba \to \bb$ and $\ba \to \bb'$ are two rewriting steps. Then, there is a word $\bc$ such that $\bb \xrightarrow{*} \bc$ and $\bb' \xrightarrow{*} \bc$. 
\end{lemma}

\begin{proof}[Proof of~\Cref{lem:loop_minimal}.]
    Using~\Cref{lem:loop_rewritestep} instead of ~\Cref{lem:paths_rewritestep}, we prove the analogue of ~\Cref{lem:path_strongrewrite}. Using this result, we then repeat the argument from the proof of ~\Cref{prop:minimal_pl_paths}.
\end{proof}

Now, we can use this to prove~\Cref{thm:pl_loop_iso}.

\begin{proof}[Proof of~\Cref{thm:pl_loop_iso}.]
    As noted below the definition of $\eta_{V}$ from ~\eqref{eq:eta_V_loops}, this map verifies the assumptions of ~\Cref{prop:univ_property_loops}. Using the universal property of the group $\Loop(V)$, there is a unique map $F: \Loop(V) \to \PL_0^{\cl}(V)$ such that $F \circ \tilde{\eta}_{V} = \eta_{V}$. It sends the element $(v,u)$ to the triangular loop $\eta_V(v,u) = (v, u-v, -u)$. This map is surjective since we can factor any element of $\PL_0^{\cl}(V)$ as a product of triangular loops. It remains to show injectivity. Suppose $\bx = \big((u_1, v_1), \ldots, (u_n, v_n)\big)_{\min}$ is the minimal representative of an element of $\Loop(V)$. Under $F$, this gets sent to the equivalence class for 
    \begin{align}
        F(\bx) = \big( u_1, v_1 - u_1, -v_1, \ldots, u_n, v_n - u_n, v_n\big).
    \end{align}
    Let $l(\cdot)$ denote the length of the minimal representative for both $\Loop(V)$ and $\PL_0(V)$. While there may be cancellations for $F(\bx)$ in $\PL_0^{\cl}$, we must have $l(\bx)\leq l(F(\bx))$ because $\bx$ is minimal in $\Loop(V)$. Indeed, the only reductions in $F(\bx)$ can occur between $-v_{i}$ and $u_{i+1}$. Suppose these vectors are linearly dependent. If $v_{i} \neq u_{i+1}$, then the word reduces to $(..., u_{i+1} - v_{i}, ...)$ and there are no further reductions at this locus. If, on the other hand, $v_{i} = u_{i+1}$, then the word reduces to $(..., v_{i} - u_{i}, v_{i+1} - u_{i+1}, ...)$ and there are no further reductions at this locus. Hence, the minimal representative for $F(\bx)$ will contain all $v_{i} - u_{i}$.

    If $\bx \in \ker(F)$, we must have $l(\bx) \leq l(F(\bx)) = 0$. This implies that $l(\bx) = 0$ and thus $\bx = \emptyset$. Hence, $F$ is injective.

\end{proof}

\subsection{Kites and Planar Loops}

In this section, we collect several useful results about kites and planar loops.

\begin{lemma} \label{lem:tail_of_planar_kite_is_planar}
    Let $\bb \in \planarloop(V)$ be a non-trivial planar loop with span $U \subset V$, let $\bx \in \PL_{0}(V)$ be a path such that $\bx \concat \bb \concat \bx^{-1} \in \planarloop(V)$ is planar, then $\bx \in \PL_0(U)$.
\end{lemma}
\begin{proof}
    Let $\bb = (u_1, ..., u_p)$ and $\bx = (w_1, ..., w_{s})$ be minimal representatives. We will prove this by induction on the length $s$. Consider 
    \begin{align}
    \bx \concat \bb \concat \bx^{-1} = (w_1, ..., w_s, u_1, ..., u_p, -w_s, ..., -w_1).
    \end{align}
    The only reductions can occur between $w_{s}$ and $u_1$, or between $u_p$ and $-w_s$. Hence, if $w_s \not \in U$, then the word is minimal. But then the span of $\bx \concat \bb \concat \bx^{-1}$ contains $U$ and $w_{s}$, contradicting the assumption that $\bx \concat \bb \concat \bx^{-1}$ is planar. Hence, $w_s \in U$. Let $\by = (w_1, ..., w_{s-1})$, which is minimal, and let $\bc = w_{s} \concat \bb \concat (-w_s)$, which is a non-trivial planar loop with span $U$. Then $\by \concat \bc \concat \by^{-1} = \bx \concat \bb \concat \bx^{-1}$ is planar. By induction $\by \in \PL_0(U)$, and therefore, $\bx = \by \concat w_s \in \PL_0(U)$.
\end{proof}

\begin{lemma} \label{lem:product_of_two_planar_loops}
    Let $\bb_1, \bb_2 \in \planarloop(V)$ be non-trivial planar loops with spans $U_1, U_{2} \subseteq V$. If $\bb_1 \concat \bb_2 \in \planarloop(V)$ is planar, then $U_1 = U_2$. 
\end{lemma}
\begin{proof}
    Let $\bb_{1} = (u_1, ..., u_{s})$ and $\bb_{2} = (v_{1}, ..., v_{r})$ be minimal representatives. Then 
    \begin{align}
    \bb_{1} \concat \bb_2 = (u_{1}, ..., u_{s}, v_{1}, ..., v_{r}). 
    \end{align}
    If $v_1 \not \in U_1$, then this word is minimal. But then the span contains $U_1$ and $v_{1}$, contradicting the assumption that the loop is planar. Hence $v_1 \in U_1 \cap U_2$. For the same reason, $u_s \in U_1 \cap U_2$. If $v_1$ and $u_s$ are linearly independent, then $U_1 = U_2$ and we're done. Hence, we assume that they are colinear, allowing us to reduce further: 
    \begin{align}
    \bb_{1} \concat \bb_2 = (u_{1}, ..., u_{s-1}, u_{s}+ v_{1}, v_2, ..., v_{r})
    \end{align}
    Now if $v_2 \not \in U_1$, then the word is minimal (possibly after removing $u_s + v_1$ if this vector is $0$). 
    Because both $\bb_1$ and $\bb_2$ are non-trivial loops, we have $s,r \geq 3$, and both $(u_1, u_2)$ and $(v_1, v_2)$ are linearly independent pairs of vectors.
    Hence, the span of $\bb_1 \concat \bb_2$ contains $U_1$ and $v_2$, again contradicting planarity of $\bb_1 \concat \bb_2$. Thus, $v_1, v_2 \in U_1$, implying that $U_1 = U_2$. 
\end{proof}

\begin{lemma} \label{lem:planar_loops_to_kites}
    Let $\ba, \bb \in \planarloop(V)$ be non-trivial planar loops with spans $U_a, U_b \subseteq V$. If $\ba \concat \bb$ is a kite $\bv \concat \bc \concat \bv^{-1}$, where $\bv \in \PL_0(V)$ and $\bc \in \planarloop(V)$, then $U_a = U_b$. 
\end{lemma}
\begin{proof}
    Suppose $\ba \concat \bb = \bv \concat \bc \concat \bv^{-1}$, where $\bc \in \planarloop(V)$ with $\SPAN(\bc) = U_c$. 
    If $\bc$ is trivial, then $\bb = \ba^{-1}$ and we are done. Hence, we assume $\bc$ to be non-trivial.
    Let
    \begin{align}
        \ba = (a_1, \ldots, a_s)_{\min} \andd \bb = (b_1, \ldots, b_r)_{\min},
    \end{align}
    where $s,r \geq 3$ since they are both nontrivial loops. 
    Consider a factorization of the minimal representative of the concatenation
    \begin{align} \label{eq:kite_comp_word1}
        \ba \concat \bb = (a_1, \ldots, a_{s'}, b'_l, \ldots, b_{r})_{\min} = \ba' \concat \bb',
    \end{align}
    where $\ba' \in \PL_0(U_a)$ and $\bb' \in \PL_0(U_b)$. To obtain the minimal representative, the only possible simplification is by applying \ref{PL0.1} to $(a_s, b_1)$. If $(a_s, b_1) \sim \emptyset$, then we may apply \ref{PL0.1} again to $(a_{s-1}, b_2)$. If both steps are possible, this implies that $a_{s-1}, a_s \in U_b$, and thus $U_a=U_b$, completing the proof. Therefore, we assume that at most two rewriting steps are possible (i.e. one instance of \ref{PL0.1} and one instance of \ref{PL0.2}), and hence $s' \geq 2$ and $l \leq 2$. \medskip
    
    Next, let $\bv = (v_1, \ldots, v_n)_{\min}$ and $\bc = (c_1, \ldots, c_m)_{\min}$ be minimal representatives. We claim that it is possible to assume that $v_{n} \not \in U_c$. Indeed, let $v_{k}$ be the first element from the right which is not in $U_c$, so that $\bu = (v_{k+1}, \ldots, v_n) \in \PL_0(U_c)$. Define $\bv' = (v_1, ..., v_{k})$, which is minimal, and let $\bc' = \bu \concat \bc \concat \bu^{-1} \in \PL_{0}^{\cl}(U_{c})$, which is a non-trivial planar loop with span $U_c$. Then $\bv = \bv' \concat \bu$ and hence
    \begin{align}
        \bv \concat \bc \concat \bv^{-1} = \bv' \concat (\bu \concat \bc \concat \bu^{-1}) \concat (\bv')^{-1} = \bv' \concat \bc' \concat (\bv')^{-1}.
    \end{align}
    Thus, since $\bv' \concat \bc' \concat (\bv')^{-1}$ is also a kite, we may simply replace $\bv$ with $\bv'$ and $\bc$ with $\bc'$, establishing the claim. The minimal representative for the kite is therefore 
    \begin{align}
        (v_{1}, ..., v_{n}, c_{1}, ..., c_{m}, -v_{n}, ..., -v_{1})_{\min}.
    \end{align} 
   So far, we have obtained two descriptions of the minimal word representative of the product $\ba \concat \bb$. Hence, we have equality of minimal words $\ba' \concat \bb' = \bv \concat \bc \concat \bv^{-1}$. There are two cases to consider, depending on where the transition between $\ba'$ and $\bb'$ occurs relative to $\bc$. 
    \begin{enumerate}
        \item First, the transition occurs in $\bc$. Then there are inclusions of words $\bv \subset \ba'$ and $\bv^{-1} \subset \bb'$. This implies that $\bv \in \PL_{0}(U_{a} \cap U_{b})$. In this case, we have $\bc = (\bv^{-1} \ba \bv) \concat (\bv^{-1} \bb \bv)$, and by ~\Cref{lem:product_of_two_planar_loops}, this implies that $U_a = U_b$.
        \item Otherwise, the transition occurs in $\bv$, implying the inclusion of words $\bc \concat \bv^{-1} \subset \bb'$, or it occurs in $\bv^{-1}$, implying the inclusion of words $\bv \concat \bc \subset \ba'$. In either case, we have $\bv, \bc \in \PL_{0}(U_i)$ for a common $i = a$ or $b$. But then $\ba \concat \bb = \bv \concat \bc \concat \bv^{-1}$ is planar, and so by ~\Cref{lem:product_of_two_planar_loops}, $U_a = U_b$.
    \end{enumerate}

\end{proof}

\begin{corollary}
\label{cor:composition_of_kites}
    Let $\bb_1, \bb_2 \in \planarloop(V)$ be non-trivial planar loops with spans $U_1, U_2 \subseteq V$. Let $\bu \in \PL_0(V)$ be a path. If $\bb_1 \concat \bu \concat \bb_2 \concat \bu^{-1} \in \planarloop(V)$ is a planar loop, then $U_1 = U_2 = U$ and $\bu \in \PL_0(U)$.
\end{corollary}
\begin{proof}
    Let $\bb_3 = \bb_1 \concat \bu \concat \bb_2 \concat \bu^{-1}$, which we assume to be a planar loop. If this loop is trivial, then $\bu \concat \bb_2 \concat \bu^{-1} = \bb_1^{-1}$. Then, since $\bu \concat \bb_2 \concat \bu^{-1}$ is planar, $\bu \in \PL_0(U_2)$ by~\Cref{lem:tail_of_planar_kite_is_planar} and hence $U_1 = U_2$. 
    Next, assuming that $\bb_3$ is non-trivial, let $U_3$ be its span. Then $\bb_1^{-1}, \bb_3$ are non-trivial planar loops with the property that $\bb_1^{-1} \concat \bb_3 = \bu \concat \bb_2 \concat \bu^{-1}$ is a kite. Then $U_1 = U_3$ by~\Cref{lem:planar_loops_to_kites} and $\bu \concat \bb_2 \concat \bu^{-1}$ is a planar loop. By ~\Cref{lem:tail_of_planar_kite_is_planar}, $\bu \in \PL_0(U_2)$. Hence $U_1 = U_2$. 
\end{proof}

\section{Free Crossed Modules of Lie Algebras} \label{apx:xlie}

In this section, we fix a field $\K$ of characteristic $0$, either $\K = \R, \C$. 

\subsection{Free Lie Algebras and Representations}
In this section, we begin with a brief discussion of the free Lie algebra generated by a vector space. Let $\Vect$ be the category of vector spaces over $\K$, and $\Lie$ be the category of Lie algebras over $\K$. There exists a natural forgetful functor
\begin{align}
    \For: \Lie \to \Vect,
\end{align}
and here we describe the corresopnding left adjoint functor
\begin{align}
     \FL: \Vect \to \Lie
\end{align}
which sends a vector space $V$ to the \emph{free Lie algebra over $V$}, denoted $\FL(V)$. Let $T(V)$ be the tensor algebra, the free associative algebra over $V$, equipped with the shuffle coproduct $\Delta: T(V) \to T(V) \otimes T(V)$, which is the algebra map defined on $V$ by $\Delta(v) = v \otimes 1 + 1 \otimes v$. We define $\FL(V)$ to be the set of \emph{primitive elements} of $T(V)$,
\begin{align}
    \FL(V) \coloneqq \Prim(T(V)) = \{ w \in T(V) \, : \, \Delta(w) = w \otimes 1 + 1 \otimes w\}.
\end{align}
Let $B$ be a basis of $V$, and let $\FL(B)$ be the free Lie algebra generated by the \emph{set} $B$. By~\cite[Theorem IV.4.2]{serre_lie_2009}, there is an isomorphism $U(\FL(B)) \cong T(V)$, where $U(\cdot)$ is the universal enveloping algebra. Then, by~\cite[Theorem III.5.4]{serre_lie_2009}, $\Prim(U(\FL(B))) \cong \FL(B)$, allowing us to identify $\FL(V) \cong \FL(B)$. Hence, our notion of free Lie algebra coincides with the usual definition in terms of sets.

\subsection{Free Crossed Modules of Lie Algebras}

Now, we move on to consider free crossed modules of Lie algebras. Consider the subcategory $\XLie(\fg)$ of crossed modules
\begin{align}
    (\delta: \fh \to \fg, \gt),
\end{align}
where $\fg$ is fixed and where morphisms are the identity on $\fg$. Taking the underlying vector space of $\fg$, we define the slice category $\slice{\Vect}{\fg}$ consisting of linear map $s: V \to \fg$. There is a forgetful functor
\begin{align} \label{eq:forgetful_Xfg_vectfg}
    \For: \XLie(\fg) \to \slice{\Vect}{\fg},
\end{align}
and in this section, we will construct a free functor
\begin{align}
    \Fr: \slice{\Vect}{\fg} \to \XLie(\fg)
\end{align}
as a left adjoint to $\For$. We will break up this construction into two steps by factoring the forgetful functor as
\begin{align}
    \XLie(\fg) \to \slice{\Rep(\fg)}{\fg} \to \slice{\Vect}{\fg},
\end{align}
where $\Rep(\fg)$ is the category of $\fg$-representations, and $\slice{\Rep(\fg)}{\fg}$ is the slice category, consisting of $\fg$-equivariant maps $s: V \to \fg$, where the target is the adjoint representation. \medskip

\subsubsection{Step 1}
Recall that $\fg$-representations can be equivalently described as $U(\fg)$-modules, $\Rep(\fg) \cong \Mod(U(\fg))$, for $U(\fg)$ the universal enveloping algebra of $\fg$. Hence, the forgetful functor $\For: \Rep(\fg) \to \Vect$ can be viewed as restricting the $U(\fg)$-action along the map $\K \to U(\fg)$. Its left adjoint is therefore the extension of scalars functor
\begin{align}
    \sE: \Vect \to \Rep(\fg), \quad V \mapsto U(\fg) \otimes_{\K} V.
\end{align}
We can slice this functor to obtain a left adjoint to $\For: \slice{\Rep(\fg)}{\fg} \to \slice{\Vect}{\fg}$, which we also denote by $E: \slice{\Vect}{\fg} \to \slice{\Rep(\fg)}{\fg}$. It sends an object $s: V \to \fg$ to the $\fg$-equivariant morphism
\begin{align}
    U(\fg) \otimes_\K V \xrightarrow{\id \otimes s} U(\fg) \otimes_\K \fg \to \fg,
\end{align}
where the second map is induced by the adjoint action. \medskip

\subsubsection{Step 2}
Next, we will construct the left adjoint to the forgetful functor $\For: \XLie(\fg) \to \slice{\Rep(\fg)}{\fg}$. Let $(f:W \to \fg) \in \slice{\Rep(\fg)}{\fg}$ be an object, where $W$ is a $\fg$-representation, and $f$ is an equivariant map. In order to define a crossed module, we must construct a Lie algebra structure based on $W$, such that the Peiffer identity holds. In fact, we will use the Peiffer identity to define the Lie bracket. 
We define the \emph{Peiffer pairing} $\langle \cdot, \cdot \rangle: W \times W \to W$ by
\begin{align}
    \langle u, v\rangle = f(u) \gt v + f(v) \gt u.
\end{align}
This pairing is symmetric and bilinear, and satisfies $x \gt \langle u, v\rangle = \langle x \gt u, v\rangle + \langle u, x \gt v\rangle$. In other words, it defines a $\fg$-equivariant map $\Pf: S^2(W) \to W$. We denote the image of this map by $\Pf(W)$, which is a subrepresentation of $W$ lying in the kernel of $f$. We define
\begin{align}
    \sQ(W) \coloneqq W / \Pf(W).
\end{align}

Next, we define a preliminary bracket
\begin{align}
    [\cdot,\cdot] : W \times W \to \sQ(W) \quad \text{by} \quad [u,v] = f(u) \gt v,
\end{align}
which is bilinear and skew-symmetric. Furthermore, we have
\begin{align}
    [u, \langle v, w\rangle] = f(u) \gt \langle v, w\rangle = \langle f(u) \gt v, w\rangle + \langle v, f(u) \gt w\rangle \in \Pf(W), 
\end{align}
implying that the bracket descends to
\begin{align} \label{eq:QV_bracket}
    [\cdot, \cdot] : \Lambda^2 \sQ(W) \to \sQ(W). 
\end{align}
This leads to the desired crossed module of Lie algebras.

\begin{lemma}
    The bracket $[\cdot, \cdot]$ defined in~\eqref{eq:QV_bracket} defines a Lie algebra structure on $\sQ(V)$, $f$ descends to an equivariant morphism of Lie algebras $\delta: \sQ(V) \to \fg$, and the $\fg$-action on $\sQ(V)$ is a derivation of $[\cdot, \cdot]$. In particular,
    $(\delta: \sQ(V) \to \fg, \gt)$ is a crossed module of Lie algebras. 
\end{lemma}
\begin{proof}
    First, since $\Pf(W)$ is a subrepresentation of $W$ contained in the kernel of $f$, it is clear that we have an induced equivariant map $\delta: \sQ(V) \to \mathfrak{g}$. Then using the equivariance, we get 
    \begin{align}
    \delta([u,v]) = f( f(u) \gt v) = f(u) \gt f(v) = [f(u), f(v)] = [ \delta(u), \delta(v)], 
    \end{align}
    showing that $\delta$ is bracket preserving. To show that $\mathfrak{g}$ acts as a derivation of the bracket,
    \begin{align}
        x \gt [u,v] &= x \gt (f(u) \gt v) \\
        &= f(u) \gt (x \gt v) + [x, f(u)] \gt v \\
        &= f( x \gt u ) \gt v + [u, x \gt v] \\ 
        &= [ x \gt u, v] + [u, x \gt v]. 
    \end{align}
    To show that $[ \cdot,\cdot ]$ defines a Lie bracket, we must show that it satisfies the Jacobi identity. This is equivalent to showing that $[u, - ]$ is a derivation of the bracket. But $[u, - ] = f(u) \gt $, which we have just shown is a derivation. 
        Finally, to show that $\sQ(W)$ is a crossed module, we must check the Peiffer identity. But this holds by construction. 
\end{proof}

The construction of $\sQ(W)$ described above is functorial, providing the functor
\begin{align}
    \sQ: \slice{\Rep(\fg)}{\fg} \to \XLie(\fg).
\end{align}

\begin{lemma}
    The functor $Q$ is left adjoint to $\For: \XLie(\mathfrak{g}) \to \slice{\Rep(\fg)}{\fg}$. 
\end{lemma}
\begin{proof}
    To exhibit the adjunction we describe the counit $\epsilon : \sQ \circ \For \to \id_{\XLie(\mathfrak{g})}$ and unit $\eta : \id_{\slice{\Rep(\fg)}{\fg}} \to \For \circ \sQ$ natural transformations. On the one hand, we have that $Q \circ\For = id_{\mathrm{XMod}(\mathfrak{g})}$, because $P(\mathfrak{h}) = 0$ when $\mathfrak{h} \to \mathfrak{g}$ arises from a crossed module. Therefore, we take $\epsilon$ to be the identity. 
    On the other hand, the unit natural transformation $\eta$ must have components 
    \begin{align}
    \eta_{W} : W \to W/\Pf(W)
    \end{align}
    which we take to be the natural projection maps. The counit-unit equations follow easily from the fact that $Q(\eta_{W}) = id_{Q(W)}$ and that $\eta_{\mathfrak{h}} = \id_{\mathfrak{h}}$, when $\mathfrak{h}$ arises from a crossed module. 
\end{proof}

\subsubsection{Universal Property}

Now, we can put these two functors together to obtain a left adjoint to the forgetful functor in~\eqref{eq:forgetful_Xfg_vectfg}.

\begin{theorem} \label{thm:free_cmlie_local}
    The forgetful functor $\For: \XLie(\fg) \to \slice{\Vect}{\fg}$ has a left adjoint
    \begin{align}
        \Fr = \sQ \circ \sE : \slice{\Vect}{\fg} \to \XLie(\fg),
    \end{align}
    which sends a linear map $s : V \to \fg$ to the crossed module
    \begin{align} \label{eq:free_cmlie_def}
        \Fr(s) \coloneqq \Big(\delta: (U(\fg) \otimes_\K V)/\Pf(U(\fg) \otimes_\K V) \to \fg, \gt\Big).
    \end{align}
\end{theorem}

This free crossed module satisfies the following universal property. 

\begin{corollary} \label{cor:univ_property_cmlie_fixed}
    The unit of the adjunction provides a distinguished map
\begin{align} \label{eq:unit_adjunction_cmlie}
    \eta_s: V \to \sQ(U(\fg) \otimes V) \in \slice{\Vect}{\fg}
\end{align}
given by including the vector $v \in V$ into the equivalence class of $1 \otimes v$.
For any crossed module $(\delta' : \fh \to \fg, \gt)$ with a map $f: V \to \fh$ in $\slice{\Vect}{\fg}$, there exists a unique map $F: \Fr(V) \to \fh$ in $\XLie(\fg)$ such that
\begin{align}
    F \circ \eta_s = f.
\end{align}
\end{corollary}

\begin{remark}
    This construction of the free crossed module is equivalent to other definitions, for instance in~\cite[Section 3.6]{cirio_categorifying_2012}, which first construct a free pre-crossed module (by taking the free Lie algebra of $\sE(V)$), and then quotienting out by the Peiffer identity (as we have done with $\sQ$). Our construction uses the fact that the Peiffer identity fully determines the Lie bracket, and bypasses the need to go through the free pre-crossed module. 
\end{remark}

\subsection{Global Universal Property}

In this this section, we will show that the free crossed module satisfies a stronger universal property within the entire category $\XLie$ of crossed modules. \medskip

\subsubsection{Pullback Property}
We will start by describing a pullback construction that leads to a factorization of maps. Suppose we have the following map of crossed modules
\begin{align}
    f = (f_1, f_0): \cmh = (\delta: \fh_1 \to \fh_0, \gt_\fh) \to \cmg = (\delta: \fg_1 \to \fg_0, \gt_\fg).
\end{align}
Our aim is to define the pullback crossed module $f_0^*\cmg$ such that this map factors as
\begin{align}
    f: \cmh \xrightarrow{u(f)} f_0^* \cmg \xrightarrow{\tf} \cmg.
\end{align}

First, consider the composition
\begin{align}
    \fh_0 \xrightarrow{f_0} \fg_0 \xrightarrow{\gt_\fg} \Der(\fg_1),
\end{align}
which allows us to define the semi-direct product $\fh_0 \ltimes \fg_1$, equipped with a Lie algebra map
\begin{align}
    \pi: \fh_0 \ltimes \fg_1 \to \fg_0, \quad \text{defined by} \quad (X, a) \mapsto f_0(X) + \delta(a).
\end{align}
Then, we define the pullback by
\begin{align}
    f_0^*\fg_1 \coloneqq \ker(\pi) = \left\{ (X,a) \, : \, f_0(X) = - \delta(a) \right\}.
\end{align}
By direct computation, one can check that the Lie bracket on $f_0^*\fg_1$ is given by
\begin{align}
    [(X,a), (Y,b)] = \left( [X, Y], -[a,b]\right), 
\end{align}
which implies that there are natural Lie algebra maps
\begin{align}
    \delta: f_0^* \fg_1 \to \fh_0, \quad (X,a) \mapsto X \andd \tf_1: f_0^* \fg_1 \to \fh_1, \quad (X,a) \mapsto -a 
\end{align}
which satisfy $\delta \circ \tf_1 = f_0 \circ \delta$. Furthermore, one can check that the action
\begin{align}
    X \gt (Y, b) \coloneqq \left( [X, Y], f_0(X) \gt b\right)
\end{align}
satisfies the properties of a crossed module. 

\begin{lemma} \label{lem:pullback_cmlie}
    The pullback defined by $f_0^*\cmg = (\delta: f_0^*\fg_1 \to \fh_0, \gt)$ is a crossed module of Lie algebras. Furthermore, $\tf = (\tf_1, f_0) : f_0^*\cmg \to \cmg$ is a morphism of crossed modules. 
\end{lemma}

Then, the desired factorization is immediate.

\begin{lemma} \label{lem:factorization_cmlie}
    Let
    \begin{align}
        f = (f_1, f_0): \cmh = (\delta: \fh_1 \to \fh_0, \gt_\fh) \to \cmg = (\delta: \fg_1 \to \fg_0, \gt_\fg)
    \end{align}
    be a morphism of crossed modules. There exists a unique map $u(f) = (u_1(f), \id_{\fh_0}) : \cmh \to f_0^* \cmg$ in $\XLie(\fh_0)$ defined by
    \begin{align}
        u_1(f) : \fh_1 \to f_0^* \fg_1, \quad a \mapsto (\delta(a), -f_1(a)),
    \end{align}
    such that $f = \tf \circ u(f)$, where $\tf$ is defined in~\Cref{lem:pullback_cmlie}.
\end{lemma}

\subsubsection{Global Universal Property} 

In this section, we will consider a more general universal property for free crossed modules. Let $\VL$ denote the comma category associated to the functors $\id: \Vect \to \Vect$ and $\For: \Lie \to \Vect$. An object of $\VL$ is given by the data of a vector space $V$, a Lie algebra $\fg$, and a linear map $s : V \to \fg$. A morphism $f = (f_1, f_0): (s : V \to \fg) \to (t: W \to \fh)$ consists of a linear map $f_1: V \to W$ and a Lie algebra morphism $f_0: \fg \to \fh$ such that $t \circ f_1 = f_0 \circ s$ as linear maps. In this section, we will exclusively refer to the natural forgetful functor
\begin{align}
\For : \XLie \to \VL.
\end{align}
Furthermore, given $(s : V \to \fg) \in \VL$, the unit $\eta_s$ will refer to a morphism $(\eta_s, \id_\fg) : (s:V \to \fg) \to \For(\Fr(s))$ in $\VL$.
We are now ready to state and prove the enhanced universal property.

\begin{theorem} \label{apxthm:free_xlie_global}
    Let $(s: V \to \fh) \in \VL$, $\cmg = (\delta: \fg_1 \to \fg_0, \gt) \in \XLie$, and let
    \begin{align}
    f = (f_1, f_0) : (s: V \to \fh) \to \For(\delta: \fg_1 \to \fg_0, \gt).
    \end{align}
    Then there is a unique morphism of crossed modules 
    \begin{align}
    F = (F_1, F_0): \Fr(s) \to \cmg,
    \end{align}
    such that $\For(F) \circ \eta_{s} = f$. In particular $F_1 \circ \eta_s = f_1$ and $F_0 = f_0$.
\end{theorem}
\begin{proof}
    Given the Lie algebra map $f_{0} : \fh \to \fg_0$, we consider the pullback crossed module $f_0^*\cmg$, and the corresponding morphism of crossed modules from~\Cref{lem:pullback_cmlie},
    \begin{align}
    \tf: f_0^* \cmg \to \cmg.
    \end{align}
    Mimicking the construction in~\Cref{lem:factorization_cmlie}, let $u_1(f) : V \to f_{0}^{*}\fg_1$ be defined as $u_1(f)(v) = (s(v), -f_{1}(v)). $ This defines a map  
    \begin{align}
    u(f) = (u_1(f), \id_\fh) : (s: V \to \fh) \to \For(f_0^*\cmg)
    \end{align}
    in $\slice{\Vect}{\fh}$ with the property that $\For(\tf) \circ u(f) = f$. Now using the universal property of $\Fr(s)$ in $\XLie(\fh)$ from~\Cref{cor:univ_property_cmlie_fixed}, there is a unique map 
    \begin{align}
    g: \Fr(s) \to f_0^*\cmg
    \end{align}
    in $\XLie(\fh)$ such that $\For(g) \circ \eta_{s} = u(f)$, where $\eta_s$ is the unit from~\eqref{eq:unit_adjunction_cmlie}. Now, we define 
    \begin{align}
    F = \Fr(s) \xrightarrow{g} f_0^* \cmg \xrightarrow{\tf} \cmg,
    \end{align}
    which is a morphism of crossed modules and it satisfies 
    \begin{align}
    \For(F) \circ \eta_{s} = \For(\tilde{f}) \circ \For(g) \circ \eta_{s} = \For(\tilde{f}) \circ u(f) = f. 
    \end{align}
    Hence, $F$ is the desired morphism. 
    
    It therefore only remains to show that $F$ is unique. To this end, let $G$ be another such morphism. By the identity $\mathrm{For}(G) \circ \eta_{s} = f$, we must have $G_{0} = f_{0}$. Then, by~\Cref{lem:factorization_cmlie}, $G$ factors as $G = \tilde{f} \circ u(G)$, where 
    \begin{align}
    u(G)= (u_1(G), \id_{\fh}) : \Fr(s) \to f_0^*\cmg
    \end{align}
    is in the category $\XLie(\fh)$. Now consider the map $l = \For(u(G)) \circ \eta_{s} : V \to f_{0}^{*}\fg_1$ in $\slice{\Vect}{\fh}$. It satisfies 
    \begin{align}
    \mathrm{For}(\tilde{f}) \circ l = \mathrm{For}(\tilde{f}) \circ \mathrm{For}(u(G)) \circ \eta_{s} = \mathrm{For}(\tilde{f} \circ u(G)) \circ \eta_{s} = \mathrm{For}(G) \circ \eta_{s} = f. 
    \end{align}
    This this implies that $l = u(f)$, or that $\mathrm{For}(u(G)) \circ \eta_{s} = u(f)$. But now appealing to the universal property of $\Fr(s)$, we see that we must have $u(G) = g$, the map defined above. Hence, 
    \begin{align}
    G = \tilde{f} \circ u(G) = \tilde{f} \circ g = F.
    \end{align}
    This establishes uniqueness. 
\end{proof}

We can rephrase the above theorem as the existence of an adjunction. 

\begin{corollary} \label{Fullfree}
    The forgetful functor $\For : \XLie \to \VL$ has a left adjoint 
\begin{align}
\Fr : \VL \to \XLie.
\end{align}
Restricting to the subcategory $\slice{\Vect}{\fg}$ this gives the functor $\sQ \circ \sE$ defined in~\Cref{thm:free_cmlie_local}. Therefore, it is given by sending a linear map $s: V \to \fg$ to the crossed module $\Fr(s)$ as defined in~\eqref{eq:free_cmlie_def}.
\end{corollary}

In~\eqref{eq:free_cmla_vector_space_functor}, we consider a special class of such free crossed modules given by the functor $\cmk: \Vect \to \XLie$. 
In fact, this crossed module has additional structure, since there is a natural action of $\GL(V)$ on $V$. For a group $G$, we define a $G$-Lie algebra to be a Lie algebra $\fh$ which is also a $G$-representation such that the action $\gtd$ of $G$ on $\fh$ preserves the bracket,
\begin{align} \label{eq:g_lie_algebra}
    g \gtd[x,y] = [g \gtd x, g \gtd y].
\end{align}

\begin{corollary} \label{cor:free_cm_equivariant}
    For any $V \in \Vect$, $\cmk(V)$ is a crossed module in the category of $\GL(V)$-Lie algebras. 
\end{corollary}
\begin{proof}
    First, we note that $\FL(V)$ has a natural $\GL(V)$ action by extending the $\GL(V)$ action on $V$ using~\eqref{eq:g_lie_algebra}. The vector space $T(V) \otimes \Lambda^2 V$ also has a natural $\GL(V)$ action via the usual tensor and wedge products of representations. The map $\sE(s_V): T(V) \otimes \Lambda^2 V \to \FL(V)$ is easily seen to be equivariant. Furthermore, the embedding $\FL(V) \to U(\FL(V)) = T(V)$ is equivariant, and thus the $\FL(V)$ action on $T(V) \otimes \Lambda^2 V$ is equivariant,
    \begin{align}
        g \gtd (X \gt A) = (g \gtd X) \gt (g \gtd A),
    \end{align}
    for $g \in \GL(V)$, $X \in \FL(V)$ and $A \in T(V) \otimes \Lambda^2(V)$.
    It is also clear to check that the Peiffer subspace is invariant under the $\GL(V)$ action, $\GL(V) \gtd \Pf(T(V) \otimes \Lambda^2 V) \subset \Pf(T(V) \otimes \Lambda^2 V)$ so that the $\GL(V)$ action descends to $\sQ(T(V) \otimes \Lambda^2 V)$. The equivariance of $\delta$ and the $\FL(V)$ action $\gt$ are also preserved. 
\end{proof}

\section{Compatible Triangulations} \label{apx:triangulations_matching}
In this appendix, we provide details on triangulations, which are used to prove~\Cref{thm:existence_compatible_representative} and~\Cref{prop:value_Hur_vanishes}. Here, we must carefully address the relationship between refinements of triangulated representatives $ r(\bX) = (X_1, \ldots, X_k) \in \FMon(\Kite^{\times}(V))$ and triangulations of their corresponding PLSCs $\Delta(r(\bX))$. First, we show that triangulated representatives exist.

\begin{lemma} \label{lem:triangulated_representative}
    There exists a triangulated representative $r(\bX)  \in \FMon(\Kite^{\times}(V))$ of every $\bX \in \PL_1(V)$.
\end{lemma}
\begin{proof}
    It suffices to show that each marked kite $X = (\bw, \bb)$ has a triangulated representative. Suppose $\bb = (b_1, \ldots, b_k)_{\min} \in \planarloop(V)$ is minimal with point representation $\bb = [\hb_0, \ldots, \hb_m]$ where $\hb_0 = 0$ and $\hb_k = \sum_{i=1}^k b_i$.

    \begin{figure}[!h]
        \includegraphics[width=\linewidth]{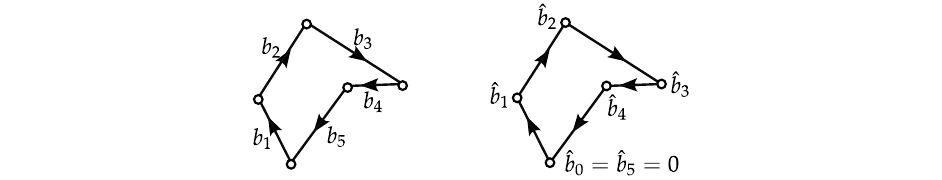}
    \end{figure}
    \noindent
    Then, define
    \begin{align}
        \tbb_r \coloneqq (\hb_r, b_{r+1}, -\hb_{r+1}) \in\planarloop(V) \andd \tX_r \coloneqq (\bw, \tbb_r).
    \end{align}
    Then, we have $\tbX = (\tX_1, \ldots, \tX_{k-2}) \sim_{\ref{PL1.1}} (X)$. The kites $\tX_r$ are either triangular, or else have $\tbb_r = \emptyset_0$. In the later case $\tX_{r} \sim_{\ref{PL1.2}} \emptyset_1$ and hence can be removed from the list. Let $\tbX'$ be the resulting element of $\FMon(\Kite^{\times}(V))$. It is a triangulated representative of the marked kite $X$. 
\end{proof}

Let $r(\bX)$ be a triangulated representative of $\bX \in \PL_{1}(V)$. In general, there is no guarantee that this representative is compatible. Therefore, in order to prove~\Cref{thm:existence_compatible_representative}, we must modify, or refine, $r(\bX)$. We will do this by triangulating the PLSC $\Delta(r(\bX))$.  

Recall that a \emph{convex polygon in $V$} is the convex hull $\mathrm{Conv}(P)$ of a finite set of points $P$ which are all contained in a $2$-dimensional affine plane $U \subset V$. Such a polygon is \emph{non-degenerate} if the points $P$ are not contained in any affine line $L \subset V$. By an edge in $V$, we simply mean a line segment in $V$ connecting a pair of distinct points. Furthermore, given an edge or a polygon $a$, we let $a^{\circ}$ denote its interior. 

\begin{definition} \label{def:triangulation}
    Let $E = \{e^1, \ldots, e^q\}$ be a collection of edges in $V$ and let $P = \{P^1, \ldots, P^m\}$ be a collection of nondegenerate convex polygons. A \emph{compatible triangulation} for $(E, P)$ is a compatible non-degenerate PLSC $T(E,P)$ such that 
    \begin{enumerate}
        \item for each $1$-simplex $\epsilon \in T(E,P)$ and each edge $e^{i} \in E$, if $|\epsilon|^\circ \cap (e^i)^\circ \neq \emptyset$, then $|\epsilon| \subseteq e^i$.
        \item for each 2-simplex $\tau \in T(E,P)$ and each polygon $P^{i} \in P$, if $|\tau|^\circ \cap (P^i)^{\circ} \neq \emptyset$, then $|\tau| \subseteq P^i$. 
        \item $|T(E,P)| = (\bigcup_{i=1}^q e^i ) \ \cup \ (\bigcup_{i=1}^m P^i )$.
        \item the edges in $E$ and in the boundaries of the polygons of $P$ are contained in $|T(E,P)_{1}|$, the realization of the 1-skeleton. 
    \end{enumerate}
\end{definition}

\begin{lemma} \label{lem:induced_triangulation}
    Let $T(E,P)$ be a compatible triangulation for a pair $(E,P)$. Given an edge in $E$ or a polygon in $P$, denoted $Q$, the subset
    \begin{align}
        T(Q) \coloneqq \{ \sigma \in T(E,P) \, : \, |\sigma| \subseteq Q\}
    \end{align}
    defines a compatible triangulation of $Q$. 
\end{lemma}
\begin{proof}
   First, we note that $T(Q)$ is indeed a PLSC, since for any $\sigma \in T(E,P)$ such that $|\sigma| \subset Q$, the faces of $\sigma$ must also be contained in $Q$. The properties of being compatible and non-degenerate are automatically inherited. Next, we note that conditions (1) and (2) of~\Cref{def:triangulation} are properties of the simplices in $T(E,P)$, and thus are inherited by $T(Q) \subset T(E,P)$. For condition (3), note first that $|T(Q)| \subseteq Q$ follows from the construction. If $Q$ is a polygon, then $Q$ is contained in the union of all 2-simplices of $T(E,P)$. And by condition (2), the simplices which are not in $T(Q)$ have their interiors disjoint from $Q^o$ and so cannot contribute any area. As a result $Q \subset |T(Q)|$. A similar argument applies if $Q$ is an edge, using condition (1) instead. Condition (4) follows essentially for topological reasons. 
\end{proof}

\begin{lemma}\label{lem:compatible_triangulation}
    Let $E$ and $P$ be, respectively, a set of edges and nondegenerate convex polygons in $V$. Then there exists a compatible triangulation for $(E, P)$. 
\end{lemma}
\begin{proof}
    The general idea behind this proof is to maximally subdivide all polygons and edges in order to obtain the compatible triangulation.
    The first step is to construct sets of affine planes $H$,  affine lines $L$, and vertices $C_{0}$. 

    \begin{itemize}
        \item[Step 1] The polygons in $P$ are supported on a finite set of affine planes, which we take to be $H$. 
        \item[Step 2] The set of lines $L$ is constructed as follows. The boundary of each polygon $P^i \in P$ is a collection of line segments. Extend these to lines and add them to $L$. Extend each edge in $E$ to a line and add it to $L$. Finally, if two distinct planes $H_{i}, H_{j} \in H$ intersect along a line, add it to $L$. 
        \item[Step 3] To construct the set of vertices $C_{0}$, we start with the endpoints of edges in $E$ and the extremal points of the polygons in $P$. If two lines in $L$ intersect at a point, we add this to $C_{0}$. If a line in $L$ intersects a plane in $H$ at a point, we add this to $C_{0}$. Finally, if two planes in $H$ intersect at a point, we add this to $C_{0}$. 
    \end{itemize}

    Restrict attention to a single plane $H_{i} \in H$. This plane will contain a subset $L_{i} \subseteq L$ of lines and a subset $C_{i} \subseteq C_0$ of vertices. Suppose that there is a point $p \in C_{i}$ which is not contained in any line from $L_{i}$. Then add a new line $l$ to $L$ which is contained in $H_{i}$ and which contains the point $p$. This will lead to new vertices in $C_{0}$ arising from the intersections between $l$ and other lines from $L$. However, there will not be further new vertices arising from intersecting $l$ with planes $H_{j} \neq H_{i}$. For this reason, adding $l$ will not introduce any \emph{new} vertices which are contained within a plane $H_{j} \in H$, but not within a line from $L$ lying in $H_{j}$. Therefore, by iterating this process, we may assume that every vertex from $C_{i}$ is contained in some line from $L_{i}$. 

    Now observe that using the vertices in $C_{0}$, each line $l \in L$ is decomposed into a finite number of closed segments $\epsilon \subset l$. Since each edge $e^{i} \in E$ is a subset of some line in $L$ lying between two vertices from $C_{0}$, it is decomposed into a union of such line segments $\epsilon$. Let $\mathcal{E} = \{ \epsilon^{1}, ..., \epsilon^{p} \}$ denote the collection of all finite line segments whose interiors intersect the interior of one of the edges from $E$. Note that $\cup_{i = 1}^{p} \epsilon^{i} = \cup_{j = 1}^{q} e^{j}$. Indeed, each $\epsilon^{i}$ is contained in some $e^{j}$, whereas each $e^{j}$ decomposes into a union of some $\epsilon^{i}$. 

    Next, consider the plane $H_{i}$ and let $L_{i} \subset L$ denote the subset of lines that lie in $H_{i}$. Let $Z_{i} = H_{i} \setminus \cup_{l \in L_{i}} l$ denote the complement of the set of lines $L_{i}$. It is an open set with the property that the closure of each bounded component is a convex polygon $Q$ in $H_{i}$. Since each $P^{i} \in P$ is the intersection of a collection of half-planes bounded by a subset of the lines $l \in L_{i}$, $P^i$ is decomposed into a union of convex polygons $Q$. Let $\mathcal{P} = \{Q^1, ..., Q^t \}$ denote the collection of all convex polygons in all planes $H_{i}$ whose interiors intersect the interior of one of the polygons from $P$. Note that the extremal points of $Q^{i}$ lie in $C_{0}$, but that there may be additional points from $C_0$ on the boundary. Note also that $\cup_{i = 1}^{t}Q^{i} = \cup_{j = 1}^{m} P^{j}$. Indeed, each $Q^{i}$ is contained in some $P^{j}$, whereas each $P^{j}$ decomposes into a union of some $Q^{i}$. 

   We now construct a PLSC $T(E,P)$ which will be our compatible triangulation. 
   \begin{enumerate}
       \item Let $T_{0} \subset V$ be the collection of all vertices $p \in C_{0}$ which are either contained in an edge $\epsilon^{i} \in \mathcal{E}$, or in a convex polygon $Q^{i} \in \mathcal{P}$. Endow $T_{0}$ with an arbitrary order. This will form the set of vertices of our PLSC. 
       \item Given each convex polygon $Q^{i} \in \mathcal{P}$, choose a triangulation that makes use of precisely the set of vertices in $C_{0}$ lying on its boundary. Define $\Delta_{2}(T)$, the set of $2$-simplices in $T(E,P)$, to be the collection of all triples $[p_0, p_{1}, p_{2}]$ such $|[p_{0}, p_{1}, p_{2}]|$ is one of the triangles from a polygon $Q^{i}$. 
       \item We define $\Delta_{1}(T)$, the set of $1$-simplices in $T(E,P)$, to consist of all faces of all $2$-simplices in $\Delta_{2}(T)$, as well as all pairs $[p_{0}, p_{1}]$ such that $|[p_{0}, p_{1}]|$ is one of the edges in $\mathcal{E}$. 
   \end{enumerate}    
   By construction $T(E,P)$ is a compatible non-degenerate PLSC whose piecewise linear realization satisfies 
   \begin{align}
       |T(E,P)| = (\bigcup_{i=1}^p \epsilon^i ) \ \cup \ (\bigcup_{i=1}^t Q^i ) =(\bigcup_{j=1}^q e^j ) \ \cup \ (\bigcup_{i=1}^m P^j ). 
   \end{align}
   The remaining conditions for a compatible triangulation are then easy to check. 
\end{proof}

The strategy to prove~\Cref{thm:existence_compatible_representative} is to start with a triangulated representative $r(\bX)$ and take its associated PLSC, $\Delta(r(\bX))$, which may not be compatible. Then, we produce a compatible triangulation $T$ and use it to construct a new representative $r'(\bX)$. This new representative will have the property that its associated PLSC is a subcomplex of $T$, thereby assuring that it is compatible. 

The following lemma starts by treating the case of a triangular kite.  

\begin{lemma} \label{lem:triangulation_decomposition_of_kite}
    Let $X = (\emptyset, \bb) \in \Kite^\times(V)$ be a triangular loop, viewed as a marked triangular kite, and let $|\sigma_{X}|$ be its associated 2-simplex, viewed as a convex polygon in $V$. Given a compatible triangulation $T(|\sigma_{X}|)$, there exists a triangulated representative $r(X) = (X_{1}, ..., X_{k})$ such that $r(X) \sim (X)$ in $\PL_1(V)$, and such that $\Delta(r(X)) \subseteq T(|\sigma_{X}|)$.

\end{lemma}
\begin{proof}
    The PLSC $T(|\sigma_{X}|)$ is homeomorphic to $|T(|\sigma_{X}|)| = |\sigma_{X}|$, which is a 2-simplex. Its 1-skeleton $\tilde{Z}$ is a connected $1$-dimensional simplicial complex which contains $Z$, the boundary of $\sigma_{X}$, as a subspace. Let $c = 0$ be the basepoint of $\sigma_{X}$. Because $\tilde{Z}$ is 1-dimensional, its fundamental group $\pi_{1}(\tilde{Z}, c)$ is free. We will need a generating set which is indexed by the 2-simplices of $T(|\sigma_{X}|)$. For this, we choose a spanning tree $T$ of $\tilde{Z}$. For each 2-simplex $\tau = [t_0, t_1, t_2]$, let $\phi_\tau$ be the boundary loop of $\tau$ based at $t_0$, and let $p_{\tau}$ be the unique path in $T$ connecting $c$ to $t_0$. Define 
    \begin{align}
        a_\tau = p_\tau \concat \phi_{\tau} \concat p_{\tau}^{-1} \in \pi_{1}(\tilde{Z}, c),
    \end{align} 
    which form a generating set for $\pi_{1}(\tilde{Z},c)$. The inclusion $\iota: Z \to \tilde{Z}$ induces a homomorphism 
    \begin{align}
    \pi_1(\iota): \pi_{1}(Z, c) \to \pi_1(\tilde{Z},c). 
    \end{align}
    Therefore, given the generator $\gamma \in \pi_{1}(Z,c)$, its image under $\pi_{1}(\iota)$ admits a factorization 
    \begin{align}
        \pi_{1}(\iota)(\gamma) = a_{\tau_{i_1}}^{\pm} \concat ... \concat a_{\tau_{i_k}}^{\pm}. 
    \end{align}
    Now let $W_{0} : \pi_{1}(\tilde{Z}, c) \to \PL_0(V)$ be the homomorphism from ~\eqref{eq:W0_def}. Then $W_{0}(\pi_1(\iota)) = \bb$ and 
    \begin{align}
        W_{0}(a_{\tau}) = \bw_{\tau} \concat \bb_{\tau} \concat \bw_{\tau}^{-1}, 
    \end{align}
    where $\bw_{\tau} = \widetilde{W}_0(p_{\tau}) \in \PL_0(V)$ and where $\bb_{\tau} = \widetilde{W}_{0}(\phi_{\tau}) \in \PL_0^{\cl}(V)$ is a triangular loop. Then, since all paths are contained in the plane spanned by $|\sigma_{X}|$, we obtain a factorization in $\PL_1(V)$,
    \begin{align}
        (X) = ((\bw_{\tau_{i_1}}, \bb_{\tau_{i_{1}}}^{\pm}), ..., (\bw_{\tau_{i_k}}, \bb_{\tau_{i_{k}}}^{\pm})). 
    \end{align}
    Since $p_{\tau}$ is a path in $T$ connecting vertices of $T(|\sigma_{X}|)$, we can choose a lift of $\bw_{\tau}$ to $\FMon(V)$ consisting of vectors in $V$ representing edges of $T(|\sigma_{X}|)$. This gives us the desired triangulated representative $r(X)$. 
\end{proof}

Finally, we treat the general case. %

\begin{proof}[Proof of~\Cref{thm:existence_compatible_representative}]
    Let $\bX \in \PL_1(V)$. By \Cref{lem:triangulated_representative}, there exists a triangulated representative $r(\bX) = (X_1, \ldots, X_n) \in \FMon(\Kite^\times(V))$. Let $\Delta(r(\bX))$ be the associated PLSC, let $E= |\Delta_{1}(r(\bX))|$ denote the set of 1-simplices realized as edges in $V$, and let $P =|\Delta_{2}(r(\bX))|$ denote the set of 2-simplices realized as convex polygons in $V$. Then, by~\Cref{lem:compatible_triangulation}, there exists a compatible triangulation $C \coloneqq T(E, V)$. 
    
    Given each marked triangular kite $X_{i} = (\bw_{i}, \bb_{i})$, the associated 2-simplex $|\sigma_{X_{i}}|$ is a polygon in $P$. By~\Cref{lem:induced_triangulation}, the subcomplex $T(|\sigma_{X_{i}}|) \subseteq C$ is a compatible triangulation of $|\sigma_{X_{i}}|$. Therefore, by~\Cref{lem:triangulation_decomposition_of_kite}, there is a triangulated representative 
    \begin{align}
        r(\bb_{i}) = ( (\bv_{1}^{i}, \bc_{1}^{i}), ..., (\bv_{q_i}^{i}, \bc_{q_i}^{i}))
    \end{align} such that $r(\bb_{i}) \sim (\emptyset, \bb_i)$ in $\PL_{1}(V)$ and such that $\Delta(r(\bb_i)) \subseteq T(|\sigma_{X_{i}}|)$ (after shifting $\Delta(r(\bb_i))$ by the displacement of $\bw_{i}$). Furthermore, $\bw_{i} = (u_{1}, ..., u_{k})$ consists of a collection of vectors in $V$, which give rise to a subset of edges from $E$. Each edge has the form $[\hat{u}_{i}, \hat{u}_{i+1}]$, where $\hat{u}_{i+1} = \hat{u}_{i} + u_{i}$. Because $C$ is a compatible triangulation, each such edge decomposes into a sequence of edges $|\epsilon|$, such that $\epsilon \in C$ is a 1-simplex. As a result, we may replace each $u_{i}$ with a sequence of vectors $(s_{1}^i, ..., s_{r_i}^i)$ which sum to $u_{i}$ and such that the corresponding edges $[\hat{s}_{j}^{i}, \hat{s}_{j+1}^{i}]$ are 1-simplices of $C$. Let $\bw_{i}' \in \FMon(V)$ denote the resulting word and note that $\bw_{i}$ and $\bw_{i}'$ are equivalent in $\PL_0(V)$. Now define 
    \begin{align}
        r(X_{i}) = ( (\bw_{i}' \concat \bv_{1}^{i}, \bc_{1}^{i}), ..., (\bw_{i}' \concat \bv_{q_i}^{i}, \bc_{q_i}^{i})) \in \FMon(\Kite^{\times}(V)). 
    \end{align}
    By construction, $r(X_{i})$ is equivalent to $X_{i}$ in $\PL_{1}(V)$ and $\Delta(r(X_{i}))$ is a subcomplex of $C$. Finally, define $r(\bX)'$ to be the product of the $r(X_{i})$, for $i = 1,..., n$. This is a triangulated representative of $\bX$ and its associated PLSC $\Delta(r(\bX)')$ is a subcomplex of $C$. It is therefore non-degenerate and compatible. Hence $r(\bX)'$ is a compatible representative for $\bX$. 
\end{proof}

Next, we prove a lemma which is required for~\Cref{prop:value_Hur_vanishes}.  

\begin{lemma} \label{lem:disjoint_2_forms}
    Suppose $C$ is a 2-dimensional non-degenerate compatible PLSC whose set of $2$-simplices is denoted by $L$. For each 2-simplex $\sigma \in L$, there exists a compactly supported 2-form $\omega_\sigma \in \Omega^2_c(V)$ such that $\int_\sigma \omega_\sigma = 1$ and such that the support of $\omega_\sigma$ is disjoint from all other 2-simplices,
    \begin{align}
        \supp(\omega_\sigma) \cap |\tau| = \emptyset
    \end{align}
    for all $\tau \in L$ such that $\tau \neq \sigma$. 
\end{lemma}
\begin{proof}
    Let $\sigma \in L$ and let $H \subset V$ denote the 2-plane which supports $\sigma$. Let $K \subset H$ denote a closed ball such that $K \subset |\sigma|^\circ$. Consider a compactly supported 2-form $\nu \in \Omega^2_c(H)$ on $H$ which is supported in $K$ and satisfies $\int_\sigma \nu = 1$. Using a metric on $V$, pullback $\nu$ along the orthogonal projection $\pi: V \to H$ to obtain a form $\pi^*(\nu)$. Then, choose a compactly supported smooth bump function $\rho : V \to \mathbb{R}$ such that $\rho|_{K} \equiv 1$ and which vanishes along each $\tau \in L$, $\tau \neq \sigma$. Then, we define 
    \begin{align}
        \omega_{\sigma} \coloneqq \rho \ \pi^*(\nu) \in \Omega^2_{c}(V).
    \end{align}
    This is the desired $2$-form. 
    
\end{proof}

\bibliographystyle{plain}
\bibliography{thin_homotopy}

\end{document}